\documentclass{article}

\usepackage[tbtags]{amsmath}    
\usepackage{amsfonts,amssymb,amsthm,euscript,makeidx,pifont,boxedminipage,fullpage}
\usepackage{bbm}
\usepackage{array}
\usepackage{enumerate}
\usepackage{braket}
\usepackage{stmaryrd}
\usepackage{subfigure}

\usepackage{fancybox}

\usepackage{ifpdf}
\ifpdf
   \usepackage[pdftex,bookmarks=true,bookmarksnumbered,%
    plainpages=false,hypertexnames=false,pdfpagelabels,%
    colorlinks=true,linkcolor=black,urlcolor=blue,citecolor=red,%
    pdfstartview=FitH]{hyperref}
   \usepackage[pdftex]{graphicx,color}
\else
    \usepackage[dvips,bookmarks=true,bookmarksnumbered,plainpages=false,%
     colorlinks=true,linkcolor=black,urlcolor=blue,citecolor=red,%
     hypertexnames=false,pdfstartview=FitH,breaklinks=true]{hyperref}
    \usepackage{pstricks,pst-node,pst-text,pst-3d}
    \usepackage[dvips]{graphicx}
\fi
\usepackage{color}
\usepackage[active]{srcltx}%
\usepackage{subfigure}
\usepackage{multirow}

\makeatletter

\usepackage{float,MnSymbol}

\usepackage[latin1]{inputenc} 

\usepackage{amsfonts,euscript,makeidx,pifont,boxedminipage,fullpage}
\usepackage{bbm}
\usepackage{array}
\usepackage{enumerate}
\usepackage{braket}
\usepackage{stmaryrd}
\usepackage{subfigure}
\usepackage{centernot}
\usepackage[normalem]{ulem}
\usepackage{soul}

\usepackage{fancybox}

\usepackage{ifpdf}
\ifpdf
   \usepackage[pdftex,bookmarks=true,bookmarksnumbered,%
    plainpages=false,hypertexnames=false,pdfpagelabels,%
    colorlinks=true,linkcolor=black,urlcolor=blue,citecolor=red,%
    pdfstartview=FitH]{hyperref}
   \usepackage[pdftex]{graphicx,color}
\else
    \usepackage[dvips,bookmarks=true,bookmarksnumbered,plainpages=false,%
     colorlinks=true,linkcolor=black,urlcolor=blue,citecolor=red,%
     hypertexnames=false,pdfstartview=FitH,breaklinks=true]{hyperref}
    \usepackage{pstricks,pst-node,pst-text,pst-3d}
    \usepackage[dvips]{graphicx}
\fi
\usepackage{color}
\usepackage[active]{srcltx}%
\usepackage{subfigure}
\usepackage{multirow}

\makeatletter

\newtheorem{theorem}{Theorem}[section]
\newtheorem{lemma}[theorem]{Lemma}

\newtheorem{remark}[theorem]{Remark}

\newcommand{\norm}[1]{\left\| #1 \right\|}

 \newcommand{\lsup}[1]{\underset{#1\to\infty}{\overline{\lim}}}

\newcommand{\Exp}{\mathbb{E}} 

\makeatother
\begin{document}

\title{An Emergent Autonomous Flow for Mean-Field Spin Glasses }

\author{James MacLaurin\footnote{New Jersey Institute of Technology. james.n.maclaurin@njit.edu}}

\maketitle
\begin{abstract}
We study the dynamics of symmetric and asymmetric spin-glass models of size $N$. The analysis is in terms of the double empirical process: this contains both the spins, and the field felt by each spin, at a particular time (without any knowledge of the correlation history). It is demonstrated that in the large $N$ limit, the dynamics of the double empirical process becomes deterministic and autonomous over finite time intervals. This does not contradict the well-known fact that SK spin-glass dynamics is non-Markovian (in the large $N$ limit) because the empirical process has a topology that does not discern correlations in individual spins at different times. In the large $N$ limit, the evolution of the density of the double empirical process approaches a nonlocal autonomous PDE operator $\Phi_t$. Because the emergent dynamics is autonomous, in future work one will be able to apply PDE techniques to analyze bifurcations in $\Phi_t$. Preliminary numerical results for the SK Glauber dynamics suggest that the `glassy dynamical phase transition' occurs when a stable fixed point of the flow operator $\Phi_t$ destabilizes. 
\end{abstract}

\section{Introduction}
This paper studies the emergent dynamics of mean-field non-spherical spin-glasses. At low temperature, spin-glass systems are characterized by slow emergent timescales that typically diverge with the system size (see \cite{Guionnet2007,Stein2013,Jagannath2019} for good surveys of known results). Probably the most famous mean-field spin glass model is that of Sherrington and Kirkpatrick \cite{Sherrington1975}. It is widely known in the physics community that the SK spin glass undergoes a `dynamical phase transition' as the temperature is lowered \cite{Arous2001,Arous2003,Guionnet2007,Stein2013}. Essentially what this means is that the average correlation-in-time of spins does not go to zero as time progresses: that is, some spins get locked into particular states and flip extremely rarely. Although there has been much progress in the study of spin glass dynamics \cite{BenArous1995,Guionnet1997a,BenArous2006}, a rigorous proof of a dynamical phase transition in the original SK spin glass model remains elusive. More precisely, although it is known that the time to equilibrium is $O(1)$ when the temperature $\beta^{-1}$ is high \cite{Bauerschmidt2019,Gheissari2019}, there is lacking a proof that the time-to-equilibrium diverges with $N$ when $\beta$ is large (to the best of this author's knowledge). Furthermore it is well-established that the equilibrium SK Spin-Glass system undergoes a `Replica Symmetry Breaking' phase transition as $\beta$ increases \cite{Guerra2003,Talagrand2006,Panchenko2013}, and this leads many scholars to expect that a phase transition should also be manifest in the initial dynamics. The equilibrium `Replica Symmetry Breaking' transition is characterized by the distribution of the overlap between two independent replica not concentrating at $0$, but possessing a continuous density over an interval away from zero \cite{Mezard1987,Talagrand2011}. A major reason for the lack of a rigorous characterization of the dynamical phase transition (as emphasized by Ben Arous \cite{Arous2003} and Guionnet \cite{Guionnet2007}) is that the existing large $N$ emergent equations are not autonomous and very difficult to analyze rigorously. This paper takes steps towards this goal by deriving an autonomous PDE for the emergent (large $N$) dynamics: this PDE should be more amenable to a bifurcation analysis (to be performed in future work) than the existing nonautonomous delay equations \cite{BenArous1995,Grunwald1996}. These results are also of great relevance to the dynamics of asymmetric spin glass models, which have seen a resurgence of interest in neuroscience in recent years \cite{Faugeras2015a,Fasoli2015,Fasoli2019,Kadmon2015,Doiron2016,Dembo2019,Faugeras2019a}.  

This paper determines the emergent dynamics of $M$ `replica' spin glass systems started at initial conditions that are independent of the connections. `Replicas' means that we take identical copies of the same static connection topology $\mathbf{J}$, and conditionally on $\mathbf{J}$, run independent and identically-distributed jump-Markov stochastic processes on each replica.  As noted above, Replicas are known to shed a lot of insight into the rich tree-like structure of `pure states' that emerge in the static SK spin glass at low temperature \cite{Mezard1987,Guerra2003,Talagrand2006,Talagrand2011,Panchenko2013}, and it is thus reasonable to conjecture that replicas will shed much insight into the dynamical phase transition. Indeed Ben Arous and Jagannath \cite{Arous2018} use the overlap of two replicas to determine bounds on the spectral gap determining the rate of convergence to equilibrium of mean-field spin glasses. Writing $\mathcal{E} = \lbrace -1,1 \rbrace$, the spins flip between $-1$ and $1$ at rate $c(\sigma^{i,j}_t,G_t^{i,j})$ for some general function $c: \lbrace -1,1\rbrace \times \mathbb{R} \to \mathbb{R}^+$, where the field felt by the spin is written as
\begin{equation}\label{eq: field definition}
G^{i,j}_t = N^{-\frac{1}{2}}\sum_{k=1}^N J^{jk}\sigma^{i,k}_{t},
\end{equation}
and $\mathbf{J} = \lbrace J^{jk} \rbrace_{1\leq j \leq k \leq N}$ are i.i.d. centered Gaussian variables with a specified level of symmetry. For Glauber dynamics for the SK spin glass \cite{Mezard1987}, the connections are symmetric (i.e. $J^{jk} = J^{kj}$) and the dynamics is reversible, with $c$ taking the form \cite{Grunwald1998},
\begin{equation}\label{eq: C Glauber}
c(\sigma,g) = \big(1+ \exp\big\lbrace 2\beta\sigma (g+h)\big\rbrace\big)^{-1},
\end{equation}
where $h$ is a constant known as the magnetization, and $\beta^{-1}$ is the temperature. In this case, the spin-glass dynamics are reversible with respect to the following Gibbs Measure
\begin{equation}\label{eq: equilibrium}
\mu^N_{\beta,\mathbf{J}}(  \boldsymbol\sigma  ) =\exp\bigg(\frac{\beta}{2}\sum_{p=1}^M\sum_{j,k=1}^N J^{jk}\sigma^{p,j}\sigma^{p,k} + h\sum_{p=1}^M \sum_{j=1}^N\tilde{\sigma}^{p,j}- NM\rho^N_{\mathbf{J}} \bigg),
\end{equation} 
where $\rho^N_{\mathbf{J}}$ is a normalizing factor, often called the free energy, given by
\begin{equation}\label{eq: rho N J}
\rho^N_{\mathbf{J}} = N^{-1}\log \sum_{\boldsymbol\sigma \in \mathcal{E}^N} \big[ \exp\big(\frac{\beta}{2}\sum_{j,k=1}^N J^{jk}\sigma^{j}\sigma^{k} + h\sum_{j=1}^N\sigma^j \big)\big].
\end{equation}
For further details on the equilibrium Gibbs measure, see the reviews in \cite{Bolthausen2007,Talagrand2011,Panchenko2013}. It is known that as $\beta$ increases from $0$, a sharp transition occurs, where the convergence to equilibrium bifurcates from being $O(1)$ in time, to timescales that diverge in $N$ \cite{Arous2003,Jagannath2019}.

One of the novelties of this paper is to study the emergent properties of the \textit{double empirical process} $(\hat{\mu}^N_t(\boldsymbol\sigma,\mathbf{G}))_{t\geq 0}$, which contains information on the distribution of the spins and fields, without knowledge of the `history' of each spin and field. Formally, $\hat{\mu}^N(\boldsymbol\sigma,\mathbf{G})$ is a c\`adl\`ag $\mathcal{P}$-valued process (where $\mathcal{P}=\mathcal{M}^+_1( \mathcal{E}^M \times \mathbb{R}^M)$), i.e.
\begin{align}
\hat{\mu}^N(\boldsymbol\sigma,\mathbf{G}): &\mathcal{D}\big([0,\infty),\mathcal{E}\big)^{MN} \times \mathcal{D}\big([0,\infty),\mathbb{R}\big)^{MN} \to \mathcal{D}\big([0,\infty),\mathcal{P} \big), \label{eq: emp process 0}\\
\hat{\mu}^N(\boldsymbol\sigma,\mathbf{G}) :=& \big\lbrace \hat{\mu}^N(\boldsymbol\sigma_t,\mathbf{G}_t) \big\rbrace_{t\in [0,\infty)}\text{ where }\\
\hat{\mu}^N(\boldsymbol\sigma_t,\mathbf{G}_t) =& N^{-1}\sum_{j\in I_N} \delta_{(\sigma^{1,j}_t,,\ldots,\sigma_{t}^{M,j}),(G^{1,j}_t,\ldots,G^{M,j}_t)}, \label{eq: emp process}
\end{align}
where $\lbrace \sigma^{i,j}_t \rbrace$ is the solution of  the jump Markov Process, and the fields are defined in \eqref{eq: field definition}.

We now overview some of the existing literature on the dynamics of the SK spin glass. In the physics literature, averaging over quenched disorder been used to derive limiting equations for the correlation functions \cite{Sompolinsky1981,Sompolinsky1981a,Houghton1983,Sommers1987,Mezard1987,Cugliandolo1994,Laughton1996,Laughton1996a}. The first rigorous mathematical results were obtained in the seminal work of Ben Arous and Guionnet \cite{BenArous1995,BenArous1998} (these results were for a similar `soft-spin' model driven by Brownian Motions). Guionnet \cite{Guionnet1997a} expanded on this work to prove that in the soft SK spin glass started at i.i.d initial conditions, the dynamics of the empirical measure converges to a unique limit, with no restriction on time or temperature. Grunwald \cite{Grunwald1996,Grunwald1998} obtained analogous equations for the limiting dynamics of the pathwise empirical measure for the jump-Markov system studied in this paper. More recent work by Ben Arous, Dembo and Guionnet has rigorously established the Cugliandolo-Kurchan \cite{Cugliandolo1994} / Crisanti-Horner-Sommers \cite{Crisanti1993} equations for spherical spin glasses using Gaussian concentration inequalities \cite{BenArous2006}. A recent preprint of Dembo, Lubetzky and Zeitouni has established universality for asymmetric spin glass dynamics, extending the work of Ben Arous and Guionnet to non-Gaussian connections, with no restriction on time or temperature \cite{Dembo2019}.

In the papers cited above, the emergent large $N$ dynamics is non-autonomous: that is, one needs to know the full history of the emergent variable (either the empirical measure, or correlation / response functions) upto time $t$ to predict the dynamics upto time $t+\delta t$. In the early work of Ben Arous, Guionnet and Grunwald \cite{BenArous1995,Guionnet1997a,BenArous1998,Grunwald1996}, the emergent variable is the pathwise empirical measure. This is an extremely rich object because it `knows' about average correlations in individual spins at different times. Ben Arous and Guionnet \cite{BenArous1995} demonstrated that the limiting dynamics of the pathwise empirical measure is the law of a complicated implicit delayed stochastic differential equation. In the later work of Ben Arous, Dembo and Guionnet on spherical spin glasses, a simpler set of emergent variables was used: the correlation and response functions \cite{Arous2001,BenArous2006} (this formalism is frequently used by physicists \cite{Houghton1983,Mezard1987,Cugliandolo1994}). In the $p=2$ case, the resultant equations are autonomous, and this allowed them to rigorously prove that there is a dynamical phase transition \cite{Arous2001}.

There is still lacking a rigorous characterization of the dynamical phase transition in the non-spherical SK model. As has been emphasized by Ben Arous \cite{Arous2003} and Guionnet \cite{Guionnet2007}, a fundamental difficulty is that all of the known emergent equations are non-autonomous (that is, they are either delay integro-differential equations, or an implicit delayed SDE \cite{BenArous1995,Grunwald1996}). A major reason that the emergent equations are not autonomous is that the emergent object studied by \cite{BenArous1995,Grunwald1996} - the pathwise empirical measure - carries too much information, because it knows about the \textit{history} of the spin-flipping. This is why  this paper focuses on determining the limiting dynamics of a different order parameter: the double empirical process (as defined in \eqref{eq: emp process 0}-\eqref{eq: emp process}) that cannot discern time-correlations in individual spins. The empirical process carries more information about the system than that of Ben Arous, Guionnet \cite{BenArous1995,Guionnet1997a,BenArous1998} and Grunwald \cite{Grunwald1996} insofar as it contains information about overlaps between different replicas, but less information insofar as it does not know about correlations-in-time of individual spins.
 The chief advantage of working with this order parameter is that the dynamics becomes autonomous in the large $N$ limit, just as in classical methods for studying the empirical process in interacting particle systems \cite{Dawson1989,Sznitman1989}. One can now apply the apparatus of PDEs to the limiting equations to study the bifurcation of the fixed points. Indeed preliminary analytic work has identified that there is a bifurcation in the fixed point of the flow \eqref{eq: p plus evolution} for SK Glauber dynamics, and 2 replica (see Remark \ref{Remark Bifurcation}). 

Many recent applications of dynamical spin glass theory have been in neuroscience, being referred to as networks of balanced excitation and inhibition. Typically the connections in these networks are almost asymmetric, unlike in the original SK model. These applications include networks driven by white noise \cite{Cessac1994,Brunel2003,Wainrib2013,Cabana2013,Faugeras2015a,Fasoli2015,Fasoli2019,Faugeras2019} and also deterministic disordered networks \cite{Aljadeff2015,Kadmon2015,Doiron2016,Crisanti2018}\footnote{One should be able to adapt the methods of this paper to this setting.}; the common element to all of these papers being the random connectivity of mean zero and high variance.   It has been argued that the highly variable connectivity in the brain is a vital component to the emergent gamma rhythm \cite{Brunel2003}. Another important application of spin-glass theory has been the study of stochastic gradient descent algorithms \cite{Baity-Jesi2018,Montanari2019}.

Our fundamental result is to show that as $N\to\infty$, the empirical process converges to have a density given by a Mckean-Vlasov-type PDE \footnote{See \cite{Dawson1989,Sznitman1989} for further discussion of such PDEs.} of the form, for $\boldsymbol\alpha \in \mathcal{E}^M$ and $\mathbf{x} \in \mathbb{R}^M$,
\begin{equation}
\frac{\partial p_t}{\partial t}(\boldsymbol\alpha,\mathbf{x}) = \sum_{i=1}^M\bigg\lbrace c(-\alpha^i,x^i) p_t(\boldsymbol\alpha[i],\mathbf{x})  - c(\alpha^i,x^i)p_t(\boldsymbol\alpha ,\mathbf{x} )+ 2L^{\xi_t}_{ii}   \frac{\partial^2 p_t }{\partial x_i^2 }(\boldsymbol\alpha,\mathbf{x})\bigg\rbrace
- \nabla\cdot \big\lbrace \mathbf{m}^{\xi_t}(\boldsymbol\alpha,\mathbf{x}) p_t(\boldsymbol\alpha,\mathbf{x})\big\rbrace, \label{eq: p plus evolution start}
\end{equation}
where $\xi_t \in \mathcal{M}^+_1\big(\mathcal{E}^{M} \times \mathbb{R}^M\big)$ is the probability measure with density $p_t$, $\boldsymbol\alpha[i]$ is the same as $\boldsymbol\alpha$, except that the $i^{th}$ spin has a flipped sign. $\mathbf{m}^{\xi_t}$ and $\mathbf{L}^{\xi_t}$ are functions defined in Section 2. 

In broad outline, our method of proof resembles that of Ben Arous and Guionnet \cite{BenArous1995} and Grunwald \cite{Grunwald1996}, insofar as (i) we freeze the interaction and (ii) study the Gaussian properties of the field variables. However our approach is different insofar as, after freezing the interaction, we do not use Girsanov's Theorem to study a tilted system, but instead study the pathwise evolution of the empirical process over small time increments. This \textit{pathwise} approach to the Large Deviations of interacting particle systems has been popular in recent years: being employed in the work of Budhiraja, Dupuis and colleagues \cite{Budhiraja2012}, in this author's work on  interacting particle systems with a sparse random topology \cite{maclaurin2016large}, and subsequent work in \cite{Coghi2018,Faugeras2019,Coppini2019}.  More precisely, we study the evolution over small time intervals of the expectation of test functions with respect to the double empirical measure: a method that has been applied to interacting particle systems in, for example, \cite{Jourdain1998} and \cite{Lucon2014}. To understand the change in the fields $\lbrace G^j_t\rbrace$ over a small increment in time, we use the law $\gamma$ of the connections, conditioned on the value of the fields at that time step. It is fundamental to our proof that - essentially due to the Woodbury formula for the inverse of a matrix with a finite-rank perturbation - the conditional Gaussian density can be written as a function of the empirical measure $\hat{\mu}^N_t(\boldsymbol\sigma, \mathbf{G}) = N^{-1}\sum_{j\in I_N} \delta_{(\boldsymbol\sigma^j_t , \mathbf{G}^j_t)}$ and the local spin and field variables (see the analysis in Section \ref{Section Gaussian}). 

\textit{Notation}: Let $\mathcal{E} = \lbrace -1 ,1 \rbrace$. For any Polish Space $\mathcal{X}$, we let $\mathcal{M}^+_1(\mathcal{X})$ denote all probability measures on $\mathcal{X}$, and $\mathcal{D}\big( [0,T], \mathcal{X} \big)$ the Skorohod space of all $\mathcal{X}$-valued c\`adl\`ag functions \cite{Billingsley1999}. We always endow $\mathcal{M}^+_1(\mathcal{X})$ with the topology of weak convergence.  Let $\mathcal{P} := \mathcal{M}^+_1\big(\mathcal{E}^M \times \mathbb{R}^M\big)$ denote the set of all probability measures on $\mathcal{E}^M \times \mathbb{R}^M$, and define the subset
\begin{equation}
\tilde{\mathcal{P}} = \big\lbrace \mu \in \mathcal{P} \; : \mathbb{E}^{\mu}\big[\norm{x}^2\big] < \infty \big\rbrace .
\end{equation}
For any vector $\mathbf{g} \in \mathbb{R}^M$, $\norm{\mathbf{g}}$ is the Euclidean norm, and $\norm{\mathbf{g}}_{\infty}$ is the supremum norm. For any square matrix $\mathbf{K} \in \mathbb{R}^{m\times m}$, $\norm{\mathbf{K}}$ is the operator norm, i.e.
 \[
 \norm{\mathbf{K}} = \sup_{\mathbf{x} \in \mathbb{R}^m: \norm{\mathbf{x}} = 1}\big\lbrace\norm{\mathbf{K}\mathbf{x}}\big\rbrace.
 \]
Let $d_W$ be the Wasserstein Metric \cite{Sznitman1989,Gibbs2002} on $\tilde{\mathcal{P}}$, i.e.
\begin{align}\label{eq: Wasserstein}
d_W(\beta , \zeta ) = \inf_{\eta}\big\lbrace \mathbb{E}^\eta\big[ \norm{\mathbf{x}-\mathbf{g}}+\norm{\boldsymbol\alpha-\boldsymbol\sigma}  \big]\big\rbrace .
\end{align}
where the infimum is over all measures $\eta \in \mathcal{M}^+_1\big( \mathcal{E}^M\times \mathbb{R}^M \times\mathcal{E}^M\times\mathbb{R}^M\big)$ with marginals $\beta$ (over the first two variables), and $\zeta$ (over the second two variables). We let $\mathcal{C}([0,T],\mathcal{X})$ denote the space of all continuous functions from $[0,T]$ to $\mathcal{X}$. $\mathcal{B}(\mathcal{X})$ denotes the Borelian subsets. 

The spins are indexed by $I_N := \lbrace 1,2,\ldots, N\rbrace$, and the replicas by $I_M := \lbrace 1,2,\ldots,M\rbrace$.  The typical indexing convention that we follow is $\boldsymbol\sigma^j_t = (\sigma^{1,j}_t,\ldots,\sigma^{M,j}_t)^T \in \mathcal{E}^M$, and $\boldsymbol\sigma_t = (\sigma^{i,j}_t)_{i \in I_M, j\in I_N} \in \mathcal{E}^{NM}$.
\section{Outline of model and main result}

Let $\big(\Omega ,{\mathcal {F}}, (\mathcal{F}_t) , \mathbb{P}\big)$ be a filtered probability space containing the following random variables. The connections $( J^{jk})_{j,k \in\mathbb{Z}}$ are centered Gaussian random variables, with joint law $\gamma \in \mathcal{M}^+_1\big(\mathbb{R}^{\mathbb{Z}^+\times\mathbb{Z}^+}\big)$. To lighten the notation we assume that there are self-connections (one could easily extend the results of this paper to the case where there are no self-connections). Their covariance is  taken to be of the form
\begin{equation}\label{eq: covariance gamma}
\Exp^{\gamma}\big[J^{jk} J^{lm} \big] =  \delta(j-l)\delta(k-m) + \mathfrak{s}\delta(j-m)\delta(k-l) .
\end{equation}
The parameter $\mathfrak{s} \in [0,1]$ is a constant indicating the level of symmetry in the connections. In the case that $\mathfrak{s} = 1$, $J^{jk} = J^{kj}$ identically, and in the case that $\mathfrak{s}=0$, $J^{jk}$ is probabilistically independent of $J^{kj}$. (One could easily extend these results to the case that $\mathfrak{s} \in [-1,0)$). $\lbrace J^{jk} \rbrace_{j,k \in \mathbb{Z}^+}$ are assumed to be $\mathcal{F}_0$-measurable.

We take $M$ replicas of the spins: this means that the connections $\mathbf{J}$ are the same across the different systems, but (conditionally on $\mathbf{J}$) the spin-jumps in different systems are independent. Our reason for working with replicas is that, as discussed in the introduction, in the case of reversible dynamics, replicas are known to shed much light on the rich `tree-like' structure of pure states in the equilibrium Gibbs measure \cite{Parisi1979,Mezard1987,Guerra2003,Talagrand2006,Panchenko2013}. If one wishes to avoid replicas, one could just take $M=1$. The spins $ \big\lbrace \sigma^{i,j}_{t} \big\rbrace_{j\in I_N , i \in I_M, t\geq 0 }$ constitute a system of jump Markov processes: $i$ being the replica index, and $j$ being the spin index. Spin $(i,j)$ flips between states in $\mathcal{E} = \lbrace -1 , 1 \rbrace$ with intensity $c(\sigma^{i,j}_t,G^{i,j}_t)$ (where $G^{i,j}_t = N^{-\frac{1}{2}}\sum_{k=1}^N J^{jk}\sigma^{i,k}_t$) for a function $c: \mathcal{E}\times \mathbb{R} \to [0,\infty)$ for which we make the following assumptions:
\begin{itemize} 
\item $c$ is strictly positive and uniformly bounded, i.e. for some constant $c_1 > 0$,
\begin{equation}\label{eq: uniform bound c sigma g}
\sup_{\sigma \in \mathcal{E}}\sup_{g\in \mathbb{R}}\big| \big(c(\sigma,g)  \big| \leq c_1 \text{ and }c(\sigma,g) > 0.
\end{equation}
\item The following Lipschitz condition is assumed: for a constant $c_L > 0$, for all $\sigma \in \mathcal{E}$ and $g_1,g_2 \in \mathbb{R}$,
\begin{align}
\big| c\big(\sigma,g_1\big)-  c\big(\sigma,g_2\big) \big| &\leq c_L \big| g_1 - g_2 \big| \label{eq: c Lipschitz 1}\\
\big| \log c\big(\sigma,g_1\big)- \log c\big(\sigma,g_2\big) \big| &\leq c_L \big| g_1 - g_2 \big| .\label{eq: c Lipschitz 2}
\end{align}
\item The following limits exist for $\sigma = \pm 1$,
\begin{equation}\label{eq: c limit}
\lim_{g\to-\infty} c(\sigma,g) \; \; , \; \; \lim_{g\to \infty} c(\sigma,g).
\end{equation}
\item The log of $c$ is bounded in the following way: there exists a constant $C_g > 0$ such that
\begin{equation}
\sup_{\alpha \in \mathcal{E}}\big| \log c(\alpha,g) \big| \leq C_g \big| g \big|. \label{eq: bound log c}
\end{equation}
\end{itemize}
We note that the Glauber Dynamics for the reversible dynamics in \eqref{eq: C Glauber} satisfy the above assumptions \cite{Glauber1963,Grunwald1998}.

To facilitate the proofs, we represent the stochasticity as a time-rescaled system of Poisson counting processes of unit intensity \cite{Ethier1986}. We thus define $\lbrace Y^{i,j}(t) \rbrace_{i\in I_M , j \in \mathbb{Z}^+}$ to be independent Poisson processes, which are also independent of the disorder variables $\lbrace J^{jk} \rbrace_{j,k \in \mathbb{Z}^+}$. We define the spin system $\lbrace \sigma^{i,j}_t \rbrace$ to be the unique solution of the following system of SDEs
\begin{equation}\label{eq: Yij definition}
\sigma^{i,j}_t = \sigma^{i,j}_0 \times A\cdot Y^{i,j}\bigg(\int_0^t c(\sigma^{i,j}_s , G^{i,j}_s)ds \bigg),
\end{equation}
where $A\cdot x := (-1)^x$. Clearly $\sigma^{i,j}_t$ depends on $N$ (for convenience this is omitted from the notation). The law of the initial condition $\boldsymbol\sigma_0$ is written as $\mu_{0} \in \mathcal{M}^+_1\big(\mathcal{E}^{MN}\big)$. $\mu_0$ is assumed to be independent of the disorder. Note that the forward Komolgorov equation describing the dynamics of the law $P^N_{\mathbf{J}}(t) \in \mathcal{M}^+_1\big(\mathcal{E}^{MN}\big)$ of the spins at time $t$ (conditioned on a realization $\mathbf{J}$ of the disorder) is \cite{Ethier1986}
\begin{align}
\frac{dP^N_{\mathbf{J}}(\boldsymbol\sigma)}{dt} = \sum_{i\in I_M,j\in I_N}\big\lbrace c(-\sigma^{i,j}_t, \hat{G}^{i,j}_t)P^N_{\mathbf{J}}(\boldsymbol\sigma[i,j]) - c(\sigma^{i,j}_t, G^{i,j}_t)P^N_{\mathbf{J}}(\boldsymbol\sigma)\big\rbrace, \label{eq: ODE for P J sigma}
\end{align}
where $\boldsymbol\sigma[i,j] \in \mathcal{E}^{MN}$ is the same as $\boldsymbol\sigma$, except that the spin with indices $(i,j)$ has a flipped sign, and $\hat{G}^{i,j}_t = N^{-1/2}\sum_{k\in I_N, k\neq j}J^{jk}\sigma^{i,k}_t - N^{-1/2}J^{jj}\sigma^{i,j}_t$.

For some fixed constant $\mathfrak{c} > 0$, define the set
\begin{equation}
\mathcal{X}^N = \big\lbrace \boldsymbol\eta \in \mathcal{E}^{NM} \; : \inf_{\mathfrak{b} \in \mathbb{R}^M \; : \norm{\mathfrak{b}}=1} \sum_{p,q\in I_M , j\in I_N} \eta^{p,j}\eta^{q,j}\mathfrak{b}^p \mathfrak{b}^q > N\mathfrak{c}\big\rbrace .\label{eq: mathcal X N assumption}
\end{equation}
We assume that the initial condition is such that
\begin{align}\label{eq: det K condition}
\lsup{N}N^{-1}\log \mathbb{P}\big( \boldsymbol\sigma_{0}\notin \mathcal{X}^N \big) < 0 .
\end{align}
Note that \eqref{eq: det K condition} is satisfied if $\lbrace \boldsymbol\sigma^j_0 \rbrace_{j\in I_N}$ are iid samples from some probability law $\tilde{\mu}_0 \in \mathcal{M}^+_1(\mathcal{E}^M)$ that is such that  
\[
 \inf_{\mathfrak{b} \in \mathbb{R}^M \; : \norm{\mathfrak{b}}=1} \mathbb{E}^{\tilde{\mu}_0}\big[\langle \mathfrak{b}, \boldsymbol\sigma \rangle^2\big] > \mathfrak{c}.
\]
One would then find that \eqref{eq: det K condition} follows from Sanov's Theorem \cite{Dembo1998}. For an arbitrary positive constant $T>0$, we define
\begin{equation}\label{eq: tilde tau N stopping}
\tau_N =T\wedge \inf\big\lbrace t: t\in [0,  T] \text{ and }\boldsymbol\sigma_t \notin \mathcal{X}^N\big\rbrace .
\end{equation}
If $\tau_N < T$, then the smallest eigenvalue of the overlap matrix $\mathbf{K}^{\hat{\mu}^N_{\tau_N}}$ (as defined in \eqref{eq: K xi defn}) is $\mathfrak{c}$. Intuitively, the stopping time is reached when the spins in different replicas are too similar. One expects that this is an extremely rare event, even on timescales diverging in $N$. See Remark \ref{Remark Overlap}. The main result of this paper is the following. We emphasize that these are `quenched' results. `Annealing' methods are not used in this paper.

\begin{theorem}\label{Theorem 1}
Fix $T > 0$. There exists a flow operator $\Phi : \mathcal{P} \to \mathcal{C}\big([0,T],\mathcal{P} \big)$ written $\Phi \cdot \mu := \lbrace\Phi_t\cdot \mu \rbrace_{t\geq 0}$ such that $\Phi_0\cdot \mu = \mu$ and for any $\epsilon > 0$
\begin{equation}\label{Theorem 1 Result}
\lsup{N}N^{-1}\log \mathbb{P}\big( \sup_{t \leq \tau_N} d_W\big( \Phi_t\cdot \hat{\mu}^N(\boldsymbol\sigma_0,\mathbf{G}_0)  , \hat{\mu}^N(\boldsymbol\sigma_t,\mathbf{G}_t) \big) \geq \epsilon \big) < 0. 
\end{equation}
The flow $\Phi$ is specified in Section \ref{Section Phi}. It follows immediately from the Borel-Cantelli Theorem that $\mathbb{P}$ almost surely
\begin{equation}
\lim_{N\to \infty} \sup_{t \leq \tau_N} d_W\big( \Phi_t\cdot \hat{\mu}^N(\boldsymbol\sigma_0,\mathbf{G}_0)  , \hat{\mu}^N(\boldsymbol\sigma_t,\mathbf{G}_t) \big)=0.
\end{equation}
\end{theorem}
\subsection{Existence and Uniqueness of the Flow $\Phi_t$}\label{Section Phi}
In this section we define $\Phi \cdot \mu \in \mathcal{C}\big([0,T],\mathcal{P} \big)$, for any $\mu \in \mathcal{P}$ such that $\Exp^{\mu(\sigma,g)}\big[g^2\big] <\infty $. We write $\Phi\cdot\mu := \lbrace \Phi_t\cdot\mu \rbrace_{t\in [0,T]}$, where $\Phi_t: \mathcal{P} \to \mathcal{P}$, and in the following we write $\xi_t =\Phi_t \cdot \mu$. 
\begin{lemma}\label{Lemma xi t}
Fix $T > 0$. For any $\mu \in \mathcal{P} := \mathcal{M}^+_1\big(\mathcal{E}^M \times \mathbb{R}^M\big)$ such that $\Exp^{\mu(\sigma,g)}\big[g^2\big] <\infty $, there exists a unique set of measures $\lbrace \xi_{t}\rbrace_{t\in [0,T]} \subset \mathcal{P}$ with the following characteristics
\begin{enumerate}
\item For all $t  \in (0, T]$, $\xi_t$ has a density in its second variable, i.e. $d\xi_t(\boldsymbol\sigma,\mathbf{x}) = p_t(\boldsymbol\sigma,\mathbf{x})d\mathbf{x}$. $p_t(\boldsymbol\sigma,\mathbf{x})$ is continuously differentiable in $t$, twice continuously differentiable in $\mathbf{x}$, and satisfies the system of equations \eqref{eq: K xi defn}- \eqref{eq: p plus evolution}.
\item $\xi_0 = \mu$, and for all $t\in [0,T]$, $t \to \xi_t$ is continuous.
\end{enumerate}
\end{lemma}
For any $\xi \in \mathcal{P}$ such that $\Exp^{\xi(\sigma,g)}\big[\norm{g}^2\big] <\infty $, define the $M\times M$ coefficient matrices $\lbrace \mathbf{L}^{\xi},\boldsymbol\kappa^{\xi},\boldsymbol\upsilon^{\xi},\mathbf{K}^{\xi}\rbrace\in\mathbb{R}^{M^2}$ to have the following elements,
\begin{align}
K^{\xi}_{jk} =& \Exp^{\xi(\boldsymbol\sigma,\mathbf{x})}\big[ \sigma^j \sigma^k \big] \label{eq: K xi defn}\\
L_{jk}^{\xi} =& \Exp^{\xi(\boldsymbol\sigma,\mathbf{x})}\big[ \sigma^{k}\sigma^j c(\sigma^j,x^j)\big]\label{eq: L xi defn} \\ 
\kappa_{jk}^{\xi} =& \Exp^{\xi(\boldsymbol\sigma,\mathbf{x})}\big[ x^k \sigma^j c(\sigma^j,x^j)\big]\label{eq: kappa xi defn} \\ 
\upsilon^{\xi}_{jk} =&\Exp^{\xi(\boldsymbol\sigma,\mathbf{x})}\big[ \sigma^k x^j \big] .\label{eq: upsilon xi defn} 
\end{align}
For any $\mu \in \mathcal{P}$, define $\Lambda^{\mu}$ to be the smallest eigenvalue of $\mathbf{K}^{\mu}$, i.e.
\begin{align}\label{eq: Lambda mu definition}
\Lambda^{\mu} = \inf_{\mathfrak{a} \in \mathbb{R}^M:\norm{\mathfrak{a}}=1}\sum_{j,k=1}^M K^{\mu}_{jk}\mathfrak{a}^j \mathfrak{a}^k = \inf_{\mathfrak{a} \in \mathbb{R}^M:\norm{\mathfrak{a}}=1} \Exp^{\mu}\big[\big(\sum_{j=1}^M \mathfrak{a}^j \sigma^j\big)^2\big] ,
\end{align}
noting that the eigenvalues of $\mathbf{K}^{\mu}$ are real (since it is symmetric) and non-negative. To facilitate the following proofs (in particular, the existence and uniqueness of the solution to the PDE), we want the following functions $\mathbf{m}^{\xi}(\boldsymbol\sigma,\mathbf{x})$ and $\mathbf{L}^{\xi}$ to be uniformly Lipschitz for all $\xi \in \mathcal{P}$. Indeed thanks to our definition of the stopping time $\tau_N$, it does not matter how $\mathbf{m}^{\xi}$ is defined for $\xi$ such that $\Lambda^{\xi} < \mathfrak{c}/2$, as long as $\epsilon$ is sufficiently small. To this end, we choose a definition that ensures that $\xi \to \mathbf{H}^{\xi}$ is uniformly Lipschitz, i.e.
\begin{equation}\label{eq: H xi definition}
\mathbf{H}^{\xi} = 
\begin{cases}
 \big(\mathbf{K}^{\xi}\big)^{-1} \text{ if }\Lambda^{\xi} \geq \mathfrak{c} / 2\\
 \big( \mathbf{I}(\mathfrak{c}/2 - \Lambda^{\xi})  + \mathbf{K}^{\xi} \big)^{-1} \text{ otherwise. }
\end{cases}
\end{equation}
Now define the vector field $\mathbf{m}^{\xi}(\boldsymbol\sigma,\mathbf{x}) : \mathcal{P}\times \mathcal{E}^M \times \mathbb{R}^M \to \mathbb{R}^M$ as follows
\begin{equation}
	\mathbf{m}^{\xi}(\boldsymbol\sigma,\mathbf{x}) = -2 \mathbf{L}^{\xi}\mathbf{H}^{\xi}\mathbf{x} - 2\mathfrak{s}\boldsymbol\kappa^{\xi}\mathbf{H}^{\xi} \boldsymbol\sigma +2\mathfrak{s} \mathbf{L}^{\xi}\mathbf{H}^{\xi}\boldsymbol\upsilon^{\xi}\mathbf{H}^{\xi}\boldsymbol\sigma .\label{eq: m xi}
\end{equation}

We can now write down the PDE that defines the density of $\xi_t := \Phi_t(\mu)$. For some $\boldsymbol\alpha \in \mathcal{E}^M$ and $\mathbf{x}\in \mathbb{R}^M$, we write $p_t(\boldsymbol\alpha,\mathbf{x})$ to be the density of $\xi_t$ in its second variable, i.e. $\xi_t\big(\boldsymbol\sigma=\boldsymbol\alpha, g^i \in [x^i,x^i+dx^i]\big) := p_t(\boldsymbol\alpha,\mathbf{x}) dx^1 \ldots dx^M$. Write $\boldsymbol\alpha[i]\in\mathcal{E}^M$ to be almost identical to $\boldsymbol\alpha$, except that the $i^{th}$ spin has a flipped sign.  
The evolution of the densities is governed by the following system of partial differential equations
\begin{equation}
\frac{\partial p_t}{\partial t}(\boldsymbol\alpha,\mathbf{x}) = \sum_{i\in I_M}\big\lbrace c(-\alpha^i,x^i) p_t(\boldsymbol\alpha[i],\mathbf{x})  - c(\alpha^i,x^i)p_t(\boldsymbol\alpha ,\mathbf{x} )+2 L^{\xi_t}_{ii}   \frac{\partial^2 p_t }{\partial (x^i)^2 }(\boldsymbol\alpha,\mathbf{x})\big\rbrace
- \nabla\cdot \big\lbrace \mathbf{m}^{\xi_t}(\boldsymbol\alpha,\mathbf{x}) p_t(\boldsymbol\alpha,\mathbf{x})\big\rbrace. \label{eq: p plus evolution}
\end{equation}

\begin{remark}
We emphasize that the convergence result in Theorem \ref{Theorem 1} does not hold for the path-wise empirical measure, i.e.
\[
\tilde{\mu}^N = N^{-1}\sum_{j\in I_N} \delta_{(\sigma^j_{[0,T]}, G^j_{[0,T]})}\in \mathcal{M}^+_1\big( \mathcal{D}([0,T],\mathcal{E}^M\times\mathbb{R}^M)\big),
\]
endowed with the Skorohod topology on the set of c\`adl\`ag paths $\mathcal{D}\big([0,T],\mathcal{E}^{M}\big)$ \cite{Ethier1986}. Indeed it is known that the limit of the pathwise empirical measure is non-Markovian, so the Markovian stochastic hybrid system with Fokker-Planck equation given by \eqref{eq: p plus evolution} is almost certainly not the limiting law for the pathwise empirical measure \cite{BenArous1995}.  This does not mean that our result in Theorem \ref{Theorem 1} is inconsistent with the non-Markovian results in the work of Ben Arous, Guionnet and Grunwald \cite{BenArous1995,Grunwald1996}, since the topology in our theorem cannot discern correlations in particular spins at different times. 
\end{remark}
\begin{remark}\label{Remark Overlap}
 It seems plausible that for any temperature $\beta> 0$ and any $T > 0$, there exists $\mathfrak{c}$ such that
\[
\lsup{N}N^{-1}\log \mathbb{P}\big(\tau_N <T \big) < 0.
\]
Perhaps one could prove this by demonstrating that the attracting manifold of the flow $\Phi_t$ is such that all eigenvalues of $\mathbf{K}^{\xi_t}$ are strictly positive. One expects this to be true because of the presence of the diffusions in the PDE. However the author has not yet seen an easy proof of this.
\end{remark}
\begin{remark}\label{Remark Bifurcation}
Suppose that the dynamics is reversible, with spin-flipping intensity given by \eqref{eq: C Glauber}, $h = 0$ and the symmetry $\mathfrak{s}=1$. Preliminary numerical work by C.MacLaurin\footnote{Private communication.} has identified a family of fixed point solutions to \eqref{eq: p plus evolution start} with two replicas (i.e. $M=2$). Let $q \geq 0$ satisfy the implicit relationship
\begin{align}
\frac{1+q}{1-q} -\cosh(2\beta h)\exp\big( 2\beta^2 q\big) =& 0 \text{ and define the matrix elements }\label{eq: overlap equation} \\
K^{\xi}_{11} = K^{\xi}_{22}=& 1 \text{ and } K^{\xi}_{12} = K^{\xi}_{21} = q\nonumber \\
\upsilon^{\xi}_{11} = \upsilon^{\xi}_{22} =& \frac{\beta}{2}(1+q^2) \text{ and }\upsilon^{\xi}_{12} = \upsilon^{\xi}_{21} = \beta q \text{ and }
\kappa^{\xi}_{ij} =0.\nonumber
\end{align}
With the above definitions, the field distributions $p(\boldsymbol\alpha, \cdot)$ in the fixed point solution to \eqref{eq: p plus evolution start}
are weighted Gaussians. For $h=0$, there is a bifurcation as $\beta$ increases through $1$ in the solutions to \eqref{eq: overlap equation}: for $\beta \leq 1$, $q=0$ is the unique solution, but for $\beta > 1$, it is no longer unique. 
\end{remark}

\subsection{Proof Outline}

We discretize time into $(n+1)$ timesteps $\lbrace t^{(n)}_a \rbrace_{0\leq a \leq n}$: writing $\Delta = t^{(n)}_{a+1} - t^{(n)}_a = Tn^{-1}$. In Section \ref{Section Groundwork} we use an argument that is reminiscent of Gronwall's Inequality to demonstrate that if the action of the flow operator over the time interval $[t^{(n)}_a,t^{(n)}_{a+1}]$ corresponds to the dynamics of the empirical process to within an error of $o(\Delta)$, then the supremum of the difference between the empirical process and the flow over the entire interval $[0,T]$ must be small. We also introduce an approximate flow $\Psi_t$, obtained by evaluating the coefficients in the PDE at $\hat{\mu}^N_t$ rather than $\xi_t$. In subsequent sections it will be easier to compare $\Psi_t$ to $\hat{\mu}^N_t$ than to compare $\Phi_t$ to $\hat{\mu}^N_t$. \\

To accurately estimate the `average' change in the fields $G^{q,j}_{t^{(n)}_a} \to G^{q,j}_{t^{(n)}_{a+1}}$ we must perform a change-of-measure to a stochastic process $\tilde{\sigma}^{q,j}_{i,t}$ whose spin-flipping is independent of the connections. The reason for this change of measure is that now the changed fields $\tilde{G}^{i,j}_t := N^{-1/2}\sum_{k\in I_N}J^{jk}\tilde{\sigma}^{i,k}_t$ are Gaussian, and their incremental behavior can be accurately predicted by studying their covariance structure. In Section \ref{Section Change of Measure} we define $C^N_{\mathfrak{n}}$ such processes $\lbrace \tilde{\boldsymbol\sigma}_{i,t} \rbrace_{1\leq i \leq C^N_{\mathfrak{n}}}$, and we demonstrate that the probability law of the original $\mathcal{E}^{MN}$-valued process $\boldsymbol\sigma_t$ must be close to at least one of them using Girsanov's Theorem. The partition of the path space $\mathcal{D}([0,T],\mathcal{E}^M)^N$ is implemented using a second, finer, discretization of time into $\lbrace t^{(m)}_a \rbrace_{0\leq a \leq m}$, for some $m$ which is an integer multiple of $n$. This finer partition of time is needed to ensure that the Girsanov exponent is sufficiently close to unity.  \\
In Section \ref{Expectation Test Functions} we demonstrate that the Wasserstein distance can be approximated arbitrarily well by taking the supremum of the difference in expectation of a finite set of smooth functions. Working now exclusively with the processes $\tilde{\boldsymbol\sigma}_{i,t}$, we Taylor expand the change in expectation of such functions from $t^{(n)}_a$ to $t^{(n)}_{a+1}$, for both the empirical measure and the flow operator $\Phi_t$. The Taylor expansion implies that only the first two moments of the empirical measure and flow operator need to match in order that the change in the Wasserstein Distance is $o(\Delta)$. There are two basic types of term in the difference of the Taylor Expansions: (i) terms that can be bounded using concentration inequalities for Poisson Processes $\lbrace Y^{q,j}(t) \rbrace_{q\in I_M,j\in I_N}$, and (ii) terms that require the law $\gamma$ of the Gaussian connections $\lbrace J^{jk} \rbrace_{j,k \in I_N}$ to be accurately bounded.\\
In Section \ref{Section Stochastic Estimates}, we bound the terms (i) whose dynamics can be accurately predicted using the Law of Large numbers for Poisson Processes. These bounds typically involve concentration inequalities for compensated Poisson Processes (which are Martingales \cite{Anderson2015}). In Section \ref{Section Field}, we bound the terms (ii), using the conditional Gaussian probability law $\gamma_{\tilde{\boldsymbol\sigma},\tilde{\mathbf{G}}}$ - obtained by taking the law $\gamma \in \mathcal{M}^+_1(\mathbb{R}^{N^2})$ of the connections $\lbrace J^{jk} \rbrace_{j,k\in I_N}$ and conditioning on the values of the NM field variables $\lbrace \tilde{G}^{q,j}_{t^{(n)}_a} \rbrace_{q\in I_M, j\in I_N}$. We demonstrate that the average change in the field terms $\tilde{G}^{q,j}_{t^{(n)}_{a+1}} - \tilde{G}^{q,j}_{t^{(n)}_a}$ is governed by the first and second moments of $\gamma_{\boldsymbol\sigma,\mathbf{G}}$. The first moment ultimately leads to the term $\mathbf{m}^{\xi_t}$ in \eqref{eq: p plus evolution}, and the second moment ultimately leads to the diffusion coefficient $\sqrt{L^{\xi_t}_{ii}}$. \\
Before we commence the above plan, we require that the flow operator $\Phi_t$ is well defined.
\begin{proof}[Proof of Lemma \ref{Lemma xi t}]
We can interpret $p_t$ as the marginal probability law of the solution of a nonlinear SDE driven by a Levy Process.  \cite{Jourdain2007} proved the existence and uniqueness of a solution to such an SDE in the case that the coefficients are uniformly Lipschitz functions of the probability law (with respect to the Wasserstein distance). By contrast, our coefficients $\mathbf{m}^{\xi_t}$ and $\big(\mathbf{L}^{\xi_t}\big)^{1/2}$ (one must take the square root of the diffusion coefficient to obtain the coefficient of the stochastic integral) are only locally Lipschitz (see Lemma \ref{Lemma Lipschitz Matrices}).

To get around this, one first uses \cite{Jourdain2007} to show existence and uniqueness for an analogous system driven by uniformly Lipschitz coefficients $\hat{\mathbf{m}}^{\xi_t}$ and $\big\lbrace \big(\hat{L}_{ii}^{\xi_t}\big)^{1/2} \big\rbrace_{i\in I_M}$. These coefficients are taken to be identical to $\mathbf{m}^{\xi_t}$ and $(L_{ii}^{\xi_t})^{1/2} $ when $\xi_t \in \mathcal{D}_{\epsilon}$, where
\[
\mathcal{D}_{\epsilon} = \big\lbrace \mu \in \mathcal{P} \; : \; \sup_{i\in I_M}\mathbb{E}^{\mu}[ (x^i)^2] \leq \epsilon^{-1} \text{ and } \inf_{i\in I_M}\mathbb{E}^{\mu}[c(\alpha^i,x^i)] \geq \epsilon \big\rbrace .
\]
The solution is written as $\xi_{\epsilon,t}$. One then shows that for small enough $\epsilon$, $\xi_{\epsilon,t} \in \mathcal{D}_{\epsilon}$ for all $t\in [0,T]$. Once one has shown this, it must be that $\xi_t := \xi_{\epsilon,t}$ is the unique solution.

To do this, one can easily show (analogously to Lemma \ref{eq: compact Psi}) that for all $\epsilon > 0$, there exist constants $C_1,C_2 > 0$ such that
\[
\frac{d}{dt}\mathbb{E}^{\xi_{\epsilon,t}}[ (x^i)^2] \leq C_1 \mathbb{E}^{\xi_{\epsilon,t}}[ (x^i)^2]  + C_2.
\]
The boundedness of $\mathbb{E}^{\xi_{\epsilon,t}}[ (x^i)^2]$ then implies a lower bound for $\hat{L}^{\xi_{\epsilon,t}}_{ii}$, since for any $u > 0$, thanks to Chebyshev's Inequality, $\xi_{\epsilon,t}(| x^i|  \leq u ) \geq 1 - \mathbb{E}^{\xi_{\epsilon,t}}[(x^i)^2] u^{-2}$, and the continuity of $c$ implies that $\inf_{|x| \leq u , \sigma \in \mathcal{E}}c(\sigma,x) > 0$. Since $\mu \to L^{\mu}_{ii}$ is uniformly Lipschitz, it must be that $\mu \to D^{\mu}_{ii}$ is uniformly Lipschitz over $\mathcal{D}_{\epsilon}$, since $L^{\mu}_{ii}$ is bounded away from zero.
\end{proof}
The above existence and uniqueness proof requires that the coefficients of the PDE in \eqref{eq: p plus evolution} are Lipschitz. This is noted in the follow Lemma.
\begin{lemma}\label{Lemma Lipschitz Matrices}
(i) There exists a constant $C_1>0$ such that for any $\beta,\zeta\in \tilde{\mathcal{P}}$,
\begin{align}
\sup_{1\leq p,q \leq M} \big| L_{pq}^{\beta} - L^{\zeta}_{pq}\big| ,  \big| K_{pq}^{\beta} - K^{\zeta}_{pq}\big| &\leq  C_1 d_W(\beta,\zeta) \\
\sup_{1\leq p,q \leq M} \big| \upsilon_{pq}^{\beta} - \upsilon^{\zeta}_{pq}\big| ,  \big| \kappa_{pq}^{\beta} - \kappa^{\zeta}_{pq}\big| &\leq  C_1\big(1 + \mathbb{E}^{\beta}\big[\norm{\mathbf{x}}^2\big]^{\frac{1}{2}}\big)d_W(\beta,\zeta).
\end{align}
(ii) There is a constant $C>0$ such that for all $\beta,\zeta \in \tilde{\mathcal{P}}$ such that $\Lambda^{\beta},\Lambda^{\zeta} \geq \mathfrak{c}/2$, all $\boldsymbol\alpha,\boldsymbol\sigma \in \mathcal{E}^M$ and all $\mathbf{x},\mathbf{g} \in \mathbb{R}^M$, 
\begin{align}
\norm{\mathbf{m}^{\beta}(\boldsymbol\alpha,\mathbf{x}) - \mathbf{m}^{\zeta}(\boldsymbol\sigma,\mathbf{g})} &\leq Cd_W(\beta,\zeta)\big\lbrace 1 + \norm{\mathbf{g}} + \mathbb{E}^{\zeta}\big[\norm{\mathbf{g}}^2\big]^{\frac{1}{2}}\big\rbrace + C\norm{\mathbf{x-\mathbf{g}}} +C\big\lbrace 1 +  \mathbb{E}^{\zeta}\big[\norm{\mathbf{g}}^2\big]^{\frac{1}{2}}\big\rbrace\norm{\boldsymbol\alpha-\boldsymbol\sigma} \\
\norm{\mathbf{m}^{\beta}(\boldsymbol\alpha,\mathbf{g})}&\leq C \norm{\mathbf{g}} + C  \big(1 + \mathbb{E}^{\beta}\big[\norm{\mathbf{g}}^2\big]^{\frac{1}{2}} \big).\label{Lemma Bound m}
\end{align}
\end{lemma}
\begin{proof}
Both results follow almost immediately from the definitions, since $|c(\cdot,\cdot)|$ is uniformly bounded, and $|c(\alpha,x) - c(\alpha,g)| \leq c_L |x-g|$. It follows from the definition in \eqref{eq: H xi definition} that $\xi \to H_{jk}^{\xi}$ is uniformly Lipschitz (for all indices $j,k\in I_M$), since (as noted in (i) of this lemma) $\xi \to K_{jk}^{\xi}$ is uniformly Lipschitz. Furthermore $\big| H_{jk}^{\xi} \big|$ is uniformly bounded, because $| K^{\xi}_{jk}| \leq 1$.
\end{proof}
\section{Organization of Proof of Theorem 2.1}\label{Section Groundwork}

This section lays the groundwork for the proof of Theorem \ref{Theorem 1}, using an argument that is reminiscent of Gronwall's Inequality. The ultimate aim of this section is to demonstrate that, if the change in the empirical process over a small increment $\Delta$ in time is similar to the incremental change induced by the flow operator $\Phi_\Delta \cdot \hat{\mu}^N_t$, then the distance $\sup_{t\in [0,T]}d_W(\hat{\mu}^N_t , \Phi_t\cdot \hat{\mu}^N)$ is $O(\Delta^2)$. Thus this section reduces the proof of Theorem \ref{Theorem 1}, to the sufficient condition in Lemma \ref{Lemma Main Lemma}. The rest of the paper is then oriented towards proving Lemma \ref{Lemma Main Lemma}. The proofs of the lemmas stated just below are deferred to later in the section.


We will express the event in the statement of Theorem \ref{Theorem 1} as a union of $a_N$ subevents, i.e.
\[
\big\lbrace \sup_{t \leq \tau_N} d_W\big( \Phi_t\cdot \hat{\mu}^N(\boldsymbol\sigma_0,\mathbf{G}_0)  , \hat{\mu}^N(\boldsymbol\sigma_t,\mathbf{G}_t) \big) \geq \epsilon \big\rbrace \subseteq \bigcup_{j=1}^{a_N}\mathcal{A}^N_j. 
\]
As is noted in the following lemma, it will then suffice to show that the probability of each of the subevents $\lbrace \mathcal{A}^N_j\rbrace$ is exponentially decaying.
\begin{lemma}\label{Lemma Many Events}
Suppose that events $\lbrace \mathcal{A}^N_j \rbrace_{j=1}^{a_N}$ are such that $\lsup{N}N^{-1}\log a_N = 0$. Then
 \[
\lsup{N}N^{-1}\log \mathbb{P}\big(\bigcup_{j=1}^a\mathcal{A}^N_j \big)\leq  \lsup{N}\sup_{1\leq j \leq a_N} \big\lbrace N^{-1}\log \mathbb{P}\big(\mathcal{A}^N_j \big)\big\rbrace
\]
\end{lemma}
\begin{proof}
Immediate from the definitions.
\end{proof}
We now outline more precisely what these events are. First, we require that the matrix of connections is sufficiently regular. Let $\mathbf{J}_N$ be the $N\times N$ matrix with $(j,k)$ element equal to $N^{-\frac{1}{2}}J^{jk}$. Define $\mathcal{J}_N$ to be the event
\begin{align}
\mathcal{J}_N &= \big\lbrace \norm{\mathbf{J}_N} \leq 3\big\rbrace \text{ and }\label{eq: event JN} \\
\mathcal{W}_2 &= \big\lbrace \mu \in \mathcal{P} : \sup_{1\leq p \leq M}\Exp^{\mu (\boldsymbol\sigma,\mathbf{g})}\big[ (g^p)^2 \big] \leq 3\big\rbrace \text{ and }
\mathcal{W}_{2,\mathfrak{c}} = \big\lbrace \mu \in \mathcal{W}_{2} : \inf_{\mathfrak{a} \in \mathbb{R}^M:\norm{\mathfrak{a}}=1}\sum_{j,k=1}^M K^{\mu}_{jk}a^j a^k \geq \mathfrak{c} \big\rbrace . \label{eq: W2 definition}
\end{align}
The following lemma notes that $\mathcal{J}_N$ is overwhelmingly likely.
\begin{lemma}\label{bound connections lemma}
\begin{enumerate}
\item \begin{equation}
\lsup{N}N^{-1}\log \gamma\big( \mathcal{J}_N^c \big) := \Lambda_J < 0.
\end{equation}
\item Also,
\begin{equation}
 \mathcal{J}_N \subseteq \big\lbrace  \text{For all }t\geq 0,\; \hat{\mu}^N_t \in \mathcal{W}_2 \big\rbrace .
\end{equation}
\end{enumerate}
\end{lemma}
Define the spaces of measures
\begin{align} \label{eq: W 0 T definition}
\mathcal{W}_{[0,T]} &= \big\lbrace \mu_{[0,T]} \in \mathcal{D}\big( [0,T] , \mathcal{P}\big) \; : \; \mu_t \in \mathcal{W}_{2} \text{ and }t\to \mu_t \text{ has finite }\# \text{ discontinuities}\big\rbrace  \\
\hat{\mathcal{W}}_{[0,T]} &= \big\lbrace \mu_{[0,T]} \in \mathcal{W}_{[0,T]} : \sup_{t\in [0,T] , p\in I_M} \mathbb{E}^{\mu_t}\big[ \norm{\mathbf{x}^p}^2 \big] \leq 3 \big\rbrace. \label{eq: hat W 0 T}
\end{align} 
Next we define a map $\Psi: \mathcal{W}_{[0,T]} \to \mathcal{C}\big( [0,T],\mathcal{P}\big)$, $\Psi := (\Psi_t)_{t\in [0,T]}$, that is an approximation of the flow $\Phi_t$, such that the coefficients of the PDE are evaluated at $\hat{\mu}^N_t$, rather than $\xi_t$. More precisely, it is such that  $\Psi\cdot \mu_{[0,T]}  := \eta_{[0,T]}$, and for $t > 0$, $\eta_t$ has density $p_t$ satisfying the PDE 
\begin{equation}
\frac{\partial p_t }{\partial t}(\boldsymbol\alpha,\mathbf{x}) = \sum_{i\in I_M}\big\lbrace c(-\alpha^i,x^i) p_t(\boldsymbol\alpha[i],\mathbf{x})  - c(\alpha^i,x^i)p_t(\boldsymbol\alpha ,\mathbf{x} )+2 L^{\mu_t}_{ii}   \frac{\partial^2 p_t }{\partial x_i^2 }(\boldsymbol\alpha,\mathbf{x})\big\rbrace
- \nabla\cdot \big\lbrace \mathbf{m}^{\mu_t} p_t(\boldsymbol\alpha,\mathbf{x})\big\rbrace. \label{eq: p plus evolution 00}
\end{equation}
We insist that $\eta_0 = \mu_0$, and that $t \to \eta_t$ is continuous. Write $\Psi_t \cdot \mu_{[0,T]} := \eta_t$. One can easily check that $\Psi$ is uniquely defined.

The following lemma states that $\Psi$ is a good approximation of $\Phi$. The second result in the lemma is necessary for us to be sure that we avoid the pathological situation of $\Lambda^{\hat{\mu}^N_t} \to 0$, which would mean that the coefficients in the PDE blowup (see the definition in \eqref{eq: H xi definition}). Incidentally, this is precisely the reason that we require the stopping time $\tau_N$ in \eqref{eq: tilde tau N stopping}.
\begin{lemma}\label{Lemma Psi Close}
Define $\tilde{d}_T: \mathcal{D}\big( [0,T], \mathcal{P}\big) \times \mathcal{D}\big( [0,T], \mathcal{P}\big) \to \mathbb{R}^+$ to be 
\begin{equation}\label{eq: supremum metric}
\tilde{d}_T( \mu_{[0,T]} , \nu_{[0,T]} ) = \sup_{t\in [0,T]} d_W( \mu_t , \nu_t),
\end{equation}
noting that $\tilde{d}_T$ does not metrize the Skorohod topology. For any $\epsilon > 0$, there exists $\delta > 0$ such that
\begin{equation}\label{eq: epsilon delta equivalence}
\big\lbrace \mu \in \hat{\mathcal{W}}_{[0,T]}  \; : \tilde{d}_T ( \Psi\cdot \mu , \mu ) < \delta \big\rbrace \subseteq \big\lbrace \mu \in \hat{\mathcal{W}}_{[0,T]} \; : \tilde{d}_T( \Phi\cdot \mu_0 , \mu ) < \epsilon \big\rbrace
\end{equation}
Furthermore, there exists $\delta_{\mathfrak{c}}$ such that for all $\delta \leq \delta_{\mathfrak{c}}$,
\begin{equation}\label{eq: equivalent epsilon 2}
\mathbf{H}^{\Psi_t \cdot \hat{\mu}^N_b} = (\mathbf{K}^{\Psi_t\cdot \hat{\mu}^N_b})^{-1} \text{ as long as } t< \tau_N \text{ and } d_W\big( \Psi_t\cdot \hat{\mu}^N  , \hat{\mu}^N(\boldsymbol\sigma_t,\mathbf{G}_t) \big) \leq \delta.
\end{equation}
\end{lemma}

Next we discretize time, and also the flow $\Psi_t$. We partition the time interval $[ 0 , T ]$ into $ \lbrace t^{(n)}_b \rbrace_{b=0}^{n-1}$, with $t^{(n)}_b = b\Delta $ and $\Delta = T/n$. For any $t\in [0,T]$, define $t^{(n)} := \sup\lbrace t^{(n)}_b \; : t^{(n)}_b \leq t \rbrace$. We write $\Psi_b := \Psi_{t^{(n)}_b}$, $\hat{\mu}^N_b(\boldsymbol\sigma,\mathbf{G}) := \hat{\mu}^N_{t^{(n)}_b}$, $\boldsymbol\sigma_b := \boldsymbol\sigma_{t^{(n)}_b}$.

We can now decompose the event in the statement of Theorem \ref{Theorem 1} into the following events. It follows from Lemma \ref{Lemma Psi Close} that for any $\tilde{\epsilon} > 0$, there must exist $\epsilon > 0$ such that
\begin{align*}
 \big\lbrace \sup_{t \leq \tau_N} d_W\big( \Phi_t\cdot \hat{\mu}^N_0  , \hat{\mu}^N_t\big) \geq \tilde{\epsilon} \big\rbrace \subseteq & \big\lbrace \sup_{t \leq \tau_N} d_W\big( \Psi_t\cdot \hat{\mu}^N , \hat{\mu}^N_t \big) \geq \epsilon \big\rbrace \\
 \subseteq &\mathcal{J}_N^c \cup \bigcup_{0\leq b \leq n}\big\lbrace \mathcal{J}_N \text{ and } \sup_{t\in [t^{(n)}_b \wedge \tau_N, t^{(n)}_{b+1}\wedge \tau_N ]} d_W\big(\Psi_t\cdot \hat{\mu}^N , \Psi_b \cdot \hat{\mu}^N\big) \geq \epsilon / 3 \big\rbrace\cup \\
&\bigcup_{0\leq b \leq n}\ \big\lbrace  \mathcal{J}_N \text{ and } \sup_{t\in [t^{(n)}_b \wedge \tau_N, t^{(n)}_{b+1} \wedge \tau_N]} d_W\big( \hat{\mu}^N_t ,  \hat{\mu}^N_b\big) \geq \epsilon / 3 \big\rbrace \cup \\
&\bigcup_{0\leq b \leq n}\big\lbrace  \mathcal{J}_N \text{ and for some }b \text{ such that }\tau_N > t^{(n)}_b,\;  d_W\big(\hat{\mu}^N_b , \Psi_b \cdot \hat{\mu}^N\big) \geq \epsilon / 3 \big\rbrace.
 \end{align*}
It is assumed that $\epsilon \leq \delta_{\mathfrak{c}}$, as defined in Lemma \ref{Lemma Psi Close}. Thanks to Lemma \ref{Lemma Many Events}, for Theorem \ref{Theorem 1}, to hold, it thus suffices to prove that some $n\in \mathbb{Z}^+$,
\begin{align}
\sup_{0\leq b < n} \lsup{N}N^{-1}\log \mathbb{P}\big( \mathcal{J}_N \text{ and } \sup_{t\in [t^{(n)}_b \wedge \tau_N, t^{(n)}_{b+1}\wedge \tau_N ]} d_W(\Psi_t\cdot \hat{\mu}^N , \Psi_b \cdot \hat{\mu}^N) \geq \epsilon / 3 \big) < 0\label{eq: to demonstrate fluctuations 0} \\
\sup_{0\leq b < n} \lsup{N}N^{-1}\log \mathbb{P}\big( \mathcal{J}_N \text{ and } \sup_{t\in [t^{(n)}_b \wedge \tau_N, t^{(n)}_{b+1} \wedge \tau_N]} d_W( \hat{\mu}^N_t ,  \hat{\mu}^N_b) \geq \epsilon / 3 \big) < 0\label{eq: to demonstrate fluctuations} \\
\lsup{N} N^{-1}\log \mathbb{P}\big( \mathcal{J}_N \text{ and for some }b \text{ such that }\tau_N > t^{(n)}_b,\;  d_W(\hat{\mu}^N_b , \Psi_b \cdot \hat{\mu}^N) \geq \epsilon / 3 \big) < 0\label{eq: to establish final} \\
\lsup{N}N^{-1}\log \mathbb{P}\big( \mathcal{J}_N^c  \big) < 0 .\label{eq: to establish JcN}
\end{align}
 \eqref{eq: to demonstrate fluctuations 0} is demonstrated in Lemma \ref{eq: compact Psi},  \eqref{eq: to demonstrate fluctuations} is established in Lemma \ref{Lemma bound fluctuations} and
\eqref{eq: to establish JcN} is a consequence of Lemma \ref{bound connections lemma}.

In order that Theorem \ref{Theorem 1} is true, it thus only remains to prove \eqref{eq: to establish final}. Define the events $\lbrace \mathcal{U}^N_b \rbrace_{b=0}^{n-1}$, for a positive constant $\mathfrak{u} > 0$ (to be specified more precisely below - for the moment we note that $\mathfrak{u}$ will be chosen independently of $n$ and $N$), and writing $\tilde{\epsilon} = \epsilon / 3$,
\begin{equation}\label{Eq: U N b event}
 \mathcal{U}^N_b = \big\lbrace \mathcal{J}_N ,d_{W}\big( \Psi_{b+1}\cdot \hat{\mu}^N  , \hat{\mu}^N_{b+1} \big) > \tilde{\epsilon}\exp\big( \mathfrak{u} t^{(n)}_{b+1} / T -\mathfrak{u} \big), d_W\big( \Psi_b\cdot \hat{\mu}^N , \hat{\mu}^N_b \big) \leq \tilde{\epsilon}\exp\big( \mathfrak{u} t^{(n)}_{b} / T -\mathfrak{u} \big)  \text{ and }\tau_N > t^{(n)}_{b} \big\rbrace ,
\end{equation}
and observe that
\[
\big\lbrace\mathcal{J}_N \text{ and for some }b \text{ such that }\tau_N > t^{(n)}_b,\;  d_W\big(\hat{\mu}^N_b , \Psi_b \cdot \hat{\mu}^N\big) > \tilde{\epsilon} \big\rbrace \subseteq \bigcup_{b=0}^{n-1} \mathcal{U}^N_b.
\]
We thus find from Lemma \ref{Lemma Many Events} that, in order that \eqref{eq: to establish final} holds, it suffices to prove that
\begin{equation}\label{eq: to demonstrate U N b 0}
\sup_{0\leq b < n}\lsup{N} N^{-1}\log \mathbb{P}\big( \mathcal{U}^N_b \big) < 0.
\end{equation}
We now make a further approximation to the operator $\Psi_t$. For any $\boldsymbol\sigma \in \mathcal{E}^{MN}$ and $\mathbf{G} \in \mathbb{R}^{MN}$, define the random measure $\xi_b(\boldsymbol\sigma,\mathbf{G}) \in \mathcal{P}$, which is such that $\xi_b(\boldsymbol\sigma,\mathbf{G}) \simeq \Psi_{b+1}\cdot \hat{\mu}^N(\boldsymbol\sigma,\mathbf{G})$, as follows. Let $\lbrace \tilde{Y}^p(t) \rbrace_{p\in I_M}$ be independent Poisson Counting Processes, and $\lbrace \tilde{W}^p_t \rbrace_{p=1}^M$ independent Wiener Processes (they are also independent of the proceses $Y^{p,j}(t)$ and connections $\mathbf{J}$ used to define the original system). Writing $\hat{\mu}^N_{b}(\boldsymbol\sigma,\mathbf{G})$ to be the law of random variables $(\boldsymbol\zeta_0,\mathbf{x}_0)$, define $\xi_b(\boldsymbol\sigma,\mathbf{G})$ to be the law of $(\boldsymbol\zeta_{\delta t} ,\mathbf{x}_{\delta t})$, where, recalling that $A\cdot x := (-1)^x$, for each $p\in I_M$,
\begin{align}
\zeta^p_{\delta t} =& \zeta^p_0 A\cdot \tilde{Y}^p\big( \Delta c(\zeta^p_0,x^p_0) \big) \label{eq: hybrid 17} \\
\mathbf{x}_{\delta t} =& \mathbf{x}_0 + \delta t\mathbf{m}^{\hat{\mu}^N_b}(\boldsymbol\zeta_0 , \mathbf{x}_0)+\mathbf{D}^{\hat{\mu}^N_b}\tilde{\mathbf{W}}_{\delta t}\text{ , where }
D^{\hat{\mu}^N_b}_{ij} = 2\sqrt{L^{\hat{\mu}^N_b}_{ii}}\delta(i,j),\label{eq: hybrid 18} 
\end{align} 
and $\Delta = T/n$. When the context is clear, we omit the argument of $\xi_b$.

It follows from the facts that (i) $ d_{W}\big( \Psi_{b+1}\cdot \hat{\mu}^N  , \hat{\mu}^N_{b+1} \big) \leq d_{W}\big(\xi_b ,  \hat{\mu}^N_{b+1} \big) +d_{W}\big( \Psi_{b+1}\cdot \hat{\mu}^N, \xi_b\big) $ and (ii)  $\exp( \mathfrak{u} t^{(n)}_{b+1} / T  -\mathfrak{u} ) \geq \exp( \mathfrak{u} t^{(n)}_{b} / T + \mathfrak{u}\Delta / 2T -\mathfrak{u} ) + \exp( \mathfrak{u} t^{(n)}_{b} / T-\mathfrak{u} )\mathfrak{u} \Delta / 2T$ (recalling that $\Delta = t^{(n)}_{b+1} - t^{(n)}_b$), that
\begin{multline}
\big\lbrace d_{W}\big( \Psi_{b+1}\cdot \hat{\mu}^N  , \hat{\mu}^N_{b+1} \big) > \tilde{\epsilon}\exp\big( \mathfrak{u} t^{(n)}_{b+1} / T -\mathfrak{u} \big)\big\rbrace \subseteq  \big\lbrace d_{W}\big(\xi_b ,  \hat{\mu}^N_{b+1} \big) > \exp( \mathfrak{u} t^{(n)}_{b} / T-\mathfrak{u} )\tilde{\epsilon} \mathfrak{u} \Delta / 2T \big\rbrace \\
\cup \big\lbrace d_{W}\big( \Psi_{b+1}\cdot \hat{\mu}^N, \xi_b\big)>  \exp\big( \mathfrak{u} t^{(n)}_{b} / T+\mathfrak{u}\Delta / 2T -\mathfrak{u} \big) \tilde{\epsilon} \text{ and }d_{W}\big(\xi_b ,  \hat{\mu}^N_{b+1} \big) \leq \exp( \mathfrak{u} t^{(n)}_{b} / T-\mathfrak{u} )\tilde{\epsilon} \mathfrak{u} \Delta / 2T\big\rbrace \\
\subseteq \big\lbrace d_{W}\big(\xi_b ,  \hat{\mu}^N_{b+1} \big) > \exp( \mathfrak{u} t^{(n)}_{b} / T-\mathfrak{u} )\tilde{\epsilon} \mathfrak{u} \Delta / 2T \big\rbrace 
\cup \big\lbrace d_{W}\big( \Psi_{b+1}\cdot \hat{\mu}^N, \xi_b\big)>  \exp\big( \mathfrak{u} t^{(n)}_{b} / T+\mathfrak{u}\Delta / 2T -\mathfrak{u} \big) \tilde{\epsilon}\big\rbrace \label{eq: temporary u bound 1}
\end{multline}
We thus find that
\begin{multline*}
\mathcal{U}^N_b \subseteq  \big\lbrace \mathcal{J}_N \text{ and } d_{W}\big(\xi_b ,  \hat{\mu}^N_{b+1} \big) > \exp( \mathfrak{u} t^{(n)}_{b} / T-\mathfrak{u} )\tilde{\epsilon} \mathfrak{u} \Delta / (2T) \big\rbrace \bigcup \\ \big\lbrace \mathcal{J}_N \text{ and }d_{W}\big( \Psi_{ b+1}\cdot \hat{\mu}^N,\xi_b \big) >  \tilde{\epsilon}\exp\big( \mathfrak{u} t^{(n)}_{b} / T + \mathfrak{u}\Delta / (2T) -\mathfrak{u} \big) \text{ and } d_W\big( \Psi_{b}\cdot \hat{\mu}^N, \hat{\mu}^N_b\big)\leq  \exp\big( \mathfrak{u} t^{(n)}_{b} / T -\mathfrak{u} \big) \tilde{\epsilon}\big\rbrace .
\end{multline*}
Therefore \eqref{eq: to demonstrate U N b 0} will be seen to be true once we demonstrate Lemmas \ref{Lemma Last Lemma 1} and \ref{Lemma Main Lemma}.
\begin{lemma}\label{Lemma Last Lemma 1}
For any $\tilde{\epsilon} > 0$, for all sufficiently large $n$, and all $b$ such that $0\leq b < n$,
\begin{multline}
\lsup{N}N^{-1}\log \mathbb{P}\big(  \mathcal{J}_N\text{ and }\tau_N > t^{(n)}_{b+1} \text{ and }d_{W}( \Psi_{ b+1}\cdot \hat{\mu}^N, \xi_b ) > \tilde{\epsilon}\exp( \mathfrak{u} t^{(n)}_{b} / T + \mathfrak{u}\Delta / 2T -\mathfrak{u} )\\ \text{ and }d_W( \Psi_{b}\cdot \hat{\mu}^N, \hat{\mu}^N_b) \leq \tilde{\epsilon}\exp( \mathfrak{u} t^{(n)}_{b} / T -\mathfrak{u} ) \big) < 0.
\end{multline}
\end{lemma}
Lemma \ref{Lemma Last Lemma 1} is proved later in this section.
\begin{lemma}\label{Lemma Main Lemma}
Suppose that for any $\bar{\epsilon} > 0$, for all sufficiently large $n$ and all $0\leq b \leq n-1$,
\begin{align}\label{Main Lemma Condition}
\lsup{N}N^{-1}\log \mathbb{P}\big(\mathcal{J}_N\text{ and }\tau_N > t^{(n)}_{b}  \text{ and }  d_W\big(\xi_b ,  \hat{\mu}^N_{b+1}\big)&\geq \bar{\epsilon}\Delta \big) < 0.
\end{align}
Then Theorem \ref{Theorem 1} must be true.
\end{lemma}
The rest of this paper is devoted to establishing Lemma \ref{Lemma Main Lemma}. In the next section, Lemma \ref{Lemma Change of Measure} determines a sufficient condition for Lemma \ref{Lemma Main Lemma} to hold, in terms of processes $\lbrace \tilde{\boldsymbol\sigma}_{i,t} \rbrace$ whose spin-flipping is independent of the connections. The rest of the sections then prove that the condition of Lemma \ref{Lemma Change of Measure} must be satisfied.

\subsection{Regularity of the Connections: Proof of Lemma \ref{bound connections lemma}}
\begin{proof}
We decompose $\mathbf{J}_N$ into a symmetric matrix and an i.i.d. matrix, i.e.  $\mathbf{J}_N = N^{-1/2}\sqrt{\mathfrak{s}}\hat{\mathbf{J}}_N + N^{-1/2}\sqrt{1-\mathfrak{s}}\tilde{\mathbf{J}}_N +N^{-1/2}\mathbf{D}_N$. Here $\mathbf{D}_N$ is diagonal, $\hat{\mathbf{J}}_N$ is symmetric and $\tilde{\mathbf{J}}_N$ is neither symmetric nor anti-symmetric. The entries in all three matrices can be taken to be i.i.d of zero mean and unit variance (in the symmetric matrix the entries are i.i.d. apart from the symmetry $J^{jk} = J^{kj}$). A union-of-events bound implies that
\begin{multline*}
\lsup{N}N^{-1}\log \mathbb{P}\big( \norm{\mathbf{J}}_N > 3 \big) \leq \max\big\lbrace \lsup{N}N^{-1}\log \mathbb{P}\big( N^{-1/2}\sqrt{\mathfrak{s}}\norm{\hat{\mathbf{J}}_N}> 4/3 \big) , \\  \lsup{N}N^{-1}\log \mathbb{P}\big( N^{-1/2}\sqrt{1-\mathfrak{s}}\norm{\tilde{\mathbf{J}}_N}> 4/3 \big) , \lsup{N}N^{-1}\log \mathbb{P}\big( N^{-1/2}\norm{\mathbf{D}_N}> 1/3 \big) \big\rbrace .
\end{multline*}
For the last term, using Lemma \ref{Lemma Many Events}
\begin{align*}
 \lsup{N}N^{-1}\log \mathbb{P}\big( N^{-1/2}\norm{\mathbf{D}_N}> 1/3 \big) =  \lsup{N}\sup_{1\leq p \leq N}N^{-1}\log \mathbb{P}\big( N^{-1/2}|D_{N,pp}| > 1/3 \big) < 0.
\end{align*}
It is a standard result from random matrix theory \cite{Anderson2011} that
\[
\lsup{N}N^{-1}\log \mathbb{P}\big( N^{-1/2}\sqrt{\mathfrak{s}}\norm{\hat{\mathbf{J}}_N}> 4/3 \big) < 0.
\]
The last bound follows from recent results on the maximum eigenvalue of the Ginibre ensemble \cite{Poplavskyi2017} ,
\[
 \lsup{N}N^{-1}\log \mathbb{P}\big( N^{-1/2}\sqrt{1-\mathfrak{s}}\norm{\tilde{\mathbf{J}}_N}> 4/3 \big) < 0.
\]
 For (2), it may be observed that
\begin{align*}
\Exp^{\hat{\mu}^N_t}\big[(g^p)^2 \big] = N^{-1}\sum_{j\in I_N} (G^{p,j}_t)^2 
\leq N^{-1}\norm{\mathbf{J}_N}\sum_{j\in I_N}(\sigma^{p,j}_t)^2 
= \norm{\mathbf{J}_N}
\leq 3,
\end{align*}
as long as $\mathcal{J}_N$ holds. 
\end{proof}

\subsection{Approximating Flow $\Psi_t$}\label{Section Psi t Definition}
This section proves that $\Psi_t$ is a good approximation to the flow $\Phi_t$. We now prove Lemma \ref{eq: compact Psi}, which implies that the operator $\Psi$ is compact.
\begin{lemma}\label{eq: compact Psi}
There exists a constant $\bar{C} > 0$ such that for all $\mu_{[0,T]} \in \hat{\mathcal{W}}_{[0,T]}$, and writing $\eta_t = \Psi_t \cdot \mu_{[0,T]}$,
\begin{align}
\sup_{0\leq t \leq T}\mathbb{E}^{\eta_t}\big[ \norm{\mathbf{x}}^2 \big] &\leq \bar{C} \\
d_W(\eta_t , \eta_u ) &\leq \bar{C}\sqrt{t-u} \text{ for all }t \geq u. \label{eq: eta t eta u}
\end{align}
\end{lemma}
\begin{proof}
 To implement the Wasserstein distance, we require a common probability space, and it is easiest to use the stochastic process with marginal probability laws given by \eqref{eq: p plus evolution 00}. That is, $\eta_t \in \mathcal{P}$ is the marginal law of the solution $(\boldsymbol\alpha_t,\mathbf{z}_t)$ of the following stochastic hybrid system. Let $\lbrace \tilde{Y}^p(t) \rbrace_{p\in I_M}$ be independent Poisson Counting Processes, and $\lbrace \tilde{W}^p_t \rbrace_{p \in I_M}$ independent Wiener Processes (these processses are also independent of the Poisson processes $\lbrace Y^{p,j}(t)\rbrace_{p \in I_M,j \in I_N}$ and connections $\lbrace J^{jk} \rbrace_{j,k\in I_N}$ used to define the original system) and define for $p\in I_M$,
\begin{align}
\alpha^p_{t} =& \alpha^p_0 A\cdot \tilde{Y}^p\bigg( \int_0^t c(\alpha^p_s,x^p_s)ds \bigg) \label{eq: hybrid 19} \\
\mathbf{x}_{t} =& \mathbf{x}_0 + \int_0^t \mathbf{m}^{\mu_s}(\boldsymbol\alpha_s , \mathbf{x}_s)ds+\int_0^t\mathbf{D}^{\mu_s}d\tilde{\mathbf{W}}_{s}\text{ , where }
D^{\mu_s}_{ij} =2 \sqrt{L^{\mu_s}_{ii}}\delta(i,j),\label{eq: hybrid 20} 
\end{align}
and the initial random variables $(\boldsymbol\alpha_0 , \mathbf{x}_0)$ are distributed according to $\mu_0$.  One easily checks that a unique solution exists to the above equation.

We first establish that there exists a constant $\tilde{C}$ such that
\begin{equation}\label{eq: to establish eta t expectation}
\frac{d}{dt}\mathbb{E}^{\eta_t}\big[ \norm{\mathbf{x}}^2 \big] \leq \tilde{C}\big\lbrace 1 + \mathbb{E}^{\eta_t}\big[ \norm{\mathbf{x}}^2 \big]  \big\rbrace.
\end{equation}
Thanks to Ito's Lemma,
\begin{align}
\norm{\mathbf{x}_{t}}^2 =& \norm{\mathbf{x}_0}^2 + \int_0^t\big\lbrace \sum_{p\in I_M}4 L^{\mu_s}_{pp} + 2\big\langle \mathbf{x}_s, \mathbf{m}^{\mu_s}(\boldsymbol\alpha_s , \mathbf{x}_s) \big\rangle \big\rbrace ds+2\int_0^t\big\langle \mathbf{x}_s,\mathbf{D}^{\mu_s}d\tilde{\mathbf{W}}_{s}\big\rangle , \label{eq: temporary Ito xt expectation}
\end{align}
where $D^{\mu_s}_{ij} =2 \sqrt{L^{\mu_s}_{ii}}\delta(i,j)$. It follows from \eqref{Lemma Bound m} (and the Cauchy-Schwarz Inequality) that 
\[
\big\langle \mathbf{x}_s, \mathbf{m}^{\mu_s}(\boldsymbol\alpha_s , \mathbf{x}_s) \big\rangle \leq C \norm{\mathbf{x}_s}^2 + C\norm{\mathbf{x}_s}(1 +\sup_{s\in [0,T]} \mathbb{E}^{\mu_s}[\norm{\mathbf{x}}^2]^{1/2}).
\]
The definition of $\hat{\mathcal{W}}_{[0,T]}$ implies that $\sup_{s\in [0,T]} \mathbb{E}^{\mu_s}[\norm{\mathbf{g}}^2] \leq 3$, and it is immediate from the definition that $|L_{ii}^{\mu_s}| \leq c_1$. Thus taking expectations of both sides of \eqref{eq: temporary Ito xt expectation}, we obtain \eqref{eq: to establish eta t expectation} as required.

An application of Gronwall's Inequality to \eqref{eq: to establish eta t expectation} implies that
\begin{equation}\label{eq: upperbound eta t expect}
\sup_{0\leq t \leq T}\mathbb{E}^{\eta_t}\big[ \norm{\mathbf{x}}^2 \big] \leq \big( \tilde{C}T + \mathbb{E}^{\mu_0}\big[\norm{\mathbf{x}}^2\big] \big)\exp\big( \tilde{C}T \big),
\end{equation}
which establishes the first identity, since (by definition) $ \mathbb{E}^{\mu_0}\big[\norm{\mathbf{x}}^2\big] \leq 3$. It remains to demonstrate uniform continuity. It follows from Ito's Lemma that for all $t > u$,
\begin{equation}
\norm{\mathbf{x}_t - \mathbf{x}_u}^2 =2 \int_u^t \big\langle \mathbf{x}_s - \mathbf{x}_u ,  \mathbf{m}^{\mu_s}(\boldsymbol\alpha_s , \mathbf{x}_s) \big\rangle ds+ 4\int_u^t \sum_{i\in I_M}L^{\mu_s}_{ii} ds + 2\int_u^t  \big\langle \mathbf{x}_s - \mathbf{x}_u , \mathbf{D}^{\mu_s}d\mathbf{W}_s \big\rangle.
\end{equation}
We thus find that, using the Cauchy-Schwarz inequality,
\begin{align*}
\frac{d}{dt}\mathbb{E}\big[ \norm{\mathbf{x}_t - \mathbf{x}_u}^2 \big]
\leq 2\mathbb{E}\big[ \norm{\mathbf{x}_t - \mathbf{x}_u}^2\big]^{1/2}\mathbb{E}\big[ \norm{ \mathbf{m}^{\mu_t}(\boldsymbol\alpha_t , \mathbf{x}_t)}^2 \big]^{1/2} + 2Mc_1,
\end{align*}
since $|L_{ii}^{\mu_t}|$ is uniformly upperbounded by $c_1$ (the uniform upperbound for the jump intensity). It follows from \eqref{Lemma Bound m} that, using the inequality $(a+b)^2 \leq 2a^2 + 2b^2$,
\begin{equation}
\mathbb{E}\big[ \norm{ \mathbf{m}^{\mu_t}(\boldsymbol\alpha_t , \mathbf{x}_t)}^2 \big] \leq 2C^2\mathbb{E}\big[\norm{\mathbf{x}_t}^2 \big] + 2C^2\big\lbrace 1+ \mathbb{E}^{\mu_t}\big[ \norm{\mathbf{x}}^2 \big]^{1/2} \big\rbrace^2.
\end{equation}
Thanks to the definition of $\hat{\mathcal{W}}_{[0,T]}$, $ \mathbb{E}^{\mu_t}[ \norm{\mathbf{x}}^2 ] \leq 3$. It therefore follows from \eqref{eq: upperbound eta t expect} that there exists a constant $\hat{C}$ such that 
\begin{align*}
\frac{d}{dt}\mathbb{E}\big[ \norm{\mathbf{x}_t - \mathbf{x}_u}^2 \big]
&\leq \hat{C}\mathbb{E}\big[ \norm{\mathbf{x}_t - \mathbf{x}_u}^2\big]^{1/2}+ 2Mc_1 \\
&\leq \hat{C}\mathbb{E}\big[ \norm{\mathbf{x}_t - \mathbf{x}_u}^2\big] + \hat{C}+ 2Mc_1.
\end{align*}
Gronwall's Inequality now implies that
\begin{equation}
\mathbb{E}\big[ \norm{\mathbf{x}_t - \mathbf{x}_u}^2 \big] \leq (t-u)\exp\big\lbrace (t-u) \hat{C} \big\rbrace  \big( 2Mc_1 + \hat{C} \big),
\end{equation}
and Jensen's Inequality therefore implies that
\begin{equation}
\mathbb{E}\big[ \norm{\mathbf{x}_t - \mathbf{x}_u} \big] \leq \big( (t-u)\exp\big\lbrace (t-u) \hat{C} \big\rbrace  ( 2Mc_1 + \hat{C} ) \big)^{1/2}.
\end{equation}
The uniform bound $c_1$ for the intensity of the spin-flipping implies that
\begin{equation}
\mathbb{E}\big[ \norm{\boldsymbol\alpha_t - \boldsymbol\alpha_u}^2 \big] \leq 4M(t-u)c_1.
\end{equation}
The above two identities imply \eqref{eq: eta t eta u}.
\end{proof}
We now prove Lemma \ref{Lemma Psi Close}.
\begin{proof}
The second result in Lemma \ref{eq: compact Psi} implies that all elements of $\Psi \cdot \mathcal{W}_{[0,T]}$ are uniformly continuous. The first result in Lemma \ref{eq: compact Psi} implies that the individual marginals $\lbrace \eta_t \rbrace$ belong to the compact space of measures 
\begin{equation}\label{eq: bar P definition}
\bar{\mathcal{P}} = \big\lbrace \mu \in \mathcal{P} \; :\mathbb{E}^{\mu}\big[ \norm{\mathbf{g}}^2 \big] \leq \bar{C}
\big\rbrace.
\end{equation}
(This space is compact thanks to Prokhorov's Theorem). 
It thus follows from the generalized Arzela-Ascoli Theorem \cite{Green1961} that $\Psi \cdot \mathcal{W}_{[0,T]}$ is compact in $\mathcal{C}([0,T],\mathcal{P})$ (this space being endowed with the supremum metric \eqref{eq: supremum metric}).

Suppose for a contradiction that the lemma were not true. Then there would have to exist some $\tilde{\epsilon} > 0$ and some sequence $\mu^n \in \hat{\mathcal{W}}_2$ such that $\tilde{d}_T( \Psi\cdot \mu^n , \mu^n ) < n^{-1}$ and $\tilde{d}_T( \Phi\cdot \mu^n_0 , \mu^n ) \geq \tilde{\epsilon}$. The compactness of the space $\Psi \cdot \hat{\mathcal{W}}_{[0,T]}$ means that $\big( \Psi\cdot \mu^n \big)_{n\in \mathbb{Z}^+}$ must have a convergent subsequence $\big( \Psi\cdot \mu^{p_n} \big)_{n\in \mathbb{Z}^+}$, converging to some $\phi = (\phi_t)_{t\in [0,T]}$. Since $\tilde{d}_T( \Psi\cdot \mu^n , \mu^n ) < n^{-1}$, it must be that $\mu^{p_n} \to \phi$ as well. Since $\Psi$ is continuous, $\Psi\cdot\mu^{p_n}$ also converges to $\Psi\cdot \phi$. We thus find that $\Psi\cdot \phi = \phi$. This contradicts the uniqueness of the fixed point $\xi_t$ established in Lemma \ref{Lemma xi t}.

It remains to prove \eqref{eq: equivalent epsilon 2}. First we note that for small enough $\epsilon$, we are certain to avoid the pathological situation of $\Lambda^{\Phi_t\cdot \hat{\mu}^N} \to 0$ for $t \leq \tau_N$. This event would imply that $\norm{(\mathbf{K}^{\xi_t})^{-1}} \to \infty$ (and the PDE in \eqref{eq: p plus evolution} would no longer be accurate). 
Let $\epsilon_{\mathfrak{c}}>0$ be the largest number such that
\begin{equation}\label{eq: epsilon c definition}
\big\lbrace \mu \in \mathcal{P} \; : \Lambda^{\mu} \geq \mathfrak{c} \big\rbrace = \big\lbrace \mu \in \mathcal{P} \; :  \Lambda^{\mu} \geq \mathfrak{c} \text{ and } \Lambda^{\nu} \geq \mathfrak{c}/2 \text{ for all }\nu \text{ such that }d_W(\mu,\nu) \leq \epsilon_{\mathfrak{c}}  \big\rbrace.
\end{equation}
Such an $\epsilon_{\mathfrak{c}}$ always exists because the map $\Lambda^{\mu}$ is continuous. We will thus assume (throughout the rest of this paper) that $\epsilon \leq \epsilon_{\mathfrak{c}}$, because in any case if the RHS of the following inequality is less than zero, then the LHS must be less than zero too, i.e.
\begin{multline}
\lsup{N}N^{-1}\log \mathbb{P}\big( \sup_{t \leq \tau_N} d_W\big( \Phi_t\cdot \hat{\mu}^N(\boldsymbol\sigma_0,\mathbf{G}_0)  , \hat{\mu}^N(\boldsymbol\sigma_t,\mathbf{G}_t) \big) \geq \epsilon \big) \leq\\
\lsup{N}N^{-1}\log \mathbb{P}\big( \sup_{t \leq \tau_N} d_W\big( \Phi_t\cdot \hat{\mu}^N(\boldsymbol\sigma_0,\mathbf{G}_0)  , \hat{\mu}^N(\boldsymbol\sigma_t,\mathbf{G}_t) \big) \geq \min \lbrace \epsilon , \epsilon_{\mathfrak{c}} \rbrace \big).
\end{multline}
With this choice of $\epsilon$, we are assured that $\mathbb{P}\big(\mathcal{Q}_N^c\big) = 0$ where 
\begin{equation}\label{eq: equivalent epsilon}
\mathcal{Q}_N = \big\lbrace \mathbf{H}^{\Phi_t \cdot \hat{\mu}^N_0} = (\mathbf{K}^{\Phi_t\cdot \hat{\mu}^N_0})^{-1} \text{ as long as } t< \tau_N \text{ and } d_W\big( \Phi_t\cdot \hat{\mu}^N(\boldsymbol\sigma_0,\mathbf{G}_0)  , \hat{\mu}^N(\boldsymbol\sigma_t,\mathbf{G}_t) \big) \leq \epsilon \big\rbrace.
\end{equation}
 As long as $\epsilon \leq \epsilon_{\mathfrak{c}}$ (defined just above \eqref{eq: epsilon c definition}), and $\delta$ is chosen such that \eqref{eq: epsilon delta equivalence} is satisfied, then \eqref{eq: equivalent epsilon 2} must hold.
\end{proof}

\subsection{Proofs of the Remaining Lemmas}\label{Section Growth in Distance}


\begin{lemma}\label{Lemma bound fluctuations}
For any $\epsilon > 0$, for all sufficiently large $n$,
\begin{align}
\sup_{0\leq b < n}\lsup{N}N^{-1}\log  \mathbb{P}\big(\mathcal{J}_N \text{ and }\sup_{s\in [t^{(n)}_b, t^{(n)}_{b+1}]}d_W\big(\hat{\mu}^N(\boldsymbol\sigma_b,\mathbf{G}_b) , \hat{\mu}^N(\boldsymbol\sigma_s,\mathbf{G}_s)\big) \geq \epsilon  \big) < 0.\label{eq: first empirical measure fluctuation} 
\end{align}
\end{lemma}\begin{proof}
It follows from the definition that
\begin{align*}
d_W\big(\hat{\mu}^N(\boldsymbol\sigma_b,\mathbf{G}_b) , \hat{\mu}^N(\boldsymbol\sigma_s,\mathbf{G}_s)\big) &\leq \big(N^{-1}\sum_{j \in I_{N},i\in I_M}  \big| G^{i,j}_{b} -G^{i,j}_{s}\big|^2 \big)^{\frac{1}{2}} + N^{-1}\sum_{j\in I_N, i\in I_M} \big| \sigma^{i,j}_{s} - \sigma^{i,j}_{b}\big|.
\end{align*}
The renewal property of Poisson Processes implies that the following processes $\lbrace Y^{q,j}_a(t) \rbrace_{q\in I_M, j \in I_N}$ are Poissonian:
\begin{align}
Y_b^{q,j}(t) :=& Y^{q,j}\big(t+ \int_0^{t^{(n)}_b} c(\sigma^{q,j}_s , G^{q,j}_s) ds \big) - Y^{q,j}\big( \int_0^{t^{(n)}_b} c(\sigma^{q,j}_s , G^{q,j}_s) ds \big) \label{eq: Y a q j t definition}
\end{align}
Now as long as the event $\mathcal{J}_N$ holds,
\begin{align*}
N^{-1}\sum_{j \in I_{N},i\in I_M}  \big| G^{i,j}_{b} -G^{i,j}_{s}\big|^2 &\leq  \frac{3}{N}\sum_{j\in I_N, i \in I_M} \lbrace \sigma^{i,j}_{s} - \sigma^{i,j}_{b}\rbrace^2 \\
&\leq 12 N^{-1}\sum_{j\in I_N, i \in I_M} Y_b^{i,j}\big(c_1\lbrace s- t^{(n)}_b\rbrace \big) \leq  12N^{-1}\sum_{j\in I_N, i \in I_M} Y_b^{i,j}\big(c_1\Delta\big).
\end{align*}
Similarly, $N^{-1}\sum_{j\in I_N, i\in I_M} \big| \sigma^{i,j}_{s} - \sigma^{i,j}_{b}\big| \leq 2N^{-1}\sum_{j\in I_N, i\in I_M} Y_b^{i,j}(c_1 \Delta)$.
Writing $\bar{\epsilon}$ to be such that $\sqrt{12\bar{\epsilon}} + 2\bar{\epsilon} = \epsilon$, and noting that $Y_b^{i,j}$ is non-decreasing, it thus suffices to prove that for any $\bar{\epsilon} >0 $,
\begin{equation}\label{eq: Sanov to show}
\lsup{N}N^{-1}\log  \mathbb{P}\big(\mathcal{J}_N \text{ and }N^{-1}\sum_{j\in I_N, i\in I_M} Y_b^{i,j}(c_1 \Delta) \geq \bar{\epsilon}  \big) < 0.
\end{equation}
Taking $\Delta$ to be such that $c_1\Delta\leq \bar{\epsilon}/2$, it suffices to prove that
\begin{equation}
\lsup{N}N^{-1}\log  \mathbb{P}\big(\mathcal{J}_N \text{ and }N^{-1}\sum_{j\in I_N, i\in I_M} Y_b^{i,j}(c_1 \Delta) - c_1\Delta M \geq \bar{\epsilon}/2  \big) < 0.
\end{equation}
Since the $\lbrace Y^{i,j} \rbrace_{i\in I_M, j\in I_N}$ are independent, and $\mathbb{E}\big[ N^{-1}\sum_{j\in I_N, i\in I_M} Y_b^{i,j}(c_1 \Delta )\big] = Mc_1\Delta$, Sanov's Theorem implies \eqref{eq: Sanov to show} \cite{Dembo1998}.
\end{proof}

We now prove Lemma \ref{Lemma Last Lemma 1}.
\begin{proof}
Let $\eta_b \in \mathcal{P} $ be the law of the same stochastic process as $\xi_b(\boldsymbol\sigma,\mathbf{G})$, except that the law of the initial value at time $t^{(n)}_b$ is given by $\Psi_b \cdot \hat{\mu}^N$ rather than the empirical measure. More precisely, writing $\Psi_b\cdot \hat{\mu}^N$ to be the law of random variables $(\boldsymbol\alpha_b,\mathbf{x}_b)$, define $\eta_b$ to be the law of $(\boldsymbol\beta_{\Delta + t^{(n)}_b} ,\mathbf{z}_{\Delta+ t^{(n)}_b})$, where, writing $A\cdot x = (-1)^x$, for $p\in I_M$, for $t\geq t^{(n)}_b$,
\begin{align}
\beta^p_{t} =& \alpha^p_b A\cdot \tilde{Y}^p\big( (t-t^{(n)}_b) c(\alpha^p_b,x^p_b) \big) \label{eq: hybrid 27} \\
\mathbf{z}_{t} =& \mathbf{x}_b +(t-t^{(n)}_b)\mathbf{m}^{\hat{\mu}^N_b}(\boldsymbol\alpha_b , \mathbf{x}_b)+\mathbf{D}^{\hat{\mu}^N_b}\tilde{\mathbf{W}}_{t-t^{(n)}_b}\text{ , where }
D^{\hat{\mu}^N_b}_{ij} = 2\sqrt{L^{\hat{\mu}^N_b}_{ii}}\delta(i,j).\label{eq: hybrid 28} 
\end{align} 
Thanks to the fact that $\exp( \mathfrak{u} t^{(n)}_{b} / T + \mathfrak{u}\Delta / (2T)  -\mathfrak{u} ) \geq \exp( \mathfrak{u} t^{(n)}_{b} / T + \mathfrak{u}\Delta / 4T -\mathfrak{u} ) + \exp( \mathfrak{u} t^{(n)}_{b} / T-\mathfrak{u} )\mathfrak{u} \Delta / 4T$, analogously to \eqref{eq: temporary u bound 1} we find that
\begin{multline}
\big\lbrace d_{W}\big( \Psi_{ b+1}\cdot \hat{\mu}^N, \xi_b \big) > \tilde{\epsilon}\exp\big( \mathfrak{u} t^{(n)}_{b} / T + \mathfrak{u}\Delta / 2T -\mathfrak{u} \big) \big\rbrace 
\subseteq \big\lbrace d_{W}(  \eta_b , \xi_b ) > \tilde{\epsilon}\exp\big( \mathfrak{u} t^{(n)}_{b} / T + \mathfrak{u}\Delta / 4T -\mathfrak{u} \big) \big\rbrace\\ \cup
\big\lbrace d_{W}\big( \Psi_{ b+1}\cdot \hat{\mu}^N, \eta_b \big) > \exp\big( \mathfrak{u} t^{(n)}_{b} / T + \mathfrak{u}\Delta / 4T-\mathfrak{u} \big)\tilde{\epsilon}\mathfrak{u}\Delta / 4T\big\rbrace
\end{multline}
Thanks to Lemma \ref{Lemma Many Events}, it thus suffices for us to prove the following three inequalities,
\begin{multline}
\lsup{N}N^{-1}\log \mathbb{P}\big(  \mathcal{J}_N \text{ and }d_{W}( \eta_b, \xi_b ) > \tilde{\epsilon}\exp\big( \mathfrak{u} t^{(n)}_{b} / T + \mathfrak{u}\Delta / 4T -\mathfrak{u} \big)\\  \text{ and }d_W\big( \Psi_{b}\cdot \hat{\mu}^N, \hat{\mu}^N_b\big) \leq \tilde{\epsilon}\exp\big( \mathfrak{u} t^{(n)}_{b} / T -\mathfrak{u} \big) \big) < 0 \text{ and }\label{eq: last lemma 1}
\end{multline}
for some $\epsilon_0 > 0$,
\begin{align}
&\lsup{N}N^{-1}\log \mathbb{P}\big(  \mathcal{J}_N,\sup_{s\in [t^{(n)}_b,t^{(n)}_{b+1}]}d_W(\hat{\mu}^N_s,\hat{\mu}^N_b) > \epsilon_0 \big) < 0 \label{eq: epsilon 0 bound mu N s b fluctuations}\\
&\lsup{N}N^{-1}\log \mathbb{P}\big(  \mathcal{J}_N,\sup_{s\in [t^{(n)}_b,t^{(n)}_{b+1}]}d_W(\hat{\mu}^N_s,\hat{\mu}^N_b) \leq \epsilon_0,d_{W}\big( \eta_b, \Psi_{ b+1}\cdot \hat{\mu}^N\big) > \frac{\tilde{\epsilon} \mathfrak{u} \Delta}{4T}\exp( \mathfrak{u} t^{(n)}_{b} / T + \mathfrak{u}\Delta / 4T -\mathfrak{u} ) \big)< 0. \label{eq: last lemma 2}
\end{align}
It has already been proved in Lemma \ref{Lemma bound fluctuations} that for any $\epsilon_0$, for all large enough $n$ \eqref{eq: epsilon 0 bound mu N s b fluctuations} must hold. 
\vspace{0.4cm}

\emph{Proof of \eqref{eq: last lemma 2}.} 

\vspace{0.4cm}
We compare the stochastic processes \eqref{eq: hybrid 19}-\eqref{eq: hybrid 20} whose law is $\Psi_{b+1}\cdot \hat{\mu}^N$ to the stochastic processes  \eqref{eq: hybrid 27}-\eqref{eq: hybrid 28} whose law is $\eta_b$. Notice that these processes have the same initial condition at time $t^{(n)}_b$. Using Ito's Lemma, for $t \geq t^{(n)}_b$,
\begin{multline}\label{eq: Ito eta b comparison}
\norm{\mathbf{x}_{t}-\mathbf{z}_t}^2 = \int_{t^{(n)}_b}^t\big\lbrace 2\big\langle \mathbf{x}_s -\mathbf{z}_s , \mathbf{m}^{\hat{\mu}^N_s}(\boldsymbol\alpha_s , \mathbf{x}_s) -  \mathbf{m}^{\hat{\mu}^N_b}(\boldsymbol\alpha_b , \mathbf{x}_b) \big\rangle +  4 \sum_{p\in I_M} ( D^{\hat{\mu}^N_s}_{pp}-D^{\hat{\mu}^N_b}_{pp})^2 \big\rbrace ds\\+2\int_{t^{(n)}_b}^t\big\langle \mathbf{x}_s -\mathbf{z}_s , (\mathbf{D}^{\hat{\mu}^N_s}-\mathbf{D}^{\hat{\mu}^N_b})d\tilde{\mathbf{W}}_{s}\big\rangle
\end{multline}
Analogously to the bound in \eqref{eq: upperbound eta t expect}, one easily establishes the following uniform bound for the moments
\begin{equation}\label{eq: unifrom bound moments breve C}
\sup_{t\in [t^{(n)}_b , t^{(n)}_{b+1}]}\big\lbrace \norm{\mathbf{x}_t}^2 , \norm{\mathbf{z}_t}^2 \big\rbrace \leq \breve{C}
\end{equation}
for some constant $\breve{C}$. Using the Lipschitz inequality for $\mathbf{m}$ in Lemma \ref{Lemma Lipschitz Matrices}, and making use of the uniform bound in \eqref{eq: unifrom bound moments breve C}, there exists a constant $\bar{C}$ such that for all $s \in [t^{(n)}_b , t^{(n)}_{b+1}]$,
\begin{align*}
\norm{\mathbf{m}^{\hat{\mu}^N_s}(\boldsymbol\alpha_s , \mathbf{x}_s) -  \mathbf{m}^{\hat{\mu}^N_b}(\boldsymbol\alpha_b , \mathbf{x}_b)} \leq  \bar{C}\big( \norm{\mathbf{x}_s - \mathbf{x}_b} + \norm{\boldsymbol\alpha_s - \boldsymbol\alpha_b} \big).
\end{align*}
Taking expectations of both sides of \eqref{eq: Ito eta b comparison}, employing the Cauchy-Schwarz Inequality, and assuming that $\sup_{s\in [t^{(n)}_b,t^{(n)}_{b+1}]}d_W(\hat{\mu}^N_s,\hat{\mu}^N_b) \leq \epsilon_0$, we obtain that
\begin{multline}\label{eq: Ito eta b comparison 2}
\mathbb{E}\big[\norm{\mathbf{x}_{t}-\mathbf{z}_t}^2\big] \leq 2\grave{C} \int_{t^{(n)}_b}^t\big( \mathbb{E}\big[\norm{\mathbf{x}_{s}-\mathbf{z}_s}^2\big] +\mathbb{E}\big[\norm{\mathbf{x}_{s}-\mathbf{z}_s}^2\big]^{1/2} \mathbb{E}\big[\norm{\boldsymbol\alpha_s - \boldsymbol\alpha_b}^2\big]^{1/2}  \big) ds + 4(t-t^{(n)}_b)\epsilon_0.
\end{multline}
Properties of the Poisson Process (see for example Lemma \ref{Lemma Upper Bound Jump Rate}) dictate that $\mathbb{E}\big[\norm{\boldsymbol\alpha_s - \boldsymbol\alpha_b}^2\big] \leq 4 c_1\Delta$, as long as $s-t^{(n)}_b \leq \Delta$. Thus for all $t$ such that $\mathbb{E}\big[\norm{\mathbf{x}_{t}-\mathbf{z}_t}^2\big] \leq \Delta$, it must hold that
\[
\mathbb{E}\big[\norm{\mathbf{x}_{t}-\mathbf{z}_t}^2\big] \leq  2\grave{C} \int_{t^{(n)}_b}^t\big( \mathbb{E}\big[\norm{\mathbf{x}_{s}-\mathbf{z}_s}^2\big]ds+(t-t^{(n)}_b)\big\lbrace \epsilon_0 +4\Delta\grave{C}\sqrt{c_1} \big\rbrace .
\]
We thus find from Gronwall's Inequality that for any $\bar{\epsilon} > 0$, through choosing $\epsilon_0$ to be sufficiently small, and $n$ to be sufficiently large,
\begin{equation}\label{eq: intermediate eta b comparsion}
\sup_{t\in [t^{(n)}_b , t^{(n)}_{b+1}]}\mathbb{E}\big[\norm{\mathbf{x}_{t}-\mathbf{z}_t}^2\big] \leq \Delta \bar{\epsilon}.
\end{equation}
Using the compensated Poisson Process representation, we obtain that
\begin{align}
\mathbb{E}\big[ \norm{\boldsymbol\alpha_{t} - \boldsymbol\beta_t}^2\big] &\leq 4\sum_{p\in I_M}\mathbb{E}\big[\big|\tilde{Y}^p\big( (t-t^{(n)}_b) c(\alpha^p_b,z^p_b)-\tilde{Y}^p\big( \int_{t^{(n)}_b}^{t} c(\alpha^p_s,x^p_s)ds  \big)\big|\big]\nonumber \\
&\leq 4\sum_{p\in I_M}\mathbb{E}\big[ \int_{t^{(n)}_b}^{t} \big| c(\alpha^p_b,z^p_b) -c(\alpha^p_s,x^p_s) \big| ds \big]\nonumber \\
&\leq 4M(c_1+c_L)\Delta\sup_{s\in [t^{(n)}_b , t^{(n)}_{b+1}]}\mathbb{E}\big[\norm{\mathbf{x}_{s}-\mathbf{x}_b} + \norm{\boldsymbol\alpha_s - \boldsymbol\alpha_b}\big],\label{eq: norm alpha beta difference expectation}
\end{align}
using the fact that $c(\cdot,\cdot)$ is Lipschitz and bounded. Since the expectation in the last term goes to zero as $\Delta\to 0$, it follows from \eqref{eq: intermediate eta b comparsion} and \eqref{eq: norm alpha beta difference expectation} that for sufficiently large $\Delta$,
\[
d_{W}\big( \eta_b, \Psi_{ b+1}\cdot \hat{\mu}^N\big) \leq \frac{\tilde{\epsilon} \mathfrak{u} \Delta}{4T}\exp( \mathfrak{u} t^{(n)}_{b} / T + \mathfrak{u}\Delta / 4T -\mathfrak{u} ).
\]
We have thus established \eqref{eq: last lemma 2} and it remains to prove \eqref{eq: last lemma 1}. Suppose that $d_W\big( \Psi_{b}\cdot \hat{\mu}^N, \hat{\mu}^N_b\big) \leq \tilde{\epsilon}\exp\big( \mathfrak{u} t^{(n)}_{b} / T -\mathfrak{u} \big)$. The definition of the Wasserstein distance implies that for any $\delta > 0$, there must exist a common probability space supporting the random variables $(\boldsymbol\zeta,\mathbf{x},\boldsymbol\beta,\mathbf{z})$, with $\hat{\mu}^N_b$ the law of $(\boldsymbol\zeta,\mathbf{x})$, and $\Psi_b \cdot \hat{\mu}^N$ the law of $(\boldsymbol\beta,\mathbf{z})$, and such that
\begin{equation}
\mathbb{E}\big[ \norm{\boldsymbol\zeta-\boldsymbol\beta} + \norm{\mathbf{x }- \mathbf{z}} \big] \leq \delta + \tilde{\epsilon}\exp\big( \mathfrak{u} t^{(n)}_{b} / T -\mathfrak{u} \big).\label{eq: combine result 1}
\end{equation}
We append the mutually independent Poisson processes $\lbrace \tilde{Y}^p(t) \rbrace_{p\in I_M}$ and Brownian motions $\lbrace \tilde{W}^p_t \rbrace_{p\in I_M}$ to this same space, and define $(\boldsymbol\zeta_{\Delta},\mathbf{x}_{\Delta})$ to satisfy \eqref{eq: hybrid 17} -\eqref{eq: hybrid 18}  and $(\boldsymbol\beta_{\Delta}, \mathbf{z}_{\Delta})$ to satisfy \eqref{eq: hybrid 27}-\eqref{eq: hybrid 28}. We then observe using the triangle inequality that
\begin{align}
\mathbb{E}\big[ \norm{\boldsymbol\zeta_{\Delta}-\boldsymbol\beta_{\Delta}} + \norm{\mathbf{x}_{\Delta}- \mathbf{z}_{\Delta}} \big]  &\leq\mathbb{E}\big[ \norm{\boldsymbol\zeta-\boldsymbol\beta} + \norm{\mathbf{x }- \mathbf{z}} + \norm{\boldsymbol\zeta_{\Delta}-\boldsymbol\zeta + \boldsymbol\beta -\boldsymbol\beta_{\Delta}} +\norm{\mathbf{x}_{\Delta}-\mathbf{x} + \mathbf{z}- \mathbf{z}_{\Delta}}  \big] \nonumber \\
&\leq \delta + \tilde{\epsilon}\exp\big( \mathfrak{u} t^{(n)}_{b} / T -\mathfrak{u} \big) + \mathbb{E}\big[ \norm{\boldsymbol\zeta_{\Delta}-\boldsymbol\zeta + \boldsymbol\beta -\boldsymbol\beta_{\Delta}} +\norm{\mathbf{x}_{\Delta}-\mathbf{x} + \mathbf{z}- \mathbf{z}_{\Delta}}  \big]  \text{ and }\label{eq: combine result 2}
\end{align}
\begin{align}
\mathbb{E}\big[  \norm{\boldsymbol\zeta_{\Delta}-\boldsymbol\zeta + \boldsymbol\beta -\boldsymbol\beta_{\Delta}} \big|\; \boldsymbol\zeta , \boldsymbol\beta, \mathbf{x},\mathbf{g} \big] \leq & 2\sum_{p\in I_M} \mathbb{E}\big[  \big| \tilde{Y}^p\big( \Delta c(\zeta^p,x^p) \big) - \tilde{Y}^p\big(\Delta c(\beta^p,z^p) \big) \big|  \;\; \big|\; \boldsymbol\zeta , \boldsymbol\beta, \mathbf{x},\mathbf{g} \big] \label{eq: intermediate almost finished zeta delta}
\end{align}
Define $v_p = \inf\big\lbrace c(\beta^p,z^p),c(\zeta^p,x^p) \big\rbrace$ and let $\lbrace \breve{Y}^p,\hat{Y}^p, \grave{Y}^p\rbrace_{p\in I_M}$ be independent Poisson Processes. Using the additive property of Poisson Processes \cite{Ethier1986}, we have the following representation
\begin{align*}
\tilde{Y}^p\big( \Delta c(\zeta^p,x^p) \big) = \breve{Y}^p(v_p) + \hat{Y}^p\big(\Delta [c(\zeta^p,x^p)-v_p]_+ \big)\\
 \tilde{Y}^p\big(\Delta c(\beta^p,z^p) \big) = \breve{Y}^p(v_p)+ \grave{Y}^p\big(\Delta [c(\beta^p,z^p)-v_p]_+ \big).
\end{align*}
Hence \eqref{eq: intermediate almost finished zeta delta} implies that
\begin{align*}
\mathbb{E}\big[  \norm{\boldsymbol\zeta_{\Delta}-\boldsymbol\zeta + \boldsymbol\beta -\boldsymbol\beta_{\Delta}} \big|\; \boldsymbol\zeta , \boldsymbol\beta, \mathbf{x},\mathbf{g} \big] &\leq  2\sum_{p\in I_M} \mathbb{E}\big[   \hat{Y}^p\big( \Delta [c(\zeta^p,x^p)-v_p]_+ \big) +  \grave{Y}^p\big(\Delta [c(\beta^p,z^p)-v_p]_+ \big) \;\; \big|\; \boldsymbol\zeta , \boldsymbol\beta, \mathbf{x},\mathbf{g} \big] \\
&= 2\Delta\sum_{p\in I_M}\big| c(\zeta^p,x^p)-c(\beta^p,z^p)\big|\\
&\leq 2M\Delta \sup_{p\in I_M} \big\lbrace c_1|\zeta^p - \beta^p | +c_L | x^p - g^p | \big\rbrace 
\end{align*}
where $c_1$ is the uniform upperbound for the jump rate, and $c_L$ is the Lipschitz constant for $c$. Taking expectations of both sides, one finds that there exists a constant $\bar{C} > 0$ such that
\begin{equation}
\mathbb{E}\big[  \norm{\boldsymbol\zeta_{\Delta}-\boldsymbol\zeta + \boldsymbol\beta -\boldsymbol\beta_{\Delta}} \big] \leq \bar{C} \Delta \mathbb{E}\big[ \norm{\boldsymbol\zeta-\boldsymbol\beta} + \norm{\mathbf{x }- \mathbf{z}} \big].\label{eq: combine result 3}
\end{equation}
We analogously find that for a constant $C> 0$,
\begin{equation}
\mathbb{E}\big[\norm{\mathbf{x}_{\Delta}-\mathbf{x} + \mathbf{z}- \mathbf{z}_{\Delta}}  \big] \leq  C \Delta \mathbb{E}\big[ \norm{\boldsymbol\zeta-\boldsymbol\beta} +\norm{\mathbf{x} - \mathbf{z}} \big],\label{eq: combine result 4}
\end{equation}
since the coefficients $\mathbf{m}$ and $\mathbf{L}$ are Lipschitz, as noted in Lemma \ref{Lemma Lipschitz Matrices}. The above results \eqref{eq: combine result 1}-\eqref{eq: combine result 4} imply that there exists a constant $\hat{C} > 0$ such that
\begin{equation}
d_{W}( \eta_b, \xi_b ) \leq d_W(\hat{\mu}^N_b, \Psi_b\cdot \hat{\mu}^N) \big\lbrace 1 + \hat{C}\Delta \big\rbrace
\end{equation}
Thus as long as $\mathfrak{u}/4T > \hat{C}$, if $d_W(\hat{\mu}^N_b, \Psi_b\cdot \hat{\mu}^N) \leq  \tilde{\epsilon}\exp\big( \mathfrak{u} t^{(n)}_{b} / T -\mathfrak{u} \big)$, it must be that $d_{W}( \eta_b, \xi_b ) \leq \tilde{\epsilon}\exp\big( \mathfrak{u} t^{(n)}_{b} / T + \mathfrak{u}\Delta / 4T -\mathfrak{u} \big)$, which establishes \eqref{eq: last lemma 1}.

\end{proof}

\section{Change of Measure}\label{Section Change of Measure}
It remains for us to prove Lemma \ref{Lemma Main Lemma}. To do this, we must `separate' the effects of the stochasticity and the disorder on the dynamics by defining new processes $\tilde{\boldsymbol\sigma}_{i,t}$ (with $i$ belonging to an index set that grows polynomially in $N$) that are such that the spin-flipping is independent of the connections. However it will be seen that $\tilde{\boldsymbol\sigma}_{i,t}$ is an excellent approximation to the old process, as long as the empirical process lies in a small subset $\mathcal{V}^N_i$ of $\mathcal{M}^+_1\big(\mathcal{D}([0,T],\mathcal{E}^M\times\mathbb{R}^M)\big)$. The number of such subsets $\lbrace \mathcal{V}^N_i \rbrace$ is polynomial in $N$: this polynomial growth will be dominated by the exponential decay of the probability bounds of subsequent sections. The fact that the new processes are independent of the connections will allow us to use a conditional Gaussian measure to accurately infer the evolution of the fields over a small time step (in Section \ref{Section Field}). In order that we may employ Girsanov's Theorem, it is essential that the processes $\tilde{\boldsymbol\sigma}_{i,t}$ are adapted to the filtration $\mathcal{F}_t$ as well. The main result of this section is Lemma \ref{Lemma Change of Measure}: this lemma gives a sufficient condition in terms of the new processes $\tilde{\boldsymbol\sigma}_{i,t}$ for the condition of Lemma \ref{Lemma Main Lemma} to be satisfied.

\subsection{Partition of the Probability Space}

Define the pathwise empirical measure
\begin{equation}\label{eq: pathwise empirical measure}
\tilde{\mu}^N = N^{-1}\sum_{j\in I_N}\delta_{(\boldsymbol\sigma^j , \mathbf{G}^j)} \in \mathcal{M}^+_1\big( \mathcal{D}([0,T] , \mathcal{E}^M \times \mathbb{R}^M) \big).
\end{equation}
The pathwise empirical measure will be used to partition the probability space. Before we partition $\mathcal{M}^+_1\big(\mathcal{D}([0,T],\mathcal{E}^M\times\mathbb{R}^M)\big)$, we must first partition the underlying state space $\mathcal{E}^{M} \times \mathbb{R}^M$. For some positive integer $\mathfrak{n}$, define the sets $\lbrace D_i \rbrace_{0\leq i \leq \mathfrak{n}^2+1}\subset \mathbb{R}$ as follows.
\begin{align}
D_0 &= (-\infty , -\mathfrak{n}] \; \; , \; \; D_{\mathfrak{n}^2+1} = (\mathfrak{n}, \infty)\\
D_i &= (-\mathfrak{n} + 2(i-1)\mathfrak{n}^{-1} , -\mathfrak{n} + 2i\mathfrak{n}^{-1}] \text{ for }1\leq i \leq \mathfrak{n}^2.
\end{align}
Next, let $\lbrace \tilde{D}_i \rbrace_{1\leq i \leq C_\mathfrak{n}} \subset \mathbb{R}^{M}$ be such that for each $i$,
\begin{align}
\tilde{D}_i = D_{p^i_1} \times D_{p^i_2} \times\ldots D_{p^i_M},
\end{align}
for integers $\lbrace p^i_j \rbrace$. The sets are defined to be such that
\begin{align}
\mathbb{R}^M = \bigcup_{i=0}^{C_{\mathfrak{n}}} \tilde{D}_i \text{ and }
\tilde{D}_i \cap \tilde{D}_j = \emptyset \text{ if }i\neq j. \label{eq: R M partition}
\end{align}
Next we partition the path space 
\begin{equation}\label{eq: path space partition}
\mathcal{D}([0,T] , \mathcal{E}^M \times \mathbb{R}^M) = \bigcup_{i=1}^{\hat{C}_{\mathfrak{n}}} \hat{D}_{i}, 
\end{equation}
where $\lbrace \hat{D}_i \rbrace$ are defined as follows. In constructing this partition, we require a more refined partition of the time interval $[0,T]$ into $(m+1)$ time points $\lbrace t^{(m)}_a \rbrace_{0\leq a \leq m}$: this is necessary for us to be able to control the Girsanov Exponent in the next section. It is assumed that $m$ is an integer multiple of $n$ (the integer dictating the number of time points in the previous section). Throughout this section, unless specified otherwise, for $0\leq a \leq m$, we write $\boldsymbol\sigma_a := \boldsymbol\sigma_{t^{(m)}_a}$. Each $\hat{D}_i  \subset \mathcal{D}([0,T], \mathcal{E}^M \times \mathbb{R}^M)$ is nonempty, and of the form
\begin{equation}
\hat{D}_i = \big\lbrace \boldsymbol\alpha_{[0,T]} \times \mathbf{g}_{[0,T]} : \mathbf{g}_a \in \tilde{D}_{r^i_a}\text{ and }\boldsymbol\alpha_a \in \tilde{D}_{q^i_a} \text{ for each }0\leq a \leq m \big\rbrace,
\end{equation}
for indices $\lbrace q^i_a, r^i_a\rbrace_{0\leq a \leq m}$, $1\leq q_a , r_a \leq C_{\mathfrak{n}}$. The indices are chosen such that (i) $\hat{D}_i \cap \hat{D}_j = \emptyset$ if $i\neq j$, (ii) $\hat{D}_i \neq \emptyset$ and (iii) \eqref{eq: path space partition} is satisfied. Let
\begin{equation} \label{eq: hat W 2}
\hat{\mathcal{W}}_2 = \big\lbrace \mu \in  \mathcal{M}^+_1\big( \mathcal{D}([0,T] , \mathcal{E}^M \times \mathbb{R}^M) \big) : \sup_{p\in I_M} \sup_{t\in [0,T]}\mathbb{E}^{\mu}[ (g^p_t)^2] \leq 3 \big\rbrace.
\end{equation}
Next, for a positive integer $C^N_{\mathfrak{n}}$, make the partition
 \begin{align}
\tilde{\mathcal{W}}_2 = \bigcup_{i=1}^{C^N_{\mathfrak{n}}} \mathcal{V}^N_{i} \label{eq: W 2 partition}
 \end{align}
 where each $\mathcal{V}^N_i$ is such that $\mu \in \mathcal{V}^N_i$ if and only if (i) $\mu \in \hat{\mathcal{W}}_2$ and (ii) for all $1\leq q,r \leq \hat{C}_{\mathfrak{n}}$,
  \begin{align}
 \mu( \boldsymbol\sigma \in \hat{D}_{q} \text{ and } \mathbf{g} \in \hat{D}_{r} ) &\in [ \hat{u}^N_{i,qr} - 1/(2N), \hat{u}^N_{i,qr} + 1/(2N)) \text{ for numbers} \label{eq: V N i mass}\\
 \hat{u}^N_{i,qr} &\in \lbrace 0, N^{-1},2N^{-1},\ldots , 1-N^{-1}, 1 \rbrace.
 \end{align}
 It is assumed that the indices are chosen such that (i) $\mathcal{V}^N_i \neq \emptyset$ and (ii) the partition is disjoint, i.e. $\mathcal{V}^N_i \cap \mathcal{V}^N_j = \emptyset$ if $i\neq j$. The motivation for the scaling of $N^{-1}$ for the mass of each set in \eqref{eq: V N i mass} is that if $\tilde{\mu}^N \in \mathcal{V}^N_j$, then we will know the precise mass assigned to each set, since the empirical process can only assign a mass that is an integer multiple of $N^{-1}$ to each set.

We next prove that the radius of the sets in the partition goes to zero uniformly, in the following sense. 
 \begin{lemma}\label{Lemma Uniform Partition Wasserstein}
 Define
\[
\mathfrak{U} = \big\lbrace f :\mathcal{E}^{M(m+1)} \times \mathbb{R}^{M(m+1)} \to\mathbb{R} \; ; \; |f(\boldsymbol\alpha , \mathbf{x}) - f(\boldsymbol\beta,\mathbf{g})| \leq \sum_{q\in I_M}\sum_{a=0}^m\big\lbrace | \alpha^q_a - \beta^q_a | + |x^q_a - g^q_a| \big\rbrace \text{ and }f(\cdot,\mathbf{0}) = 0\big\rbrace.
\]
 For $f\in \mathfrak{U}$, write $\hat{f} : \mathcal{D}([0,T],\mathcal{E}^M \times\mathbb{R}^M) \to \mathbb{R}$ to be $\hat{f}(\boldsymbol\alpha,\mathbf{x}) := f\big( (\alpha^q_{t^{(m)}_a})_{0\leq a \leq m, q\in I_M} , (x^q_{t^{(m)}_a})_{0\leq a \leq m, q\in I_M} \big)$. We find that for any $m\geq 1$,
 \begin{equation}
 \lim_{\mathfrak{n}\to \infty}\lsup{N}\sup_{1\leq i \leq C^N_{\mathfrak{n}}}\sup_{\mu,\nu \in \mathcal{V}^N_i}\sup_{f\in \mathfrak{U}} \big| \mathbb{E}^{\mu}[\hat{f}] - \mathbb{E}^{\nu}[\hat{f}] \big| = 0.
 \end{equation}
 \end{lemma}
 \begin{proof}
 We notice that for any $1\leq i \leq C^N_{\mathfrak{n}}$ and any $\mu \in \mathcal{V}^N_i$,
 \begin{align*}
 \mathbb{E}^{\mu(\boldsymbol\alpha,\mathbf{x})}\big[ \hat{f}(\boldsymbol\alpha,\mathbf{x}) \chi\big\lbrace \sup_{q\in I_M , 0\leq a \leq m}|x^q_a| \geq \mathfrak{n} \big\rbrace \big]
\leq \mathfrak{n}^{-1}\mathbb{E}^{\mu(\boldsymbol\alpha,\mathbf{x})}\big[\sum_{a=0}^m \norm{\mathbf{x}_a}^2 \chi\big\lbrace \sup_{q\in I_M , 0\leq a \leq m}|x^q_a| \geq \mathfrak{n} \big\rbrace \big] \leq 3(m+1)\mathfrak{n}^{-1},
 \end{align*}
 using the fact that $\mathcal{V}^N_i \subset \hat{\mathcal{W}}_2$ (as defined in \eqref{eq: hat W 2}). Thus the mass assigned to non-bounded sets goes to zero uniformly as $\mathfrak{n} \to 0$. Furthermore it can be seen from the definition in \eqref{eq: R M partition} that the radius of the bounded sets goes to zero uniformly as $\mathfrak{n}\to\infty$, which implies the lemma.
 \end{proof}
 Next we observe that the number of sets in the partition is subexponential in $N$: this is an essential property, because it means that the partition size is dominated by the exponential decay of the probabilities in coming sections.
\begin{lemma}\label{eq: polynomial partition}
For any $\mathfrak{n} \in \mathbb{Z}^+$,
\begin{equation}
\lsup{N}N^{-1}\log C^N_{\mathfrak{n}} = 0
\end{equation}
\end{lemma}
\begin{proof}
We notice from \eqref{eq: V N i mass} that each $\mathcal{V}^N_i$ can assign $(N+1)$ possible values to the mass of each set $\hat{D}_q \times \hat{D}_r \in \mathcal{E}^M \times \mathbb{R}^M$. Since there are $\tilde{C}_{\mathfrak{n}}^2$ such sets, the number of such $\mathcal{V}^N_i$ must be upperbounded by $(N+1)^{\tilde{C}^2_{\mathfrak{n}}}|$. Since this is polynomial in $N$, we have established the lemma.
\end{proof}
\subsubsection{Definition of the Approximating Process}

We are now in a position to define the adapted stochastic process $\tilde{\boldsymbol\sigma}_i$ (for each $1\leq i \leq C^N_{\mathfrak{n}}$), written $\tilde{\boldsymbol\sigma}_i := (\tilde{\sigma}^{q,j}_{i,t})_{q\in I_M, j\in I_N, t\in [0,T]}$.  Write $\tilde{\mathcal{V}}^N_{i,t} \subset \mathcal{M}^+_1\big(\mathcal{D}([0,t],\mathcal{E}^M)\big)$ to be the projection of the probability measures in $ \mathcal{M}^+_1\big(\mathcal{D}([0,T],\mathcal{E}^M \times \mathbb{R}^M)\big)$ onto their marginals over $\mathcal{D}([0,t],\mathcal{E}^M)$ - and define $\mathcal{V}^N_{i,t}$ to be the analogous projection onto the marginal over $\mathcal{D}([0,t],\mathcal{E}^M \times\mathbb{R}^M)$. We write the intensity of $\tilde{\sigma}^{q,j}_{i,t}$ as $\tilde{\mathfrak{G}}^{q,j}_{i,t}$. We will choose the intensities to be such that as long as $\tilde{\mu}^N_{[0,t]}(\tilde{\boldsymbol\sigma}) := N^{-1}\sum_{j\in I_N}\delta_{\boldsymbol\sigma^j_{[0,t]}} \in  \tilde{\mathcal{V}}^N_{i,t}$, then necessarily $\tilde{\mu}^N_{[0,t]}(\tilde{\boldsymbol\sigma}_i,\tilde{\mathfrak{G}}_i) \in \mathcal{V}^N_{i,t}$. This property is essential for us to be able to control the Girsanov Exponent in the next section.

We first find any set of paths $\boldsymbol\alpha_i$ and intensities $\mathfrak{G}_i$ that are such that their empirical process is in $\mathcal{V}^N_i$.
\begin{lemma}
For each $1\leq i \leq C^N_{\mathfrak{n}}$, there exists $\boldsymbol\alpha_i \in \mathcal{D}\big([0,T],\mathcal{E}^M\big)^N$ and $\mathfrak{G}_i \in \mathcal{D}\big([0,T],\mathbb{R}^M\big)^N$ such that
\begin{align}
\tilde{\mu}^N(\boldsymbol\alpha_i,\mathfrak{G}_i) :=& N^{-1}\sum_{j\in I_N}\delta_{(\boldsymbol\alpha_i^j,\mathfrak{G}_i^j)}\in \mathcal{V}^N_i \text{ where }\boldsymbol\alpha_i = (\boldsymbol\alpha_i^j)_{j\in I_N}\text{, }\mathfrak{G}_i = (\mathfrak{G}_i^j)_{j\in I_N} \text{ and }\\
\mathfrak{G}_{i,t} =& \mathfrak{G}_{i,t_a^{(m)}} \text{ for all }t\in [t^{(m)}_a,t^{(m)}_{a+1}) \\
\boldsymbol\alpha_{i,t} =&\boldsymbol\alpha_{i,t_a^{(m)}}  \text{ for all }t\in [t^{(m)}_a,t^{(m)}_{a+1}). 
\end{align}
\end{lemma}
\begin{proof}
Let $\breve{\pi}:  \mathcal{M}^+_1\big( \mathcal{D}([0,T],\mathcal{E}^M \times \mathbb{R}^M)\big) \to \mathcal{M}^+_1\big( \mathcal{E}^{MN(m+1)} \times \mathbb{R}^{MN(m+1)}\big)$ be the projection of a measure onto its marginal at times $\lbrace t^{(m)}_a \rbrace_{0\leq a \leq m}$. Because empirical measures are dense in $ \mathcal{M}^+_1\big( \mathcal{E}^{MN(m+1)} \times \mathbb{R}^{MN(m+1)}\big)$, for all large enough $N$, there must exist $\tilde{\boldsymbol\alpha}_i \in \mathcal{E}^{MN(m+1)}$, written $\tilde{\boldsymbol\alpha}_i := (\tilde{\boldsymbol\alpha}_{i,a})_{0\leq a \leq m}$, and $\tilde{\mathfrak{G}}_i \in \mathbb{R}^{MN(m+1)}$, written $\tilde{\mathfrak{G}}_i := (\tilde{\mathfrak{G}}_{i,a})_{0\leq a \leq m}$ such that
\begin{equation}
\hat{\mu}^N(\tilde{\boldsymbol\alpha}_i,\tilde{\mathfrak{G}}_i) := N^{-1}\sum_{j\in I_N}\delta_{(\tilde{\boldsymbol\alpha}_i^j, \tilde{\mathfrak{G}}^j_i)}  \in \breve{\pi} \cdot \mathcal{V}^N_i.
\end{equation}
We can now define $\boldsymbol\alpha_i := (\boldsymbol\alpha_{i,t})_{t\in [0,T]} \in \mathcal{D}\big([0,T],\mathcal{E}^{M}\big)^N$ and $\mathfrak{G}_{i} := (\mathfrak{G}_{i,t})_{t\in [0,T]} \in  \mathcal{D}\big([0,T],\mathcal{E}^M\big)^N$ as follows: for each $0\leq a \leq m$,
\begin{align*}
\boldsymbol\alpha_{i,t^{(m)}_a} :=& \tilde{\boldsymbol\alpha}_{i,a}, \; \; \; \text{ and }\; \; \;
\boldsymbol\alpha_{i,t} =\boldsymbol\alpha_{i,t_a^{(m)}} \text{ for }t\in [t^{(m)}_a , t^{(m)}_{a+1})\\
\mathfrak{G}_{i,t^{(m)}_a} :=& \tilde{\mathfrak{G}}_{i,a}, \; \; \; \text{ and }\; \; \;
\mathfrak{G}_{i,t} = \mathfrak{G}_{i,t_a^{(m)}} \text{ for all }t\in [t^{(m)}_a,t^{(m)}_{a+1}).
\end{align*}
\end{proof}
Next, we prove that if $\tilde{\mu}^N_{[0,t]}(\tilde{\boldsymbol\sigma}) \in \tilde{\mathcal{V}}^N_{i,t}$, then we must be able to find a permutation of the intensities $\lbrace \mathfrak{G}_{i,t}^j \rbrace$ that ensures that the associated empirical process is in $\mathcal{V}^N_{i}$. Define $\mathfrak{P}^N$ to be the set of all permutations on $I_N$ (i.e. each member of $\mathfrak{P}^N$ is a bijective map $I_N \to I_N$).
\begin{lemma}\label{Lemma definition tilde pi permutation}
 For any $\tilde{\boldsymbol\sigma}_i \in \mathcal{D}([0,T] , \mathcal{E}^M)^N$ and any $t < \tilde{\tau}_i$, define $\pi_{t,\tilde{\boldsymbol\sigma}_i} \in \mathfrak{P}^N$ to be such that
\begin{align}
\tilde{\sigma}^{q,j}_{i,t^{(m)}_a} &= \alpha^{q,\pi_{t,\tilde{\boldsymbol\sigma}_i}(j)}_{i,t^{(m)}_a} \text{ for all }t^{(m)}_a < \tilde{\tau}_i \text{ and }\label{eq: pi permutation definition}\\
\pi_t &= \pi_{t^{(m)}_a} \text{ for all }t\in [t^{(m)}_a,t^{(m)}_{a+1}).
\end{align}
$\pi_{t,\tilde{\boldsymbol\sigma}}$ is well-defined, but not uniquely defined. Furthermore $\pi_{\cdot,\cdot}: \mathcal{D}([0,T],\mathcal{E}^M)^N \times [0,T] \to \mathfrak{P}^N$ is progressively-measurable
\end{lemma}
\begin{proof}
Write $\breve{\boldsymbol\alpha}_{i,t} := \boldsymbol\alpha_{i,t\wedge \tilde{\tau}_i}$. We first claim that $\breve{\pi}\cdot \tilde{\mu}^N(\tilde{\boldsymbol\sigma}_i) = \breve{\pi}\cdot\tilde{\mu}^N(\breve{\boldsymbol\alpha}_{i})$, as long as $t < \tilde{\tau}_i$. This is because $\mathcal{V}^N_i$ specifies the mass of each set to an accuracy of $N^{-1}$, but the mass assigned to any set by the empirical measure must also be a multiple of $N^{-1}$. This means that we must be able to find a permutation such that \eqref{eq: pi permutation definition} is satisfied.
\end{proof}
We can now formally define the stochastic process $\tilde{\boldsymbol\sigma}_i$. First, $\tilde{\sigma}^{q,j}_{i,t}$ is `stopped' once the empirical measure is no longer in $\tilde{\mathcal{V}}^N_{i,t}$, i.e.
\begin{align}
\tilde{\sigma}^{q,j}_{i,t} &:= \tilde{\sigma}^{q,j}_{i,\tilde{\tau}_i} \text{ where }\label{eq: stopped tilde sigma i}\\
\tilde{\tau}_i &= \inf\big\lbrace t \geq 0 : \tilde{\mu}^N_{[0,t]}(\tilde{\boldsymbol\sigma}) \notin \tilde{\mathcal{V}}^N_{i,t} \big\rbrace \text{ and }\tilde{\mu}^N_{[0,t]}(\tilde{\boldsymbol\sigma}) := N^{-1}\sum_{j\in I_N}\delta_{\boldsymbol\sigma^j_{[0,t]}} \in \mathcal{M}^+_1\big( \mathcal{D}([0,t] , \mathcal{E}^M ) \big).\label{eq: tilde tau stopping time}
\end{align}
For all $t\leq \tilde{\tau}_i$, we stipulate that $\tilde{\sigma}^{q,j}_{i,t}$ satisfies the identity,
\begin{equation}\label{eq: tilde sigma q j t definition}
\tilde{\sigma}^{q,j}_{i,t} = \sigma^{q,j}_{0} A \cdot Y^{q,j}\bigg( \int_0^t c(\tilde{\sigma}^{q,j}_{i,s}, \mathfrak{G}_{i,s}^{q,\pi_{s,\tilde{\boldsymbol\sigma}_i}(j)} ) ds \bigg),
\end{equation}
recalling that $A\cdot x$ is defined to be $-1^{x}$. Recall from \eqref{eq: stopped tilde sigma i} that $\tilde{\sigma}^{q,j}_t$ is defined to be stopped for $t\geq \tilde{\tau}_i$.

\begin{lemma}
The stochastic processes $\big\lbrace \tilde{\sigma}^{q,j}_{i,t} \big\rbrace_{j\in I_N,q\in I_M, t\in [0,T]}$ are uniquely well-defined and are adapted to the filtration $\mathcal{F}_t$. Also if $\tilde{\tau}_i > T$, then, writing $\tilde{\mathfrak{G}}^{q,j}_{i,s} := \mathfrak{G}^{q,\pi_{s,\tilde{\boldsymbol\sigma}}(j)}_{i,s}$ and $\tilde{\mathfrak{G}}^{q,j}_{i} = (\tilde{\mathfrak{G}}^{q,j}_{i,s})_{s\in [0,T]}$, it must be that
\begin{equation}
\tilde{\mu}^N(\tilde{\boldsymbol\sigma}_i, \tilde{\mathfrak{G}}_{i}) \in \mathcal{V}^N_i.
\end{equation}
\end{lemma}
\begin{proof}
This is immediate from the definitions.
\end{proof}

\subsection{Girsanov's Theorem}\label{Subsection Girsanov}

In this section we demonstrate that the probability law of the original system $\boldsymbol\sigma_t$ can be well-approximated by the law of one of the processes $\lbrace \tilde{\boldsymbol\sigma}_{i,t} \rbrace_{1\leq i \leq C^N_{\mathfrak{n}}}$. The main result is Lemma \ref{Lemma Change of Measure}: the implication of this lemma is that if we can show that the flow operator accurately describes the dynamics of the empirical processes generated by each of the $\tilde{\boldsymbol\sigma}_i$, then it must accurately describe the original empirical process as well.

Let $R^N_{i} \in \mathcal{M}^+_1 \big(  \mathcal{D}\big( [ 0 , T] , \mathcal{E}^{M} \big)^N \big)$ be the probability law of the processes $\big\lbrace \tilde{\sigma}^{q,j}_{i,t} \big\rbrace_{j\in I_N,q\in I_M, t\in [0,T]}$.  Define the stopping time $\tau_i$ that is the analog of $\tilde{\tau}_i$ in \eqref{eq: tilde tau stopping time}, i.e.
\begin{equation}
\tau_i = \inf\big\lbrace t\geq 0 : \tilde{\mu}^N_{[0,t]}(\boldsymbol\sigma) \notin \pi_t \cdot \tilde{\mathcal{V}}^N_{i,t} \big\rbrace .
\end{equation}
Notice that, necessarily,
\begin{equation}
\tau_i \in \big\lbrace t^{(m)}_a \big\rbrace_{0\leq a \leq m}. \label{eq: tau i definition}
\end{equation}
Let $P^N_{\mathbf{J}} \in  \mathcal{M}^+_1 \big(  \mathcal{D}\big( [ 0 , T] , \mathcal{E}^{M} \big)^N \big)$ be the law of the original spin system $\big\lbrace \sigma^{q,j}_{i,t \wedge \tau_i \wedge T} \big\rbrace_{j\in I_N,q\in I_M, t\in [0,T]}$, conditioned on a realization of the connections $\mathbf{J}$, and stopped at time $\tau_i$. Write
\begin{equation}
\hat{\mathfrak{G}}^{q,j}_{i,s} := \mathfrak{G}^{q,\pi_{s,\boldsymbol\sigma}(j)}_{i,s},
\end{equation}
where $\pi_{\cdot,\cdot}$ is defined in Lemma \ref{Lemma definition tilde pi permutation}. Define the Girsanov exponent
 \begin{multline}
 \Gamma^N_{i}\big(\boldsymbol\sigma_{[0, T]},\mathbf{J}\big) = N^{-1}\sum_{q \in I_M , j\in I_N} \bigg\lbrace \int_{0}^{\tau_i \wedge T}\big\lbrace c(\sigma^{q,j}_s , \hat{\mathfrak{G}}^{q,j}_{i,s}) -c(\sigma^{q,j}_{s},G^{q,j}_{s}) \big\rbrace ds\\ +\int_0^{\tau_i \wedge T}\big\lbrace \log c\big(\sigma^{q,j}_s,G^{q,j}_s \big)-\log c\big(\sigma^{q,j}_{s}, \hat{\mathfrak{G}}^{q,j}_{i,s} \big)\big\rbrace d\hat{\sigma}^{q,j}_s\bigg\rbrace,
 \end{multline}
 and we have defined $\hat{\sigma}^{i,j}_s$ to be the integer-valued nondecreasing c\`adl\`ag process specifying how many times that $\sigma^{i,j}_s$ has changed sign over the time period $[0,s)$, i.e. $\sigma^{i,j}_s = \sigma^{i,j}_0 \times (-1)^{\hat{\sigma}^{i,j}_s}$. It follows from Girsanov's Theorem \footnote{A quick way to see why this formula holds is to note that the probability of a jump occurring over a small time interval is approximately exponentially distributed, i.e. $\mathbb{P}(|\sigma^{i,j}_{\Delta} - \sigma^{i,j}_0| > 0 \big) \simeq  c(\sigma^{i,j}_0, G^{i,j}_0)\exp\big(-\Delta c(\sigma^{i,j}_0, G^{i,j}_0)\big)$. Taking the ratio of two such densities, multiplying over many time intervals, and then taking $\Delta \to 0$, we obtain the formula \eqref{eq: Girsanov}.}\cite{Grunwald1996,Jacod2002}
 that the Radon-Nikodym derivative satisfies
\begin{equation}\label{eq: Girsanov}
\frac{dP^N_{\mathbf{J}}}{dR^N_i}(\boldsymbol\sigma_{[0, T]}) = \exp\big( N \Gamma^N_{i}\big(\boldsymbol\sigma_{[0,T]},\mathbf{J}\big) \big).
\end{equation}
Write $\tilde{G}^{q,j}_{i,t} = N^{-1/2}\sum_{k\in I_N}J^{jk}\tilde{\sigma}^{q,k}_{i,t}$ and define $\tilde{\tau}_N$ to be the analog of \eqref{eq: tilde tau N stopping},i.e.
\begin{equation}\label{eq: tilde tau N stopping 2}
\tilde{\tau}_N =T\wedge \inf\big\lbrace t: t\in [0,  T] \text{ and }\boldsymbol\sigma_t \notin \mathcal{X}^N\big\rbrace .
\end{equation}
 \begin{lemma}\label{Lemma Change of Measure}
Suppose that for any $\bar{\epsilon} > 0$, there exists $n_0 \in \mathbb{Z}^+$ such that for all $n \geq n_0$, there exists $\mathfrak{n}_0(n) \in \mathbb{Z}^+$, such that for all $\mathfrak{n} \geq \mathfrak{n}_0(n)$, there exists $m_0(n,\mathfrak{n})$ such that for all $m\geq m(n,\mathfrak{n})$,
  \begin{multline}\label{eq: to prove Q N change of measure}
 \sup_{0\leq b < n}\sup_{1\leq i \leq C^N_{\mathfrak{n}}}\lsup{N} N^{-1}\log \mathbb{P}\big(\mathcal{J}_N,\tilde{\tau}_N > t^{(n)}_{b} ,\tilde{\mu}^N(\tilde{\boldsymbol\sigma}_i,\tilde{\mathbf{G}}_i) \in \mathcal{V}^N_i \text{ and }\\ d_W\big( \xi_b(\tilde{\boldsymbol\sigma}_{i,t^{(n)}_b},\tilde{\mathbf{G}}_{i,t^{(n)}_b}) ,  \hat{\mu}^N(\tilde{\boldsymbol\sigma}_{i,t^{(n)}_{b+1}},\tilde{\mathbf{G}}_{i,t^{(n)}_{b+1}})\big)  \geq \bar{\epsilon}Tn^{-1} \big) := - \mathfrak{k} < 0,
 \end{multline}
 for some $\mathfrak{k} > 0$. Then the condition of Lemma \ref{Lemma Main Lemma} is satisfied, i.e. for any $\tilde{\epsilon} > 0$, for large enough $n\in \mathbb{Z}^+$,
 \begin{equation}
\lsup{N}N^{-1}\log \mathbb{P}\big(\mathcal{J}_N \text{ and }\tau_N > t^{(n)}_{b} \text{ and }  d_W\big( \xi_b(\boldsymbol\sigma_{t^{(n)}_b},\mathbf{G}_{t^{(n)}_b}) ,  \hat{\mu}^N(\boldsymbol\sigma_{t^{(n)}_{b+1}},\mathbf{G}_{t^{(n)}_{b+1}}) \big)\geq \tilde{\epsilon} Tn^{-1} \big)  < 0.
 \end{equation}
 \end{lemma}
 \begin{proof}
The event $\mathcal{J}_N$ necessarily implies that $\tilde{\mu}^N \in \hat{\mathcal{W}}_2$. We can thus apply a union-of-events bound to the partition in \eqref{eq: W 2 partition} to obtain that
  \begin{multline}
\mathbb{P}\big(\mathcal{J}_N,\tau_N > t^{(n)}_{b},d_W\big( \xi_b(\boldsymbol\sigma_{t^{(n)}_b},\mathbf{G}_{t^{(n)}_b}) ,  \hat{\mu}^N(\boldsymbol\sigma_{t^{(n)}_{b+1}},\mathbf{G}_{t^{(n)}_{b+1}}) \big)\geq \tilde{\epsilon} \Delta \big)  \\
\leq \sum_{i=1}^{C^N_{\mathfrak{n}}}\mathbb{P}\big(\mathcal{J}_N, \tau_N > t^{(n)}_{b}, d_W\big( \xi_b(\boldsymbol\sigma_{t^{(n)}_b},\mathbf{G}_{t^{(n)}_b}) ,  \hat{\mu}^N(\boldsymbol\sigma_{t^{(n)}_{b+1}},\mathbf{G}_{t^{(n)}_{b+1}}) \big)\geq \tilde{\epsilon} \Delta , \tilde{\mu}^N(\boldsymbol\sigma,\mathbf{G}) \in \mathcal{V}^N_i , \big| \Gamma^N_i(\boldsymbol\sigma,\mathbf{J}) \big| \leq \mathfrak{k} / 2 \big) \\
+ \sum_{i=1}^{C^N_{\mathfrak{n}}}\mathbb{P}\big(\mathcal{J}_N , \tilde{\mu}^N(\boldsymbol\sigma,\mathbf{G}) \in \mathcal{V}^N_i,\big| \Gamma^N_i(\boldsymbol\sigma,\mathbf{J}) \big| > \mathfrak{k} / 2 \big),\label{eq: RHS temporary}
 \end{multline}
 noting that the constant $\mathfrak{k}$ is defined in \eqref{eq: to prove Q N change of measure}. Noting that $C^N_{\mathfrak{n}}$ is polynomial in $N$ (as proved in Lemma \ref{eq: polynomial partition}), thanks to Lemma \ref{Lemma Many Events}, it suffices to prove that each of the terms on the right hand side of \eqref{eq: RHS temporary} are exponentially decaying in $N$. Using the Radon-Nikodym derivative \eqref{eq: Girsanov},
 \begin{align*}
 \mathbb{P}\big(\mathcal{J}_N, \tau_N > t^{(n)}_{b}, &d_W\big( \xi_b(\boldsymbol\sigma_{t^{(n)}_b},\mathbf{G}_{t^{(n)}_b}) ,  \hat{\mu}^N(\boldsymbol\sigma_{t^{(n)}_{b+1}},\mathbf{G}_{t^{(n)}_{b+1}}) \big)\geq \tilde{\epsilon} \Delta , \tilde{\mu}^N(\boldsymbol\sigma,\mathbf{G}) \in \mathcal{V}^N_i , | \Gamma^N_i(\boldsymbol\sigma,\mathbf{J}) | \leq \mathfrak{k} / 2 \big)\\
 &\leq \exp(N\mathfrak{k} / 2)  \mathbb{P}\big(\mathcal{J}_N, \tilde{\tau}_N > t^{(n)}_{b}, d_W\big( \xi_b(\tilde{\boldsymbol\sigma}_{i,t^{(n)}_b},\tilde{\mathbf{G}}_{i,t^{(n)}_b}) ,  \hat{\mu}^N(\tilde{\boldsymbol\sigma}_{i,t^{(n)}_{b+1}},\tilde{\mathbf{G}}_{i,t^{(n)}_{b+1}}) \big)\geq \tilde{\epsilon} \Delta , \tilde{\mu}^N(\tilde{\boldsymbol\sigma}_{i},\tilde{\mathbf{G}}_i) \in \mathcal{V}^N_i \big)\\ &\leq \exp(-N\mathfrak{k}/2),
 \end{align*}
 using the assumption \eqref{eq: to prove Q N change of measure} in the statement of the lemma. It thus remains to prove that
 \begin{equation}\label{eq: to establish tilde Gamma}
\lsup{N} \sup_{1\leq i \leq C^N_{\mathfrak{n}}}N^{-1}\log \mathbb{P}\big(\mathcal{J}_N , \tilde{\mu}^N(\boldsymbol\sigma,\mathbf{G}) \in \mathcal{V}^N_i,\big| \Gamma^N_i(\boldsymbol\sigma,\mathbf{J}) \big| > \mathfrak{k} / 2 \big) < 0.
 \end{equation}
Notice that $\tilde{\mu}^N(\boldsymbol\sigma,\mathbf{G}) \in \mathcal{V}^N_i$ implies that $\tau_i > T$. Recalling that $\hat{\mathfrak{G}}^{q,j}_a := \hat{\mathfrak{G}}^{q,j}_{t^{(m)}_a}$ and $\sigma^{q,j}_a := \sigma^{q,j}_{t^{(n)}_a}$, define the following time-discretized approximation of the Girsanov Exponent,
  \begin{multline}
 \tilde{\Gamma}^N_{i}\big(\boldsymbol\sigma_{[0, T]},\mathbf{J}\big) = N^{-1}\sum_{q \in I_M , j\in I_N} \bigg\lbrace Tm^{-1}\sum_{a=0}^{m-1} \big\lbrace c(\sigma^{q,j}_a , \hat{\mathfrak{G}}^{q,j}_{i,a}) -c(\sigma^{q,j}_{a},G^{q,j}_{a}) \big\rbrace \\ -\frac{1}{2}\sum_{a=0}^{m-1} \big\lbrace \chi\lbrace |G^{q,j}_a| \leq \mathfrak{n} \rbrace\log c\big(\sigma^{q,j}_a,G^{q,j}_a \big)-\chi\lbrace |\hat{\mathfrak{G}}^{q,j}_{i,a}| \leq \mathfrak{n} \rbrace\log c\big(\sigma^{q,j}_{a}, \hat{\mathfrak{G}}^{q,j}_{i,a} \big)\big\rbrace \sigma^{q,j}_a(\sigma^{q,j}_{a+1} - \sigma^{q,j}_a)\bigg\rbrace. \label{eq: discretized Girsanov}
 \end{multline}
One expects the above approximation to be very accurate for large $m \in \mathbb{Z}^+$ because
 \begin{equation}\label{eq: equivalence discretized stochastic integral}
\hat{\sigma}^{q,j}_{a+1} - \hat{\sigma}^{q,j}_a \in \lbrace 0,1 \rbrace \text{ implies that }-\frac{1}{2} \sigma^{q,j}_a(\sigma^{q,j}_{a+1} - \sigma^{q,j}_a) = \hat{\sigma}^{q,j}_{a+1} - \hat{\sigma}^{q,j}_a.
 \end{equation}
 (The probability that $ \hat{\sigma}^{q,j}_{a+1} - \hat{\sigma}^{q,j}_a \geq 2$ is very small once the time interval $Tm^{-1}$ is small). Thus to establish \eqref{eq: to establish tilde Gamma}, it suffices to establish the follow two identities
  \begin{align}\label{eq: to establish tilde Gamma 2}
\lsup{N} \sup_{1\leq i \leq C^N_{\mathfrak{n}}}N^{-1}\log \mathbb{P}\big(\mathcal{J}_N , \tilde{\mu}^N(\boldsymbol\sigma,\mathbf{G}) \in \mathcal{V}^N_i, \Gamma^N_i(\boldsymbol\sigma,\mathbf{J}) - \tilde{\Gamma}^N_i(\boldsymbol\sigma,\mathbf{J})  > \mathfrak{k} / 4 \big) < 0 \\
\lsup{N} \sup_{1\leq i \leq C^N_{\mathfrak{n}}}N^{-1}\log \mathbb{P}\big(\mathcal{J}_N , \tilde{\mu}^N(\boldsymbol\sigma,\mathbf{G}) \in \mathcal{V}^N_i,\big| \tilde{\Gamma}^N_i(\boldsymbol\sigma,\mathbf{J}) \big| > \mathfrak{k} / 4 \big) < 0.\label{eq: to establish tilde Gamma 3}
 \end{align}
We start by establishing \eqref{eq: to establish tilde Gamma 3}. We observe from \eqref{eq: discretized Girsanov} that there exists a function $\mathcal{H}: \mathcal{D}\big( [0,T], \mathcal{E}^M \times \mathbb{R}^M\big) \to \mathbb{R}$ such that
 \begin{equation}
 \tilde{\Gamma}^N_i \big(\boldsymbol\sigma_{[0, T]},\mathbf{J}\big) = \mathbb{E}^{\tilde{\mu}^N(\boldsymbol\sigma , \mathbf{G})}[\mathcal{H}] -  \mathbb{E}^{\tilde{\mu}^N(\boldsymbol\sigma , \hat{\mathfrak{G}_i})}[\mathcal{H}]. 
 \end{equation}
Furthermore $\mathcal{H}$ is a function of the values of the variables at the times $\lbrace t^{(m)}_a \rbrace_{0\leq a \leq m}$. Now if $\tilde{\mu}^N(\boldsymbol\sigma , \mathbf{G}) \in \mathcal{V}^N_i$, then necessarily $\tilde{\mu}^N(\boldsymbol\sigma , \hat{\mathfrak{G}}) \in \mathcal{V}^N_i$. It now follows from (i) the fact that the functions $c$ and $\log c$ are uniformly Lipschitz in their second argument and (ii) Lemma \ref{Lemma Uniform Partition Wasserstein}, that for large enough $\mathfrak{n}$, it must be that
\[
\big| \mathbb{E}^{\tilde{\mu}^N(\boldsymbol\sigma , \mathbf{G})}[\mathcal{H}] -  \mathbb{E}^{\tilde{\mu}^N(\boldsymbol\sigma , \hat{\mathfrak{G}}_i)}[\mathcal{H}] \big| \leq \mathfrak{k} / 4.
\]
We have thus established \eqref{eq: to establish tilde Gamma 3}. It remains to establish \eqref{eq: to establish tilde Gamma 2}. Write
\begin{align*}
F^{q,j}_s =& \chi\lbrace -\mathfrak{n} \leq G^{q,j}_s  \leq \mathfrak{n} \rbrace\log c\big(\sigma^{q,j}_s,G^{q,j}_s \big)-\chi\lbrace -\mathfrak{n} \leq \mathfrak{G}^{q,j}_s  \leq \mathfrak{n} \rbrace\log c\big(\sigma^{q,j}_{s}, \hat{\mathfrak{G}}^{q,j}_{i,s}) \\
f^{q,j}_s =& \chi\lbrace |G^{q,j}_s| > \mathfrak{n} \rbrace \log c\big(\sigma^{q,j}_s,G^{q,j}_s \big).
\end{align*}
We wish to split $\Gamma^N_i(\boldsymbol\sigma,\mathbf{J}) - \tilde{\Gamma}^N_i(\boldsymbol\sigma)$ into the sum of five terms and bound each term separately. First, using \eqref{eq: equivalence discretized stochastic integral}, we notice that the difference of the stochastic integral in $\Gamma^N_i(\boldsymbol\sigma,\mathbf{J})$ and its time-discretized equivalent in $\tilde{\Gamma}^N_i(\boldsymbol\sigma)$ is
\begin{align*}
\int_{t^{(n)}_{a}}^{t^{(m)}_{a+1}} F^{q,j}_{t^{(m)}_a} d\hat{\sigma}^{q,j}_s + \frac{1}{2} F^{q,j}_{t^{(m)}_a}\sigma^{q,j}_a(\sigma^{q,j}_{a+1} - \sigma^{q,j}_a) =  F^{q,j}_{t^{(m)}_a} (\hat{\sigma}^{q,j}_{t^{(m)}_{a+1}} - \hat{\sigma}_{t^{(m)}_{a}}^{q,j}) \chi\big\lbrace \hat{\sigma}^{q,j}_{t^{(m)}_{a+1}} - \hat{\sigma}_{t^{(m)}_{a}}^{q,j} \geq 2 \big\rbrace.
\end{align*}
Second, it is immediate from the definition that it is always the case that $-\mathfrak{n} \leq \hat{\mathfrak{G}}^{q,j}_s  \leq \mathfrak{n}$.
In order that \eqref{eq: to establish tilde Gamma 2} is satisfied, it suffices to demonstrate the following identities,
\begin{align}
&\lsup{N}N^{-1}\log \mathbb{P}\bigg(\mathcal{J}_N, \sum_{j\in I_N,q\in I_M}\int_{0}^{T}f^{q,j}_s d\hat{\sigma}^{q,j}_s   \geq \frac{N\mathfrak{k}}{20} \bigg) < 0\label{eq: to establish Girsanov 1}\\
&\lsup{N}N^{-1}\log \mathbb{P}\bigg(\mathcal{J}_N, \sum_{j\in I_N,q\in I_M}\sum_{a=0}^{m-1}\int_{t^{(m)}_a}^{t^{(m)}_{a+1}}(F^{q,j}_s - F^{q,j}_{t^{(m)}_a}) (d\hat{\sigma}^{q,j}_s - c(\sigma^{q,j}_s,G^{q,j}_{s})ds)  \geq \frac{N\mathfrak{k}}{20} \bigg) < 0\label{eq: to establish Girsanov 2}\\
&\lsup{N}N^{-1}\log \mathbb{P}\bigg(\mathcal{J}_N , \tilde{\mu}^N(\boldsymbol\sigma,\mathbf{G}) \in \mathcal{V}^N_i ,  \sum_{j\in I_N,q\in I_M}\sum_{a=0}^{m-1}\int_{t^{(m)}_a}^{t^{(m)}_{a+1}}(F^{q,j}_s - F^{q,j}_{t^{(m)}_a}) c(\sigma^{q,j}_s,G^{q,j}_{s})ds  \geq \frac{N\mathfrak{k}}{20}  \bigg) < 0\label{eq: to establish Girsanov 3}\\
&\lsup{N}N^{-1}\log \mathbb{P}\bigg(\mathcal{J}_N, \sum_{a=0}^{m-1}\sum_{q\in I_M,j\in I_N} F^{q,j}_{t^{(m)}_a} (\hat{\sigma}^{q,j}_{t^{(m)}_{a+1}} - \hat{\sigma}_{t^{(m)}_{a}}^{q,j}) \chi\big\lbrace \hat{\sigma}^{q,j}_{t^{(m)}_{a+1}} - \hat{\sigma}_{t^{(m)}_{a}}^{q,j} \geq 2 \big\rbrace  \geq \frac{N\mathfrak{k}}{20} \bigg) < 0\label{eq: to establish Girsanov 4}\\
&\lsup{N}N^{-1}\log \mathbb{P}\bigg(\mathcal{J}_N , \tilde{\mu}^N(\boldsymbol\sigma,\mathbf{G}) \in \mathcal{V}^N_i \text{ and }\bigg|\sum_{q \in I_M , j\in I_N}  \sum_{a=0}^{m-1} \bigg\lbrace \frac{Tc(\sigma^{q,j}_a , \hat{\mathfrak{G}}^{q,j}_{i,a})-Tc(\sigma^{q,j}_{a},G^{q,j}_{a})}{m}\nonumber  \\&- \int_{t^{(m)}_a}^{t^{(m)}_{a+1}}\big\lbrace  c(\sigma^{q,j}_s , \hat{\mathfrak{G}}^{q,j}_{i,s}) -  c(\sigma^{q,j}_{s},G^{q,j}_{s}) \big\rbrace ds\bigg\rbrace \bigg| > \frac{N\mathfrak{k}}{20} \bigg) < 0. \label{eq: to establish Girsanov 5}
\end{align}
We start with \eqref{eq: to establish Girsanov 1}. The event $\mathcal{J}_N$ implies that $N^{-1}\sum_{j\in I_N} \chi\lbrace |G^{q,j}_s| > \mathfrak{n} \rbrace \leq 3\mathfrak{n}^{-2}$. Thus, since $c(\cdot,\cdot)$ is uniformly upperbounded by $c_1$,
\[
N^{-1}\sum_{q\in I_M,j\in I_N}  \chi\lbrace |G^{q,j}_s| > \mathfrak{n} \rbrace \exp( f^{q,j}_s) \leq  3M\mathfrak{n}^{-2}\log(c_1).
\]
Since the right hand side goes to zero as $\mathfrak{n}\to\infty$, \eqref{eq: to establish Girsanov 1} follows from (ii) of Lemma \ref{Lemma Concentration Poisson}, as long as $\mathfrak{n}$ is large enough. 

\eqref{eq: to establish Girsanov 2} follows from the concentration inequality in (i) of Lemma \ref{Lemma Concentration Poisson}, employing the facts that (i) $|F^{q,j}_s|$ is uniformly upperbounded, and (ii)
$\hat{\sigma}^{q,j}_t - \int_0^t c(\sigma^{q,j}_s,G^{q,j}_s)ds$ is a compensated Poisson Process (a Martingale \cite{Anderson2015}). 

For \eqref{eq: to establish Girsanov 3}, the boundedness of $c(\cdot,\cdot)$ by $c_1$ (in the first line), and Jensen's Inequality (in the second line) imply that 
\begin{align*}
  N^{-1}\big|\sum_{j\in I_N,q\in I_M}\sum_{a=0}^{m-1}\int_{t^{(m)}_a}^{t^{(m)}_{a+1}}(F^{q,j}_s - F^{q,j}_{t^{(m)}_a}) c(\sigma^{q,j}_s,G^{q,j}_{s})ds \big| &\leq
N^{-1}c_1   \sum_{j\in I_N,q\in I_M}\sum_{a=0}^{m-1}\int_{t^{(m)}_a}^{t^{(m)}_{a+1}}\big| F^{q,j}_s - F^{q,j}_{t^{(m)}_a}\big| ds   \\
&\leq c_1    \int_0^{T}\big\lbrace N^{-1}\sum_{j\in I_N,q\in I_M}\big| F^{q,j}_t - F^{q,j}_{t^{(m)}}\big|^2 \big\rbrace^{1/2} dt \\
&\leq c_1\sqrt{c_L} \sqrt{3}   \int_0^{T}\big\lbrace N^{-1}\sum_{j\in I_N,q\in I_M}\big| \sigma^{q,j}_t - \sigma^{q,j}_{t^{(m)}}\big|^2 \big\rbrace^{1/2} dt ,
\end{align*}
using (i) the fact that $\log c(\cdot,\cdot)$ has Lipschitz constant $c_L$ (in its second argument), and (ii) as long as the event $\mathcal{J}_N$ holds. Define the renewed Poisson Processes $\lbrace Y^{q,j}_a(t) \rbrace_{q\in I_M, j \in I_N}$ to be
\begin{align}
Y_a^{q,j}(t) :=& Y^{q,j}\big(t+ \int_0^{t^{(m)}_a} c(\sigma^{q,j}_s , G^{q,j}_s) ds \big) - Y^{q,j}\big( \int_0^{t^{(m)}_a} c(\sigma^{q,j}_s , G^{q,j}_s) ds \big). \label{eq: Y a q j t definition 2}
\end{align}
Now since the flipping intensity is uniformly upperbounded by $c_1$, if $t \leq t^{(m)}_{a+1}$ then
\begin{align*}
\sum_{q\in I_M,j\in I_N}\big| \sigma^{q,j}_t - \sigma^{q,j}_{t_a^{(m)}}\big|^2 \leq & 4\sum_{q\in I_M,j\in I_N} \chi\big\lbrace \hat{Y}_a^{q,j}(c_1t - c_1t^{(m)}_a) \geq 1 \big\rbrace \text{ where }\\
\hat{Y}_a^{q,j}(t) =& Y^{q,j}_a\big( t \wedge \hat{\tau}_a^{q,j} \big) \text{ and }\\
\hat{\tau}_a^{q,j} =& \inf\big\lbrace u \geq 0 : u = \int_{t^{(m)}_a}^{t^{(m)}_{a+1}}  c(\sigma^{q,j}_s , G^{q,j}_s) ds \big\rbrace .
\end{align*}
Now $t - t^{(m)} \leq \delta$, where $\delta = Tm^{-1}$. Jensen's Inequality thus implies that
\[
 \int_0^{T}\big\lbrace N^{-1}\sum_{j\in I_N,q\in I_M}\chi\big\lbrace \hat{Y}_a^{q,j}(c_1 \delta ) \geq 1 \big\rbrace  \big\rbrace^{1/2} dt \leq \sqrt{T} \bigg\lbrace \int_0^T N^{-1}\sum_{j\in I_N,q\in I_M}\chi\big\lbrace \hat{Y}_a^{q,j}(c_1\delta) \geq 1 \big\rbrace dt \bigg\rbrace^{1/2}.
\]
We thus find that there is a constant $C$ such that
\begin{align*}
 \mathbb{P}\bigg(\mathcal{J}_N , \tilde{\mu}^N(\boldsymbol\sigma,\mathbf{G}) \in \mathcal{V}^N_i ,  \sum_{j\in I_N,q\in I_M}\sum_{a=0}^{m-1}\int_{t^{(m)}_a}^{t^{(m)}_{a+1}}(F^{q,j}_s - F^{q,j}_{t^{(m)}_a}) c(\sigma^{q,j}_s,G^{q,j}_{s})ds  \geq \frac{N\mathfrak{k}}{16}  \bigg) \\
 \leq \mathbb{P}\big(  N^{-1}\sum_{a=0}^{m-1}\sum_{j\in I_N,q\in I_M}\chi\big\lbrace \hat{Y}_a^{q,j}(c_1\delta) \geq 1 \big\rbrace \geq C \big).
\end{align*}
For large enough $m$, this probability is exponentially decaying, thanks to Lemma \ref{Lemma Upper Bound Jump Rate}.

For \eqref{eq: to establish Girsanov 4}, since the flipping rate is uniformly upperbounded by $c_1$, there exists a constant $C(\mathfrak{n})$ such that $F^{q,j}_S \leq C(\mathfrak{n})$. Thus by Chernoff's Inequality,
\begin{align}
\mathbb{P}\bigg(\mathcal{J}_N, \sum_{a=0}^{m-1}\sum_{q\in I_M,j\in I_N} &F^{q,j}_{t^{(m)}_a} (\hat{\sigma}^{q,j}_{t^{(m)}_{a+1}} - \hat{\sigma}_{t^{(m)}_{a}}^{q,j}) \chi\big\lbrace \hat{\sigma}^{q,j}_{t^{(m)}_{a+1}} - \hat{\sigma}_{t^{(m)}_{a}}^{q,j} \geq 2 \big\rbrace  \geq \frac{N\mathfrak{k}}{20} \bigg)\nonumber \\
&\leq \mathbb{P}\bigg( \sum_{a=0}^{m-1}\sum_{q\in I_M,j\in I_N} \hat{Y}_a^{q,j}(c_1\delta)\chi\big\lbrace \hat{Y}_a^{q,j}(c_1\delta) \geq 2 \big\rbrace  \geq \frac{N\mathfrak{k}}{16C(\mathfrak{n})}  \bigg) \nonumber\\
&\leq \mathbb{E}\bigg[ \exp\bigg( v\sum_{a=0}^{m-1}\sum_{q\in I_M,j\in I_N} \hat{Y}_a^{q,j}(c_1\delta)\chi\big\lbrace \hat{Y}_a^{q,j}(c_1\delta) \geq 2 \big\rbrace -\frac{N\mathfrak{k}v}{20C(\mathfrak{n})}  \bigg) \bigg], \label{eq: iterative expectation a}
\end{align}
for some constant $v > 0$. To bound \eqref{eq: iterative expectation a}, we start by evaluating the integral conditionally on $\mathcal{F}_{t^{(m)}_{m-1}}$. Notice that $\lbrace \hat{Y}^{q,j}_{m-1} \rbrace_{q\in I_M, j\in I_N}$ are independent of $\mathcal{F}_{t^{(m)}_{m-1}}$ (thanks to the renewal property of Poisson Processes). Also $\hat{Y}^{q,j}_a(c_1 \delta) \chi\lbrace \hat{Y}^{q,j}_a(c_1 \delta ) \geq 2 \rbrace \leq Y^{q,j}_a(c_1 \delta) \chi\lbrace Y^{q,j}_a(c_1 \delta ) \geq 2 \rbrace$. We thus find that, for $a= m-1$, and using the fact that $\mathbb{P}(Y^{q,j}_a(c_1 \delta) = r) = \exp(-r\delta c_1)(\delta c_1)^r / (r!)$, 
\begin{align*}
 \mathbb{E}\big[ \exp\big( v\sum_{q\in I_M,j\in I_N} \hat{Y}^{q,j}_a(c_1\delta ) \chi\lbrace \hat{Y}^{q,j}_a(c_1\delta) \geq 2 \rbrace \big) \; | \; \mathcal{F}_{t^{(m)}_a}\big]
\leq \big[\sum_{r=2}^\infty  \big\lbrace \delta c_1 \exp\big(-\delta c_1 + v\big) \big\rbrace^{r} \big]^{NM}.
\end{align*}
We take $m$ to be large enough that $\delta c_1 \exp(-\delta c_1 + v) \leq 1/4$, which means that the formula for summing a geometric sequence implies that
\[
 \big[\sum_{r=2}^\infty  \big\lbrace \delta c_1 \exp\big(-\delta c_1 + v\big) \big\rbrace^{r} \big]^{NM} \leq 4^{-NM}.
\]
We then continue the argument, evaluating \eqref{eq: iterative expectation a} conditionally on $\mathcal{F}_{t^{(m)}_{m-2}}$, then $\mathcal{F}_{t^{(m)}_{m-3}}$ ... and finally $\mathcal{F}_{t^{(m)}_0}$. We find that \eqref{eq: iterative expectation a} must be less than or equal to $4^{-NM(m+1)}\exp\big( -\frac{N\mathfrak{k}v}{20C(\mathfrak{n})} \big)$. We have established  \eqref{eq: to establish Girsanov 4}.

We see that \eqref{eq: to establish Girsanov 5} is a difference between an integral and its time-discretized approximation, and can easily be shown to be true for large enough $m$.

\end{proof}
\section{Taylor Expansion of Test Functions}\label{Expectation Test Functions}

After the change of measure of the previous section, our task is easier, because now the spin-flipping intensity of $\tilde{\boldsymbol\sigma}_{i,t}$ is independent of the connections $\mathbf{J}$. This section (and the remainder of the paper) is oriented towards proving condition \eqref{eq: to prove Q N change of measure} of Lemma \ref{Lemma Change of Measure}. This proof is accomplished through the comparison of the expectations of test functions, using the dual Kantorovich representation of the Wasserstein distance. We will Taylor expand the test functions to second order, and (in subsequent sections) demonstrate that the expectation with respect to the flow operator $\Psi_t$ almost matches the expectation with respect to the empirical process.\\
 Let $\mathfrak{H}$ be the set of all functions that are uniformly Lipschitz, i.e.
\begin{align}
\mathfrak{H} = \big\lbrace f \in \mathcal{C}\big(\mathcal{E}^M \times \mathbb{R}^M\big) \; : \; |f(\boldsymbol\alpha,\mathbf{x}) - f(\boldsymbol\beta,\mathbf{z}) | \leq \norm{\boldsymbol\alpha-\boldsymbol\beta} + \norm{\mathbf{x} - \mathbf{z}} \text{ and }f(\mathbf{0}) = 0 \big\rbrace.
\end{align}
It follows from the Kantorovich-Rubinstein theorem \cite{Gibbs2002} that 
\begin{equation}
d_W\big( \mu,\nu\big) = \sup_{\mathbf{f}\in \mathfrak{H}}\big\lbrace \big| \mathbb{E}^{\mu}[ f ] -  \mathbb{E}^{\nu}[ f ] \big| \big\rbrace. \label{eq: kantorovich}
\end{equation}
Our proofs only make use of a finite number of test functions: so we must demonstrate that the right hand side of the above equation can be approximated arbitrarily well by taking the supremum over a finite subset. Furthermore we require that the test functions are three-times differentiable in order that the expectations of stochastic fluctuations converge smoothly. Let $\mathfrak{H}_a$ be the set of all $f \in \mathfrak{H}$ satisfying the following assumptions.
\begin{itemize}
\item $f(\boldsymbol\alpha,\mathbf{x}) = 0$ for $\norm{\mathbf{x}} \geq a$. 
\item $f(\boldsymbol\alpha,\mathbf{x}) = \chi\lbrace \boldsymbol\alpha = \boldsymbol\beta \rbrace \bar{f}(\mathbf{x})$, for some fixed $\boldsymbol\beta \in \mathcal{E}^M$ and $\bar{f} \in \mathcal{C}^3(\mathbb{R}^M)$.
\item Write the first, second and third order partial derivatives, in the second variable, as (respectively) $\bar{f}_j , \bar{f}_{jk} , \bar{f}_{jkl}$, for $j,k,l \in I_M$. These are all assumed to be uniformly bounded by $1$.
\end{itemize}
\begin{lemma}
For any $\delta > 0$, there exists $a \in \mathbb{Z}^+$ and a finite subset $\bar{\mathfrak{H}}_a \subset \mathfrak{H}_a$ such that for all $\mu,\nu \in \mathcal{W}_2$,
\begin{equation}
d_W(\mu,\nu) \leq \delta + \sup_{f\in \bar{\mathfrak{H}}_a}\big\lbrace \big| \mathbb{E}^{\mu}[ f ] -  \mathbb{E}^{\nu}[ f ] \big| \big\rbrace
\end{equation}
\end{lemma}
\begin{proof}
For any $\mu \in \mathcal{W}_2$, any $f\in \mathfrak{H}$, and $a > 0$,
\begin{align}
\mathbb{E}^{\mu}\big[ f \chi\lbrace \norm{\mathbf{x}} \geq a \rbrace\big] \leq\mathbb{E}^{\mu}\big[ \norm{\mathbf{x}} \chi\lbrace \norm{\mathbf{x}} \geq a \rbrace\big] \nonumber\\
\leq a^{-1}\mathbb{E}^{\mu}\big[ \norm{\mathbf{x}}^2\big] \leq 3/a.
\end{align}
Thus for any $\delta >0$, for large enough $a$,
\begin{equation}
d_W\big( \mu,\nu\big) \leq \delta / 2 + \sup_{\mathbf{f}\in \tilde{\mathfrak{H}}_a}\big\lbrace \big| \mathbb{E}^{\mu}[ f ] -  \mathbb{E}^{\nu}[ f ] \big| \big\rbrace,
\end{equation}
where $\tilde{\mathfrak{H}}_a$ is the set of all $f \in \mathfrak{H}$ such that $f(\boldsymbol\alpha,\mathbf{x}) = 0$ if $\norm{\mathbf{x}} \geq a$. It remains to demonstrate that we can find a finite subset $\bar{\mathfrak{H}}_a$ of $\mathfrak{H}_a$ such that
\[
\sup_{\mathbf{f}\in \tilde{\mathfrak{H}}_a}\big\lbrace \big| \mathbb{E}^{\mu}[ f ] -  \mathbb{E}^{\nu}[ f ] \big| \big\rbrace \leq \delta / 2 + \sup_{\mathbf{f}\in \bar{\mathfrak{H}}_a}\big\lbrace \big| \mathbb{E}^{\mu}[ f ] -  \mathbb{E}^{\nu}[ f ] \big| \big\rbrace .
\]
Since continuous functions on compact domains can be approximated arbitrarily well by smooth functions, it must be that
\[
\sup_{\mathbf{f}\in \tilde{\mathfrak{H}}_a}\big\lbrace \big| \mathbb{E}^{\mu}[ f ] -  \mathbb{E}^{\nu}[ f ] \big| \big\rbrace  = \sup_{\mathbf{f}\in \mathfrak{H}_a}\big\lbrace \big| \mathbb{E}^{\mu}[ f ] -  \mathbb{E}^{\nu}[ f ] \big| \big\rbrace. 
\]
It follows from the Arzela-Ascoli Theorem that $\mathfrak{H}_a$ is compact. Thus we can find a finite cover of $\mathfrak{H}_a$ such that every function in $\mathfrak{H}_a$ is within $\delta / 2$ of a function in the finite cover (relative to the supremum norm).
\end{proof}

Now set $\delta = \Delta \bar{\epsilon} / 2$, and let $a \in \mathbb{R}^+$ and $\mathfrak{h}_a$ be such that for all $\mu,\nu \in \bar{\mathcal{P}}$, 
$d_W(\mu,\nu) \leq \delta + \sup_{\mathbf{f}\in \bar{\mathfrak{H}}_a}\big\lbrace \big| \mathbb{E}^{\mu}[ f ] -  \mathbb{E}^{\nu}[ f ] \big| \big\rbrace$. We write $\mathfrak{F} \subset \mathcal{C}^3(\mathbb{R}^M)$ to be such that
\begin{align}
\bar{\mathfrak{H}}_a =& \big\lbrace f(\boldsymbol\alpha,\mathbf{x}) = \chi\lbrace \boldsymbol\alpha=\boldsymbol\beta\rbrace \phi(\mathbf{x}) \text{ for some }\boldsymbol\beta \in \mathcal{E}^M \text{ and }\phi \in \mathfrak{F}\big\rbrace 
\end{align}
and we define the pseudo-metric\footnote{This satisfies all of the axioms of a metric, except that $d_K(\mu,\nu)=0$ does not necessarily imply that $\mu=\nu$.}
\begin{equation}
d_K(\mu,\nu) =  \sup_{\phi \in \mathfrak{F}, \boldsymbol\beta \in \mathcal{E}^M}\big\lbrace \big| \mathbb{E}^{\mu}\big[ \phi(\mathbf{x})\chi\lbrace\boldsymbol\alpha = \boldsymbol\beta \rbrace \big] - \mathbb{E}^{\nu}\big[ \phi(\mathbf{x})\chi\lbrace\boldsymbol\alpha = \boldsymbol\beta \rbrace \big]  \big| \big\rbrace.
\end{equation}
Henceforth we drop the subscript $\mathfrak{q}$ from the processes $\tilde{\boldsymbol\sigma}_{\mathfrak{q},t}$ and $\tilde{\mathbf{G}}_{\mathfrak{q},t}$. We find that for the condition \eqref{eq: to prove Q N change of measure} of Lemma \ref{Lemma Change of Measure} to be satisfied, it suffices for us to prove that for any $\bar{\epsilon}> 0$, for all sufficiently large $n$ and all $0\leq b \leq n-1$,
\begin{multline}\label{Lemma Main Lemma 2}
\lsup{N}\sup_{1\leq \mathfrak{q} \leq C^N_{\mathfrak{n}}} N^{-1}\log \mathbb{P}\big(\mathcal{J}_N,\tilde{\mu}^N(\tilde{\boldsymbol\sigma},\tilde{\mathbf{G}}) \in \mathcal{V}^N_{\mathfrak{q}}, \tilde{\tau}_N > t^{(n)}_b, d_K\big(\xi_b(\tilde{\boldsymbol\sigma},\tilde{\mathbf{G}}) ,  \hat{\mu}^N(\tilde{\boldsymbol\sigma}_{b+1},\tilde{\mathbf{G}}_{b+1})\big)\geq \bar{\epsilon}\Delta \big) < 0,
\end{multline}
recalling that $\tilde{G}^{p,j}_t = N^{-1/2}\sum_{k\in I_N}J^{jk} \tilde{\sigma}^{p,k}_t$ and $\xi_b(\tilde{\boldsymbol\sigma},\tilde{\mathbf{G}})$ is the law of the random variables in \eqref{eq: hybrid 17}-\eqref{eq: hybrid 18}. We emphasize that throughout the rest of this paper, $\tilde{\boldsymbol\sigma}_{b} := \tilde{\boldsymbol\sigma}_{t^{(n)}_{b}}$: that is the subscript is with respect to the $n+1$-point time discretization. Write the first derivative of $\phi \in \mathfrak{F}_{\mathfrak{m}}$ with respect to the $j^{th}$ variable as $\phi_j$, the second derivative of $\phi \in \mathfrak{F}_{\mathfrak{m}}$ with respect to the $j^{th}$ and $k^{th}$ variables as $\phi_{jk}$, and the third derivative as $\phi_{jkl}$.

We enumerate $\mathfrak{F}$ as $\mathfrak{F} = \big\lbrace \phi^a \big\rbrace_{a=1}^{|\mathfrak{F}|}$. For $\boldsymbol\alpha \in \mathcal{E}^M$, define
\begin{align}
Q^{a,\boldsymbol\alpha}_b =& \Exp^{\hat{\mu}^N(\tilde{\boldsymbol\sigma}_{b+1},\tilde{\mathbf{G}}_{b+1})}\big[ \phi^a(\mathbf{x}) \chi( \boldsymbol\sigma= \boldsymbol\alpha)\big] = N^{-1}\sum_{j\in I_N}\chi\lbrace \tilde{\boldsymbol\sigma}^j_{b+1}= \boldsymbol\alpha \rbrace \phi^a(\tilde{\mathbf{G}}^j_{b+1}) \\
R^{a,\boldsymbol\alpha}_b =& \Exp^{\xi_{b+1}(\tilde{\boldsymbol\sigma},\tilde{\mathbf{G}})}\big[ \phi^a(\mathbf{x}) \chi(\boldsymbol\sigma = \boldsymbol\alpha )\big].
\end{align}
We first establish a more workable expression for $R^{a,\boldsymbol\alpha}_b$.
\begin{lemma}\label{Lemma Exp Empirical Measure} 
Recall that $\boldsymbol\alpha[i] \in \mathcal{E}^M$ is the same as $\boldsymbol\alpha$, except that the $i^{th}$ spin has a flipped sign.
\begin{multline}
R^{a,\boldsymbol\alpha}_{b}=\Delta N^{-1}\sum_{j\in I_N,i\in I_M}\phi^a(\tilde{\mathbf{G}}^j_b)\chi\lbrace \tilde{\boldsymbol\sigma}^j_b= \boldsymbol\alpha[i] \rbrace c(-\alpha^i,\tilde{G}^{i,j}_b) +N^{-1}\sum_{j\in I_N} \chi\lbrace \tilde{\boldsymbol\sigma}^j_b= \boldsymbol\alpha \rbrace\big[\phi^a\big(\tilde{\mathbf{G}}^j_b \big) + \Delta\sum_{i\in I_M}\big\lbrace  \\ \phi^a_i(\tilde{\mathbf{G}}^{j}_b)  m^{\hat{\mu}^N_b,i}(\boldsymbol\alpha,\tilde{\mathbf{G}}^j_b) +2L^{\hat{\mu}^N_b}_{ii}\phi^a_{ii}(\tilde{\mathbf{G}}^{j}_b)  -\phi^a\big(\tilde{\mathbf{G}}^j_b \big)  c(\alpha^i,\tilde{G}^{i,j}_b)\big\rbrace \big]+ O\big( (\Delta)^{3/2}\big).
\end{multline}
\end{lemma}
\begin{proof}
Recall the definition of $\xi_b(\tilde{\boldsymbol\sigma},\tilde{\mathbf{G}})$, in terms of independent unit-intensity Poisson processes $\lbrace \tilde{Y}^p(t)\rbrace_{p\in I_M}$ and independent Brownian motions $\lbrace \tilde{W}^p_t \rbrace_{p\in I_M}$ in \eqref{eq: hybrid 17}-\eqref{eq: hybrid 18}. We can then write $R^{a,\boldsymbol\alpha}_b$ as
\begin{align}
R^{a,\boldsymbol\alpha}_b=& N^{-1}\sum_{j\in I_N} \mathbb{E}\big[\chi\lbrace \boldsymbol\zeta^j_{\Delta} = \boldsymbol\alpha \rbrace\phi^a\big(\mathbf{x}^j_{\Delta} \big) \; | \tilde{\boldsymbol\sigma}_b, \tilde{\mathbf{G}}_b \big] \text{ where }\boldsymbol\zeta^j_{\Delta} = (\zeta_{\Delta}^{p,j})_{p\in I_M} \text{ and } \label{eq: R transformed 0 0} \\
\zeta^{p,j}_{\Delta} =& \tilde{\sigma}^{p,j}_b A\cdot \tilde{Y}^p\big( \Delta c(\tilde{\sigma}^{p,j}_b,\tilde{G}^{p,j}_b) \big) \label{eq: hybrid 17 again} \\
\mathbf{x}^j_{\Delta} =& \tilde{\mathbf{G}}^j_b +\Delta \mathbf{m}^{\hat{\mu}^N_b}(\tilde{\boldsymbol\sigma}^j_b , \tilde{\mathbf{G}}^j_b)+\mathbf{D}^{\hat{\mu}^N_b}\tilde{\mathbf{W}}_{\Delta}\text{ where }
D^{\hat{\mu}^N_b}_{ij} = 2\sqrt{L^{\hat{\mu}^N_b}_{ii}}\delta(i,j),\label{eq: hybrid 18 again} 
\end{align}
and the expectation in \eqref{eq: R transformed 0 0} is taken with respect to the $\tilde{\mathbf{Y}}$ and $\tilde{\mathbf{W}}_{\Delta}$ random variables, holding $\tilde{\boldsymbol\sigma}_b$ and $\tilde{\mathbf{G}}_b$ to be fixed, and $A \cdot x := (-1)^x$. Write 
\[
X^j_{b+1}=\chi\lbrace \boldsymbol\zeta^j_{\Delta}= \boldsymbol\alpha \rbrace -  \chi\lbrace \tilde{\boldsymbol\sigma}^j_b= \boldsymbol\alpha \rbrace +\Delta\sum_{i\in I_M}\big\lbrace c(\tilde{\sigma}^{i,j}_b, \tilde{G}^{i,j}_b)\chi\lbrace \tilde{\boldsymbol\sigma}^j_b= \boldsymbol\alpha \rbrace-c(-\tilde{\sigma}^{i,j}_b, \tilde{G}^{i,j}_b)\chi\lbrace \tilde{\boldsymbol\sigma}^j_b= \boldsymbol\alpha[i] \rbrace \big\rbrace.
\]
Basic properties of the Poisson process - and recalling that the jump intensity $c$ is uniformly upperbounded by $c_1$ - imply that $\mathbb{E}[ X^j_{b+1} \; | \; \tilde{\boldsymbol\sigma}_b, \tilde{\mathbf{G}}_b ] = O(\Delta^2)$ \cite{Ethier1986} (one can see from the Komolgorov Forward equation \eqref{eq: ODE for P J sigma} why this is true). We thus find that
\begin{multline*}
R^{a,\boldsymbol\alpha}_b=N^{-1}\sum_{j\in I_N}\big(\chi\lbrace \tilde{\boldsymbol\sigma}^j_b= \boldsymbol\alpha \rbrace +\Delta\sum_{i\in I_M}\big\lbrace c(-\tilde{\sigma}^{i,j}_b, \tilde{G}^{i,j}_b)\chi\lbrace \tilde{\boldsymbol\sigma}^j_b= \boldsymbol\alpha[i] \rbrace - c(\tilde{\sigma}^{i,j}_b, \tilde{G}^{i,j}_b) \chi\lbrace \tilde{\boldsymbol\sigma}^j_b= \boldsymbol\alpha \rbrace  \big\rbrace\big) \\
\times\mathbb{E}\big[\phi^a\big(\tilde{\mathbf{G}}^j_b + \Delta\mathbf{m}^{\hat{\mu}^N_b}(\tilde{\boldsymbol\sigma}^j_b, \tilde{\mathbf{G}}^j_b) + \mathbf{D}^{\hat{\mu}^N_b}\tilde{\mathbf{W}}_{\Delta} \big)\; | \tilde{\boldsymbol\sigma}_b , \tilde{\mathbf{G}}_b \big]
 + O\big(\Delta^2\big).
\end{multline*}
Applying a Taylor expansion, and noting that the third order partial derivatives of $\phi^a$ are uniformly bounded, we obtain that
\begin{multline}
\mathbb{E}\big[\phi^a\big(\tilde{\mathbf{G}}^j_b + \Delta\mathbf{m}^{\hat{\mu}^N_b}(\tilde{\boldsymbol\sigma}^j_b, \tilde{\mathbf{G}}^j_b) + \mathbf{D}^{\hat{\mu}^N_b}\tilde{\mathbf{W}}_{\Delta} \big)\; | \; \tilde{\boldsymbol\sigma},\tilde{\mathbf{G}} \big] \\ =\phi^a(\tilde{\mathbf{G}}^j_b )+\Delta \sum_{i\in I_M} \big\lbrace  \phi^a_i(\tilde{\mathbf{G}}^j_b) m^{\hat{\mu}^N_b,i}(\tilde{\boldsymbol\sigma}^j_b, \tilde{\mathbf{G}}^j_b)+(D^{\hat{\mu}^N_b}_{ii})^2\phi^a_{ii}(\tilde{\mathbf{G}}^j_b) / 2\big\rbrace + O\big(\Delta^{3/2}\big),
\end{multline}
since $\mathbb{E}[ \| \tilde{\mathbf{W}}_{\Delta}\|^3 ] = O\big( (\Delta)^{3/2}\big)$. This implies the lemma.
\end{proof}
Using a union-of-events bound, we obtain that 
  \begin{multline}
 \mathbb{P}\big( \mathcal{J}_N,\tilde{\mu}^N \in \mathcal{V}^N_{\mathfrak{q}}, d_K\big( \xi_b(\tilde{\boldsymbol\sigma},\tilde{\mathbf{G}}) ,  \hat{\mu}^N(\tilde{\boldsymbol\sigma}_{b+1},\tilde{\mathbf{G}}_{b+1}) \big) \geq \bar{\epsilon}\Delta \big) \\
\leq  \sum_{\phi^a \in \mathfrak{F}}\sum_{\boldsymbol\alpha\in\mathcal{E}^M}\mathbb{P}\big( \mathcal{J}_N, \tilde{\mu}^N \in \mathcal{V}^N_{\mathfrak{q}},\big|Q^{a,\boldsymbol\alpha}_b -R^{a,\boldsymbol\alpha}_b \big| \geq \bar{\epsilon} \Delta    \big). 
\end{multline}
The implication of the above argument is that, in order that \eqref{Lemma Main Lemma 2} is satisfied, and making use of Lemma \ref{Lemma Many Events}, it suffices to prove the following lemma. 
\begin{lemma}\label{Main Lemma Taylor}
In order that condition \eqref{eq: to prove Q N change of measure} of Lemma \ref{Lemma Change of Measure} is satisfied, it suffices to prove the following statement. For any $\epsilon > 0$, there exists $n_0$ such that for all $n \geq n_0$, there exists $\mathfrak{n}_0 \in \mathbb{Z}^+$ such that for all $\mathfrak{n} \geq \mathfrak{n}_0$,\begin{equation}\nonumber
\sup_{0\leq b \leq n-1} \lsup{N} \sup_{1\leq \mathfrak{q} \leq C^N_{\mathfrak{n}}}\sup_{\phi^a \in \mathfrak{F},\boldsymbol\alpha \in \mathcal{E}^M} N^{-1}\log  \mathbb{P}\big(\mathcal{J}_N ,\tilde{\mu}^N(\tilde{\boldsymbol\sigma},\tilde{\mathbf{G}}) \in \mathcal{V}^N_{\mathfrak{q}} ,\tilde{\tau}_N > t^{(n)}_b, \big|Q^{a,\boldsymbol\alpha}_b -R^{a,\boldsymbol\alpha}_b \big| \geq \epsilon \Delta\big)< 0.
\end{equation}
\end{lemma}
Substituting the expression for $R^{a,\boldsymbol\alpha}_{b+1}$ in Lemma \ref{Lemma Exp Empirical Measure}, we find that the difference can be decomposed as
\begin{equation}\label{eq: alpha decomposition}
\sum_{\boldsymbol\alpha \in \mathcal{E}^M}(Q^{a,\boldsymbol\alpha}_{b+1}- R^{a,\boldsymbol\alpha}_{b+1}) =
O(\Delta^{3/2}) + \sum_{\boldsymbol\alpha\in \mathcal{E}^M}\sum_{i=1}^6\beta^i(\boldsymbol\alpha,\tilde{\boldsymbol\sigma},\tilde{\mathbf{G}})  ,
\end{equation}
and $\lbrace \beta^i \rbrace_{i=1}^5$ are defined as follows (the dependence of $\beta^i$ on $a$ has been neglected from the notation). Our aim is to decompose the difference into terms that can either be controlled with Poisson concentration inequalities or controlled with the Gaussian law of the connections. Here and below, $\hat{\mu}^N_b := \hat{\mu}^N_b(\tilde{\boldsymbol\sigma},\tilde{\mathbf{G}})$. The term $\beta^1$ represents the leading order change in the two expectations due to jumps in the spins, while holding the field to be constant, i.e.
\begin{multline}
\beta^1(\boldsymbol\alpha,\tilde{\boldsymbol\sigma},\tilde{\mathbf{G}}) = N^{-1}\sum_{j\in I_N} \phi^a(\tilde{\mathbf{G}}^{j}_b) \big( \chi\lbrace \tilde{\boldsymbol\sigma}^j_{b+1}= \boldsymbol\alpha \rbrace - \chi\lbrace \tilde{\boldsymbol\sigma}^j_b= \boldsymbol\alpha \rbrace \\- \Delta \sum_{i\in I_M}\big\lbrace c(-\alpha^i,\tilde{G}^{i,j}_b)\chi\lbrace \tilde{\boldsymbol\sigma}^j_b= \boldsymbol\alpha[i] \rbrace -c(\alpha^i,\tilde{G}^{i,j}_b)\chi\lbrace \tilde{\boldsymbol\sigma}^j_b= \boldsymbol\alpha \rbrace \big\rbrace\big) ,
\end{multline}
recalling that $\boldsymbol\alpha[i]$ is the same as $\boldsymbol\alpha$, except that the $i^{th}$ element has a flipped sign.

The sum of the terms $\beta^2 + \beta^3$ represents the leading order change in the two expectations due to changes in the field $\tilde{\mathbf{G}}_t$, while holding the spin to be constant. A Taylor approximation is used: $\beta^2$ contains the linear terms, and $\beta^3$ the quadratic terms,
\begin{align*}
\beta^2(\boldsymbol\alpha,\tilde{\boldsymbol\sigma},\tilde{\mathbf{G}})  =& N^{-1}\sum_{j\in I_N}\sum_{i\in I_M}\phi^a_i(\tilde{\mathbf{G}}^{j}_b)\chi\lbrace \tilde{\boldsymbol\sigma}^j_b= \boldsymbol\alpha \rbrace\big\lbrace \tilde{G}^{i,j}_{b+1} - \tilde{G}^{i,j}_b- \Delta m^{\hat{\mu}^N_b,i}(\boldsymbol\alpha,\tilde{\mathbf{G}}^j_b) \big\rbrace \\
\beta^3(\boldsymbol\alpha,\tilde{\boldsymbol\sigma},\tilde{\mathbf{G}})  =&  (2N)^{-1}\sum_{j\in I_N}\sum_{i,p\in I_M}\chi\lbrace \tilde{\boldsymbol\sigma}^{j}_b=\boldsymbol\alpha \rbrace\big(\phi^a_{ip}(\tilde{\mathbf{G}}_b^{j})\big\lbrace \tilde{G}^{i,j}_{b+1} - \tilde{G}^{i,j}_b \big\rbrace \big\lbrace \tilde{G}^{p,j}_{b+1} - \tilde{G}^{p,j}_b \big\rbrace -4L^{\hat{\mu}^N_b}_{ii}\Delta\big).
\end{align*}
$\beta^4$ can be thought of as the average `cross-variation' between the spins and the fields:
\begin{equation}
\beta^4(\boldsymbol\alpha,\tilde{\boldsymbol\sigma},\tilde{\mathbf{G}})  = N^{-1}\sum_{j\in I_N} \big\lbrace \phi^a(\tilde{\mathbf{G}}^{j}_{b+1})- \phi^a(\tilde{\mathbf{G}}^{j}_b)\big\rbrace\big\lbrace \chi \lbrace \tilde{\boldsymbol\sigma}^{j}_{b+1} = \boldsymbol\alpha \rbrace - \chi \lbrace \tilde{\boldsymbol\sigma}^{j}_b = \boldsymbol\alpha \rbrace \big\rbrace. \label{eq: beta 4 definition}
\end{equation}
The term $\beta^5$ is the remainder, such that \eqref{eq: alpha decomposition} holds identically. This means that
\begin{multline}\label{eq: beta 7 b}
\beta^5(\boldsymbol\alpha,\tilde{\boldsymbol\sigma},\tilde{\mathbf{G}}) = N^{-1}\sum_{j\in I_N}\chi\lbrace \tilde{\boldsymbol\sigma}^j_b= \boldsymbol\alpha\rbrace \big( \phi^a(\tilde{\mathbf{G}}_{b+1}^{j} )- \phi^a(\tilde{\mathbf{G}}_{b}^{j} )-\sum_{i\in I_M} \phi
^a_i(\tilde{\mathbf{G}}^{j}_b)\lbrace \tilde{G}^{i,j}_{b+1} - \tilde{G}^{i,j}_b \rbrace \\ -\frac{1}{2}\sum_{i,p\in I_M }\phi^a_{ip}(\tilde{\mathbf{G}}_b^{j})\big\lbrace \tilde{G}^{i,j}_{b+1} - \tilde{G}^{i,j}_b \big\rbrace \big\lbrace \tilde{G}^{p,j}_{b+1} - \tilde{G}^{p,j}_b \big\rbrace \big) .
\end{multline}
We further decompose $\beta^2$ and $\beta^3$ as follows. The terms $\beta^6$, $\beta^8$ and $\beta^9$ - to be outlined just below - will be bounded in Section \ref{Section Field} using the conditional Gaussian law of the connections. The term $\tilde{\mathbf{m}}^j_b$ - to be outlined just below - is the mean of a conditional Gaussian expectation, and $\tilde{\mathbf{L}}^{ij}_b$ is approximately half the conditional variance. Define the  $M\times M$ matrices $\lbrace \tilde{\mathbf{K}}_b,\tilde{\mathbf{L}}_b,\tilde{\boldsymbol\kappa}_b,\tilde{\boldsymbol\upsilon}_b \rbrace$ to have the following elements: for $p,q\in I_M$
\begin{align}
\tilde{K}^{pq}_b &= N^{-1}\sum_{l=1}^N \tilde{\sigma}^{p,l}_{b}\tilde{\sigma}^{q,l}_b \; \; , \; \; 
\tilde{L}^{pq}_b =  N^{-1}\sum_{k=1}^{N}\tilde{\sigma}^{q,k}_b\big(\tilde{\sigma}^{p,k}_{b}- \tilde{\sigma}^{p,k}_{b+1}\big)\label{eq: tilde K definition 0 0}  \\
\tilde{\kappa}^{pq}_b &= N^{-1}\sum_{k=1}^{N} \tilde{G}^{q,k}_b \big(\tilde{\sigma}^{p,k}_b-\tilde{\sigma}^{p,k}_{b+1}\big) \; \; , \; \;
\tilde{\upsilon}_b^{pq} = N^{-1}\sum_{k=1}^{N}\tilde{\sigma}^{p,k}_b \tilde{G}^{q,k}_b.\label{eq: tilde upsilon definition 0 0}
\end{align}
If $\tilde{\tau}_N > t^{(n)}_b$, $\tilde{\mathbf{K}}_b$ is invertible, and we write $\tilde{\mathbf{H}}_b = \tilde{\mathbf{K}}_b^{-1}$. For $j\in I_N$, writing $\tilde{\boldsymbol\sigma}^j_b = \big(\sigma^{1,j}_b,\ldots,\sigma^{M,j}_b\big)$, $\tilde{\mathbf{G}}^j_b = \big(\tilde{G}^{1,j}_b,\ldots,\tilde{G}^{M,j}_b \big)$ and $\tilde{\mathbf{m}}^j_b = \big(\tilde{m}^{1,j}_b,\ldots,\tilde{m}_b^{M,j}\big)$, we define
	\begin{align}
	\tilde{\mathbf{m}}^j_b =& - \tilde{\mathbf{L}}_b\tilde{\mathbf{H}}_b\tilde{\mathbf{G}}^j_b -\mathfrak{s} \tilde{\boldsymbol\kappa}_b\tilde{\mathbf{H}}_b\tilde{\boldsymbol\sigma}^j_b + \mathfrak{s}\tilde{\mathbf{L}}_b\tilde{\mathbf{H}}_b\tilde{\boldsymbol\upsilon}_b\tilde{\mathbf{H}}_b\tilde{\boldsymbol\sigma}^j_b .\label{eq: temp H 0}
	\end{align}
We can now further decompose $\beta^2$ as follows,
\begin{align}
\beta^2(\boldsymbol\alpha,\tilde{\boldsymbol\sigma},\tilde{\mathbf{G}})   =& \beta^6(\boldsymbol\alpha,\tilde{\boldsymbol\sigma},\tilde{\mathbf{G}}) + \beta^7(\boldsymbol\alpha,\tilde{\boldsymbol\sigma},\tilde{\mathbf{G}})  \text{ where } \\
\beta^6(\boldsymbol\alpha,\tilde{\boldsymbol\sigma},\tilde{\mathbf{G}})   =& N^{-1}\sum_{j\in I_N}\sum_{i\in I_M}\phi^a_i(\tilde{\mathbf{G}}^{j}_b)\chi\lbrace \tilde{\boldsymbol\sigma}_b^j=\boldsymbol\alpha \rbrace\big\lbrace \tilde{G}^{i,j}_{b+1} - \tilde{G}^{i,j}_b  -\tilde{m}^{i,j}_b\big\rbrace \label{eq: beta 6}\\ 
\beta^{7}(\boldsymbol\alpha,\tilde{\boldsymbol\sigma},\tilde{\mathbf{G}})   =& N^{-1}\sum_{j\in I_N}\sum_{i\in I_M}\phi^a_i(\tilde{\mathbf{G}}^{j}_b)\chi\lbrace \tilde{\boldsymbol\sigma}_b^j=\boldsymbol\alpha \rbrace\big\lbrace \tilde{m}^{i,j}_b -\Delta m^{\hat{\mu}^N_b,i}(\tilde{\boldsymbol\sigma}^j_b, \tilde{\mathbf{G}}^j_b) \big\rbrace ,
\end{align}
noting that $m^{\hat{\mu}^N_b}$ is defined in \eqref{eq: m xi}. We further decompose $\beta^3(\boldsymbol\alpha,\tilde{\boldsymbol\sigma},\tilde{\mathbf{G}}) $ as follows
\begin{align}
\beta^3(\boldsymbol\alpha,\tilde{\boldsymbol\sigma},\tilde{\mathbf{G}})  =& \beta^8(\boldsymbol\alpha,\tilde{\boldsymbol\sigma},\tilde{\mathbf{G}})+\beta^9(\boldsymbol\alpha,\tilde{\boldsymbol\sigma},\tilde{\mathbf{G}}) +\beta^{10}(\boldsymbol\alpha ,\tilde{\boldsymbol\sigma}_b,\tilde{\mathbf{G}}_b)+\beta^{11}(\boldsymbol\alpha ,\tilde{\boldsymbol\sigma}_b,\tilde{\mathbf{G}}_b) \text{ where}\nonumber\\ 
\beta^8 (\boldsymbol\alpha,\tilde{\boldsymbol\sigma},\tilde{\mathbf{G}}) =& N^{-1}\sum_{j\in I_{N}}\sum_{i,p \in I_M}\chi\big\lbrace \tilde{\boldsymbol\sigma}_b^j=\boldsymbol\alpha \big\rbrace\phi^a_{ip}(\tilde{\mathbf{G}}_b^{j})\tilde{m}^{i,j}_b \big\lbrace \tilde{G}^{p,j}_{b+1} - \tilde{G}^{p,j}_b-\tilde{m}^{p,j}_b\big\rbrace\label{eq: beta 8} \\
\beta^9(\boldsymbol\alpha,\tilde{\boldsymbol\sigma},\tilde{\mathbf{G}})  =&  (2N)^{-1}\sum_{j\in I_{N}}\sum_{i,p \in I_M}\chi\big\lbrace \tilde{\boldsymbol\sigma}_b^j=\boldsymbol\alpha \big\rbrace\phi^a_{ip}(\tilde{\mathbf{G}}_b^{j})\big\lbrace \tilde{G}^{i,j}_{b+1} - \tilde{G}^{i,j}_b-\tilde{m}^{i,j}_b \big\rbrace \big\lbrace \tilde{G}^{p,j}_{b+1} - \tilde{G}^{p,j}_b-\tilde{m}^{p,j}_b\big\rbrace \nonumber \\ &-(N)^{-1}\sum_{j\in I_{N}}\sum_{i \in I_M}\chi\big\lbrace \tilde{\boldsymbol\sigma}_b^j=\boldsymbol\alpha \big\rbrace \phi^a_{ii}(\tilde{\mathbf{G}}_b^{j})\tilde{L}^{ii}_b\label{eq: beta 9} \\
\beta^{10}(\boldsymbol\alpha,\tilde{\boldsymbol\sigma},\tilde{\mathbf{G}})  =&(2N)^{-1}\sum_{j\in I_N}\sum_{i,p\in I_M}\chi\big\lbrace \tilde{\boldsymbol\sigma}_b^j=\boldsymbol\alpha \big\rbrace\phi^a_{ip}(\tilde{\mathbf{G}}_b^{j}) \tilde{m}^{i,j}_b \tilde{m}^{p,j}_b   \\
\beta^{11}(\boldsymbol\alpha,\tilde{\boldsymbol\sigma},\tilde{\mathbf{G}}) = &N^{-1}\sum_{j\in I_{N}}\sum_{i\in I_M}\chi\big\lbrace \tilde{\boldsymbol\sigma}_b^j=\boldsymbol\alpha \big\rbrace\phi^a_{ii}(\tilde{\mathbf{G}}_b^{j})\big( \tilde{L}^{ii}_b -2\Delta L_{ii}^{\hat{\mu}^N_b}\big) .
\end{align}
We can now decompose the criteria of Lemma \ref{Main Lemma Taylor} into the following set of criteria.
\begin{lemma} \label{Lemma general beta bound}
To prove Lemma \ref{Lemma Change of Measure} it suffices for us to show that for any $\bar{\epsilon} > 0$, there exists $n_0$ such that for all $n \geq n_0$, there exists $\mathfrak{n}_0 \in \mathbb{Z}^+$ such that for all $\mathfrak{n} \geq \mathfrak{n}_0$, for each $i$ such that $1 \leq i \leq 11$ (with $i\neq 2,3$), (recalling that $\Delta = Tn^{-1}$) 
\begin{equation}
\sup_{0\leq b < n}\sup_{\phi^a\in \mathfrak{F},\boldsymbol\alpha \in \mathcal{E}^M}\lsup{N}N^{-1}\log \sup_{1\leq \mathfrak{q}\leq C^N_{\mathfrak{n}}}\mathbb{P}\big(\mathcal{J}_N,\tilde{\mu}^N(\tilde{\boldsymbol\sigma},\tilde{\mathbf{G}}) \in \mathcal{V}^N_{\mathfrak{q}},\tilde{\tau}_N > t^{(n)}_b,\big|\beta^i(\boldsymbol\alpha,\tilde{\boldsymbol\sigma},\tilde{\mathbf{G}}) \big| \geq \bar{\epsilon} \Delta / 18 \big) < 0.
\end{equation}
\end{lemma}
\begin{proof}
The above analysis implies that for large enough $n$,
\begin{multline}
\big\lbrace d_K\big(\hat{\mu}^N(\tilde{\boldsymbol\sigma}_{b+1},\tilde{\mathbf{G}}_{b+1}) , \xi(\tilde{\boldsymbol\sigma}_{b}, \tilde{\mathbf{G}}_{b})\big) \geq \bar{\epsilon}\Delta \big\rbrace \subseteq \bigcup_{1\leq i \leq 11, i\neq 2,3} \big\lbrace \big|\beta^i(\boldsymbol\alpha,\tilde{\boldsymbol\sigma},\tilde{\mathbf{G}}) \big| \geq \bar{\epsilon} \Delta / 18  \big\rbrace \\
\cup \big\lbrace \big|\Delta N^{-1}\sum_{j\in I_N,i\in I_M}\phi^a(\tilde{\mathbf{G}}^j_b)\chi\lbrace \tilde{\boldsymbol\sigma}^j_b= \boldsymbol\alpha[i] \rbrace c(-\alpha^i,\tilde{G}^{i,j}_b) +N^{-1}\sum_{j\in I_N} \chi\lbrace \tilde{\boldsymbol\sigma}^j_b= \boldsymbol\alpha \rbrace\big[\phi^a\big(\tilde{\boldsymbol\sigma}^j_b \big) + \Delta\sum_{i\in I_M}\big\lbrace  \\ \phi^a_i(\tilde{\mathbf{G}}^{j}_b)  m^{\hat{\mu}^N_b,i}(\boldsymbol\alpha,\tilde{\mathbf{G}}^j_b) +2L^{\hat{\mu}^N_b}_{ii}\phi^a_{ii}(\tilde{\mathbf{G}}^{j}_b)  -\phi^a\big(\tilde{\boldsymbol\sigma}^j_b \big)  c(\alpha^i,\tilde{G}^{i,j}_b)\big\rbrace \big] -R^{a,\boldsymbol\alpha}_{b} \big| \geq \bar{\epsilon} \Delta / 2 \big\rbrace.
\end{multline}
By Lemma \ref{Lemma Exp Empirical Measure}, as long as $\Delta$ is sufficiently small,
\begin{multline}
\big\lbrace \big| \Delta N^{-1}\sum_{j\in I_N,i\in I_M}\phi^a(\tilde{\mathbf{G}}^j_b)\chi\lbrace \tilde{\boldsymbol\sigma}^j_b= \boldsymbol\alpha[i] \rbrace c(-\alpha^i,\tilde{G}^{i,j}_b) +N^{-1}\sum_{j\in I_N} \chi\lbrace \tilde{\boldsymbol\sigma}^j_b= \boldsymbol\alpha \rbrace\big[\phi^a\big(\tilde{\boldsymbol\sigma}^j_b \big) + \Delta\sum_{i\in I_M}\big\lbrace  \\ \phi^a_i(\tilde{\mathbf{G}}^{j}_b)  m^{\hat{\mu}^N_b,i}(\boldsymbol\alpha,\tilde{\mathbf{G}}^j_b) +2L^{\hat{\mu}^N_b}_{ii}\phi^a_{ii}(\tilde{\mathbf{G}}^{j}_b)  -\phi^a\big(\tilde{\boldsymbol\sigma}^j_b \big)  c(\alpha^i,\tilde{G}^{i,j}_b)\big\rbrace \big]-R^{a,\boldsymbol\alpha}_{b} \big| \geq \bar{\epsilon} \Delta / 2 \big\rbrace = \emptyset
\end{multline}
The Lemma now follows as a consequence of Lemma \ref{Lemma Many Events}, since $|\mathfrak{F}| < \infty$. 
\end{proof}
The nine bounds necessary for Lemma \ref{Lemma general beta bound} are contained in the next two sections. They are split into two types: the terms directly requiring the law of the Gaussian connections (i.e. $\beta^6, \beta^8, \beta^9$) are bounded in Section \ref{Section Field}. The other six terms mostly require concentration inequalities for Poisson processes, and they are bounded in Section \ref{Section Stochastic Estimates}.
\section{Stochastic Bounds} \label{Section Stochastic Estimates}

This section is devoted to bounding the terms in Lemma \ref{Lemma general beta bound} that do not directly require the law of the Gaussian connections (i.e. $\gamma$). The terms that are bounded in the first part of this section are $\beta^4$ (the `cross-variation' of the spins and fields) and $\beta^5$ (the remainder after the Taylor Expansion). In the next subsection, the remaining terms $\beta^1,\beta^7,\beta^{10},\beta^{11}$ are bounded: the bounding of these terms requires concentration inequalities for sums of compensated Poisson Processes. Throughout this section we omit the $\mathfrak{q}$ subscript from the stochastic process, writing $\tilde{\boldsymbol\sigma}_{\mathfrak{q},t} := \tilde{\boldsymbol\sigma}_t$.

Throughout this section $\boldsymbol\alpha \in \mathcal{E}^M$ is a fixed constant. Define for $ u \geq 0$, 
\begin{equation}\label{eq: Y i j b definition}
Y_b^{i,j}(u) = Y^{i,j}\bigg(u + \int_0^{t^{(n)}_b} c\big(\tilde{\sigma}^{i,j}_s , \tilde{\mathfrak{G}}^{i,j}_{\mathfrak{q},s}(\tilde{\boldsymbol\sigma}) \big) ds \bigg)-Y^{i,j}\bigg(\int_0^{t^{(n)}_b} c\big(\tilde{\sigma}^{i,j}_s , \tilde{\mathfrak{G}}^{i,j}_{\mathfrak{q},s}(\tilde{\boldsymbol\sigma})\big) ds \bigg),
\end{equation}
and notice that $\lbrace Y_b^{i,j}(t) \rbrace_{i\in I_M,j\in I_N}$ are distributed as iid unit intensity Poisson Processes. Recalling that $A\cdot x := (-1)^x$, it may be inferred from the definition in \eqref{eq: tilde sigma q j t definition} that for $t\geq t^{(n)}_b$,
\begin{equation}\label{eq: tilde sigma alternative definition}
\tilde{\sigma}^{i,j}_t = \tilde{\sigma}^{i,j}_b A\cdot Y_b^{i,j}\big(\int_{t^{(n)}_b}^{t} c(\tilde{\sigma}^{i,j}_s, \tilde{\mathfrak{G}}^{i,j}_{\mathfrak{q},s}) ds \big).
\end{equation}
Let $\mathcal{I}_N = \big\lbrace j\in I_N \; : \text{ For some }i\in I_M \; , \; Y^{i,j}_b(c_1 \Delta ) \geq 1 \big\rbrace $, recalling that $c_1$ is the uniform upper bound for the spin flipping rate. Clearly if $j\notin \mathcal{I}_N$ then $\tilde{\boldsymbol\sigma}^j_b = \tilde{\boldsymbol\sigma}^j_{b+1}$. Write $\mathcal{I}_N^c = \lbrace j\in I_N : j\notin \mathcal{I}_N \rbrace$. Splitting the indices as $I_N = \mathcal{I}_N \cup \mathcal{I}_N^c$ is useful because the fields $\lbrace \tilde{G}^{p,j}_{b+1} - \tilde{G}^{p,j}_b \rbrace_{j\in \mathcal{I}_N^c}$ are independent.

We start with a lemma concerning the average change in fields indexed by $\mathcal{I}_N$.
\begin{lemma}\label{Lemma square norm G difference}
There exists $n_0 \in \mathbb{Z}^+$ and a constant $\hat{C}_{\gamma}$ such that for all $n \geq n_0$,
\begin{equation}
\lsup{N}\sup_{1\leq \mathfrak{q}\leq C^N_{\mathfrak{n}}} N^{-1} \log \mathbb{P}\big( N^{-1}\sum_{j\in \mathcal{I}_N}\norm{\tilde{\mathbf{G}}^{j}_{b+1} - \tilde{\mathbf{G}}^{j}_b}^2 \geq \Delta^{3/2}\hat{C}_{\gamma}\big) < 0.
\end{equation}
\end{lemma}
\begin{proof}
Let $\tilde{\mathbf{J}}_N$ be the $|\mathcal{I}_N | \times  |\mathcal{I}_N |$ square matrix with entries given by $\lbrace N^{-1/2} J^{jk} \rbrace_{j,k \in \mathcal{I}_N}$. Let its operator norm be $\norm{\tilde{\mathbf{J}}_N}$. Observe that
\begin{align}
N^{-1}\sum_{j\in \mathcal{I}_N,i\in I_M}|\tilde{G}^{i,j}_{b+1} - \tilde{G}^{i,j}_b|^2 &\leq N^{-1}\norm{\tilde{\mathbf{J}}_N} \sum_{j\in \mathcal{I}_N \; , i\in I_M}\big|\tilde{\sigma}^{i,j}_{b+1} - \tilde{\sigma}^{i,j}_b\big|^2\nonumber \\ &\leq  4MN^{-1}\norm{\tilde{\mathbf{J}}_N} \sum_{j\in \mathcal{I}_N}\chi\big\lbrace \text{For some }i\in I_M \; ,  \tilde{\sigma}^{i,j}_{b+1} \neq \tilde{\sigma}^{i,j}_b\big\rbrace .
\end{align}
Writing $\hat{C}_{\gamma} = 6\sqrt{2} c_1^{3/2}M(M+1)$, we observe that if  $ \norm{\tilde{\mathbf{J}}_N} 
\leq 3\sqrt{c_1\Delta / 2}$ and $N^{-1}\sum_{j\in \mathcal{I}_N}\chi\big\lbrace \text{For some }i\in I_M \; ,  \tilde{\sigma}^{i,j}_{b+1} \neq \tilde{\sigma}^{i,j}_b\big\rbrace \leq c_1 \Delta( M+1)$ then $N^{-1}\sum_{j\in \mathcal{I}_N}\norm{\tilde{G}^{i,j}_{b+1} - \tilde{G}^{i,j}_b}^2 \leq \hat{C}_{\gamma}\Delta^{3/2}$. We thus find that,
\begin{multline*}
\big\lbrace N^{-1}\sum_{j\in \mathcal{I}_N}\norm{\tilde{G}^{i,j}_{b+1} - \tilde{G}^{i,j}_b}^2 >  \Delta^{3/2}\hat{C}_{\gamma} \big\rbrace \subseteq
\big\lbrace N^{-1} |\mathcal{I}_N|  \notin  [c_1\Delta / 2 , c_1 \Delta( M+1)] \big\rbrace \\ \cup \big\lbrace N^{-1} |\mathcal{I}_N|  \in  [c_1\Delta / 2 , c_1 \Delta( M+1)] \text{ and }\norm{\tilde{\mathbf{J}}_N} > 3\sqrt{c_1\Delta / 2}\big\rbrace.
\end{multline*}
It follows from basic properties of Poisson Processes (noted in Lemma \ref{Lemma Upper Bound Jump Rate}) that the probability of the first term on the right hand side is exponentially decaying. It thus remains to prove that 
\begin{equation}
\lsup{N} N^{-1}\log \mathbb{P}\big( |\mathcal{I}_N|  \in  [c_1\Delta / 2 , c_1 \Delta( M+1)] \text{ and }\norm{\tilde{\mathbf{J}}_N} > 3\sqrt{c_1\Delta / 2} \big) < 0.
\end{equation}
Define $\bar{\mathbf{J}}_N$ to be the $|\mathcal{I}_N| \times |\mathcal{I}_N| $ square matrix with elements $|\mathcal{I}_N|^{-\frac{1}{2}} J^{jk}$: that is, $\bar{\mathbf{J}}_N =\sqrt{N}  |\mathcal{I}_N|^{-\frac{1}{2}}\tilde{\mathbf{J}}_N$. This means that 
\[
\big\lbrace \norm{\tilde{\mathbf{J}}_N} \geq 3\sqrt{c_1\Delta / 2} \text{ and } N^{-1} |\mathcal{I}_N|  \in [c_1\Delta / 2 , c_1 \Delta( M+1)]  \big\rbrace \subseteq
\big\lbrace \norm{\bar{\mathbf{J}}_N} \geq 3 \text{ and } N^{-1} |\mathcal{I}_N|  \in [c_1\Delta / 2 , c_1 \Delta( M+1)]  \big\rbrace.
\]
Notice that the (random) indices in $\mathcal{I}_N$ are independent of the static connections $\lbrace J^{jk} \rbrace_{j,k \in \mathbb{Z}^+}$ - since the Poisson Processes $\lbrace Y^{i,j}(t)\rbrace$ are Markovian and independent of the static connections. We can now use known bounds on the dominant eigenvalue of random matrices (as noted in (3) of Lemma \ref{bound connections lemma}) to obtain that
\begin{align}
\lsup{N}N^{-1}\log\mathbb{P}\big( & \norm{\bar{\mathbf{J}}_N} \geq 3  \text{ and } N^{-1} |\mathcal{I}_N|  \in [c_1\Delta / 2 , c_1 \Delta( M+1)]  \big)\nonumber \\
&= \lsup{N}N^{-1}\log\mathbb{E}\big[\mathbb{P}\big( \norm{\bar{\mathbf{J}}_N} \geq 3 \text{ and }N^{-1} |\mathcal{I}_N|  \in [c_1\Delta / 2 , c_1 \Delta( M+1)]  \; | \; \mathbf{Y}(t) \big)\big]\nonumber \\
&\leq  \lsup{N}N^{-1}\log\mathbb{E}\big[\chi\lbrace |\mathcal{I}_N| \geq Nc_1 \Delta / 2  |\rbrace \exp\big(-|\mathcal{I}_N| \Lambda_J \big)\big]\nonumber  \\ &< 0,\label{eq: bound tilde J matrix norm}
\end{align}
as required, where the constant $\Lambda_J$ is defined in Lemma \ref{bound connections lemma}.
\end{proof}

We start with the bound of $\beta^4$ (which is defined in \eqref{eq: beta 4 definition}): this can be thought of as the average `cross-variation' between the spins and fields over the small time interval $\Delta$.

\begin{lemma}
For any $\bar{\epsilon} > 0$, for sufficiently large $n$,
\[
\lsup{N} \sup_{1\leq \mathfrak{q}\leq C^N_{\mathfrak{n}}} N^{-1}\log \mathbb{P}\big(\tilde{\mu}^N \in \mathcal{V}^N_{\mathfrak{q}}, \big|\beta^4\big(\boldsymbol\alpha, \tilde{\boldsymbol\sigma}_b,\tilde{\mathbf{G}}\big)\big| \geq \bar{\epsilon}\Delta \big) < 0.
\]
\end{lemma}
\begin{proof}

Now for some $z > 0$,
\begin{multline}
\lsup{N}N^{-1}\log \mathbb{P}\big(\tilde{\mu}^N \in \mathcal{V}^N_{\mathfrak{q}},\big|\beta^4\big(\boldsymbol\alpha, \tilde{\boldsymbol\sigma},\tilde{\mathbf{G}}\big)\big| \geq \bar{\epsilon} \Delta\big) 
\leq \max\big\lbrace  \lsup{N}N^{-1}\log \mathbb{P}\big(\big|\mathcal{I}_N\big| > N (Mc_1 \Delta + z)  \big) ,\\ \lsup{N} \sup_{1\leq \mathfrak{q}\leq C^N_{\mathfrak{n}}}N^{-1}\log  \mathbb{P}\big(  \big|\beta^4\big(\boldsymbol\alpha, \tilde{\boldsymbol\sigma},\tilde{\mathbf{G}}\big)\big|\geq \bar{\epsilon}\Delta ,\big|\mathcal{I}_N\big| \leq N ( Mc_t \Delta + z) ,\tilde{\mu}^N \in \mathcal{V}^N_{\mathfrak{q}} \big) \big\rbrace \label{eq: tmp Q n beta 6}
\end{multline}
By Lemma \ref{Lemma Upper Bound Jump Rate},
\[
 \lsup{N}N^{-1}\log \mathbb{P}\big(\big|\mathcal{I}_N\big| > N ( M c_1 \Delta + z)  \big) < 0.
\]
It remains to prove that the second term on the right-hand-side of \eqref{eq: tmp Q n beta 6} is negative. Thanks to the identity in \eqref{eq: tilde sigma alternative definition}, if $Y^{i,j}_b(c_1 \Delta+ z ) = 0$ for all $i\in I_M$ then $\tilde{\boldsymbol\sigma}_{b+1} = \tilde{\boldsymbol\sigma}_b$ and $\chi \lbrace \tilde{\boldsymbol\sigma}^{j}_{b+1} = \boldsymbol\alpha \rbrace - \chi \lbrace \tilde{\boldsymbol\sigma}^{j}_b = \boldsymbol\alpha \rbrace  = 0$. We thus have that
\begin{align}
 \beta^4\big(\boldsymbol\alpha, \tilde{\boldsymbol\sigma}_b,\tilde{\mathbf{G}}\big)  =& N^{-1} \sum_{j\in \mathcal{I}_N}\big\lbrace \phi^a(\tilde{\mathbf{G}}^{j}_{b+1})- \phi^a(\tilde{\mathbf{G}}^{j}_b)\big\rbrace\big\lbrace \chi (\tilde{\boldsymbol\sigma}^{j}_{b+1} = \boldsymbol\alpha ) - \chi (\tilde{\boldsymbol\sigma}^{j}_b = \boldsymbol\alpha ) \big\rbrace \label{eq: beta 6 Taylor 0} \\
=& N^{-1}\sum_{j\in \mathcal{I}_N \; i\in I_M} \phi^a_i(\bar{\mathbf{G}}^{j})(\tilde{\mathbf{G}}^{i,j}_{b+1} -\tilde{\mathbf{G}}^{i,j}_{b}  )\big\lbrace \chi (\tilde{\boldsymbol\sigma}^{j}_{b+1} = \boldsymbol\alpha ) - \chi (\tilde{\boldsymbol\sigma}^{j}_b = \boldsymbol\alpha ) \big\rbrace,\label{eq: beta 6 Taylor}
\end{align}
for $\bar{\mathbf{G}}^j = \lambda_j \tilde{\mathbf{G}}^j_b + (1-\lambda_j)\tilde{\mathbf{G}}^j_{b+1}$, for some $\lambda_j \in [0,1]$, by the Taylor Remainder Theorem. It follows from \eqref{eq: beta 6 Taylor 0} that if $|\mathcal{I}_N| < N\bar{\epsilon}\Delta / 4$, then since $|\phi^a| \leq 1$, it must necessarily be the case that
\[
\big|  \beta^4\big(\boldsymbol\alpha, \tilde{\boldsymbol\sigma},\tilde{\mathbf{G}}\big) \big| < \bar{\epsilon} \Delta, 
\]
as required. It thus suffices for us to prove that 
\begin{equation}
\lsup{N} N^{-1}\log \mathbb{P}\big(\tilde{\mu}^N \in \mathcal{V}^N_{\mathfrak{q}},\big|  \beta^4\big(\boldsymbol\alpha, \tilde{\boldsymbol\sigma},\tilde{\mathbf{G}}\big) \big| \geq \bar{\epsilon}\Delta\text{ and }\big|\mathcal{I}_N\big| \in [N\bar{\epsilon}\Delta / 4, N ( Mc_1 \Delta + z) ] \big) < 0.
\end{equation}
Using a union-of-events bound,
\begin{multline*}
\lsup{N} N^{-1}\log \mathbb{P}\big( \big|  \beta^4\big(\boldsymbol\alpha, \tilde{\boldsymbol\sigma}_b,\tilde{\mathbf{G}}\big) \big| \geq \bar{\epsilon}\Delta\text{ and }\big|\mathcal{I}_N\big| \in [N\bar{\epsilon}\Delta / 4, N ( Mc_1 \Delta + z) ] \big)  \\
\leq \max\big\lbrace \lsup{N}N^{-1}\log \mathbb{P}\big( \norm{\tilde{\mathbf{J}}_N} \geq 3 |\mathcal{I}_N|^{\frac{1}{2}}N^{-1/2} \text{ and }|\mathcal{I}_N| \geq N\bar{\epsilon}\Delta / 4\big) , \\
\lsup{N}N^{-1}\log \mathbb{P}\big( \big|  \beta^4\big(\boldsymbol\alpha, \tilde{\boldsymbol\sigma}_b,\tilde{\mathbf{G}}\big) \big| \geq \bar{\epsilon}\Delta\text{ and }\big|\mathcal{I}_N\big| \in [N\bar{\epsilon}\Delta / 4, N ( Mc_1 \Delta + z) ] \text{ and }\norm{\tilde{\mathbf{J}}} \leq 3 |\mathcal{I}_N|^{\frac{1}{2}}N^{-1/2}  \big) \big\rbrace,
\end{multline*}
where $\tilde{\mathbf{J}}_N$ is the $|\mathcal{I}_N | \times  |\mathcal{I}_N |$ square matrix with entries given by $\lbrace N^{-1/2} J^{jk} \rbrace_{j,k \in \mathcal{I}_N}$. Just as we proved in \eqref{eq: bound tilde J matrix norm}, for small enough $\Delta$,
\begin{align*}
 \lsup{N}N^{-1}\log \mathbb{P}\big( \norm{\tilde{\mathbf{J}}_N} \geq 3 |\mathcal{I}_N|^{\frac{1}{2}}N^{-1/2} \text{ and }|\mathcal{I}_N| \geq N\bar{\epsilon}\Delta / 4\big)  
 \leq  \lsup{N}N^{-1}\log \mathbb{P}\big( \norm{\tilde{\mathbf{J}}_N} \geq 3 N\bar{\epsilon}\Delta / 4\big)  < 0.
\end{align*}
It thus suffices for us to prove that for $\Delta$ sufficiently small,
\begin{equation}\label{eq: to show beta 6}
\lsup{N} N^{-1}\log \mathbb{P}\big(  \big|  \beta^4\big(\boldsymbol\alpha, \tilde{\boldsymbol\sigma}_b,\tilde{\mathbf{G}}\big) \big| \geq \bar{\epsilon}\Delta\text{ and }\big|\mathcal{I}_N\big| \in [N\bar{\epsilon}\Delta / 4, N ( Mc_1 \Delta + z) ] \text{ and }\norm{\tilde{\mathbf{J}}} \leq 3 |\mathcal{I}_N|^{\frac{1}{2}}N^{-1/2}  \big) < 0.
\end{equation}
 To this end, we obtain from \eqref{eq: beta 6 Taylor} that, since $|\phi^a_i | \leq 1$, by the Cauchy-Schwarz Inequality,
\begin{align*}
 \big|  \beta^4\big(\boldsymbol\alpha, \tilde{\boldsymbol\sigma}_b,\tilde{\mathbf{G}}\big) \big|^2 &\leq N^{-2}\sum_{j\in \mathcal{I}_N \; i\in I_M}(\tilde{\mathbf{G}}^{i,j}_{b+1} -\tilde{\mathbf{G}}^{i,j}_{b}  )^2\times \sum_{j\in \mathcal{I}_N} \big\lbrace \chi (\tilde{\boldsymbol\sigma}^{j}_{b+1} = \boldsymbol\alpha ) - \chi (\tilde{\boldsymbol\sigma}^{j}_b = \boldsymbol\alpha ) \big\rbrace^2 \\
&\leq N^{-2}\sum_{j\in \mathcal{I}_N \; i\in I_M}\norm{\tilde{\mathbf{J}}_N}\big(\tilde{\sigma}^{i,j}_{b+1} - \tilde{\sigma}^{i,j}_b\big)^2 \times  \sum_{j\in \mathcal{I}_N} \big\lbrace \chi (\tilde{\boldsymbol\sigma}^{j}_{b+1} = \boldsymbol\alpha ) - \chi (\tilde{\boldsymbol\sigma}^{j}_b = \boldsymbol\alpha ) \big\rbrace^2.
\end{align*}
Now if $ |\mathcal{I}_N| \leq N (M c_1 \Delta + z)$ and $\norm{\tilde{\mathbf{J}}_N}\leq 3 |\mathcal{I}_N|^{\frac{1}{2}}N^{-1/2} \leq 3 (Mc_1 \Delta + z )^{\frac{1}{2}}$, it must be that, (since $(\tilde{\sigma}^{i,j}_{b+1} - \tilde{\sigma}^{i,j}_b)^2 \leq 4$),
\begin{align}
 \big|  \beta^4\big(\boldsymbol\alpha, \tilde{\boldsymbol\sigma}_b,\tilde{\mathbf{G}}\big) \big|^2 \leq 16 N^{-2} |\mathcal{I}_N|^2 3 \big(M c_1 \Delta + z \big)^{\frac{1}{2}}.\label{eq: beta 6 temporary}
\end{align}
We choose $z = \Delta$, and find that \eqref{eq: beta 6 temporary} implies that 
\begin{align*}
 \big|  \beta^4\big(\boldsymbol\alpha, \tilde{\boldsymbol\sigma}_b,\tilde{\mathbf{G}}\big) \big|^2 \leq \text{Const}\times (\Delta)^{\frac{5}{2}} .
\end{align*}
This means that for $\Delta$ sufficiently small,
\[
\mathbb{P}\big(  \big|  \beta^4\big(\boldsymbol\alpha, \tilde{\boldsymbol\sigma}_b,\tilde{\mathbf{G}}\big) \big| \geq \bar{\epsilon}\Delta\text{ and }\big|\mathcal{I}_N\big| \in [N\bar{\epsilon}\Delta, N ( M c_1 \Delta + z) ] \text{ and }\norm{\tilde{\mathbf{J}}} \leq 3 |\mathcal{I}_N|^{\frac{1}{2}}  \big) = 0,
\]
which implies \eqref{eq: to show beta 6}, as required.

\end{proof}

\begin{lemma} \label{Lemma beta 5 bound}
For any $\bar{\epsilon} > 0$, for large enough $n$
\begin{equation}
\sup_{0\leq b < n}\sup_{\boldsymbol\alpha \in \mathcal{E}^M}\lsup{N}N^{-1}\log \sup_{1\leq \mathfrak{q}\leq C^N_{\mathfrak{n}}}\mathbb{P}\big(\mathcal{J}_N,\tilde{\mu}^N \in \mathcal{V}^N_{\mathfrak{q}},\tilde{\tau}_N > t^{(n)}_b,\big|\beta^5(\boldsymbol\alpha,\tilde{\boldsymbol\sigma},\tilde{\mathbf{G}}) \big| \geq \bar{\epsilon} \Delta / 18 \big) < 0
\end{equation}
\end{lemma}
\begin{proof}
Recall the definition of $\beta^5$:
\begin{multline}\label{eq: beta 5 definition}
\beta^5(\boldsymbol\alpha,\tilde{\boldsymbol\sigma},\tilde{\mathbf{G}}) = N^{-1}\sum_{j\in I_N}\chi\lbrace \tilde{\boldsymbol\sigma}^j_b= \boldsymbol\alpha\rbrace \big( \phi^a(\tilde{\mathbf{G}}_{b+1}^{j} )- \phi^a(\tilde{\mathbf{G}}_{b}^{j} )-\sum_{i\in I_M} \phi
^a_i(\tilde{\mathbf{G}}^{j}_b)\lbrace \tilde{G}^{i,j}_{b+1} - \tilde{G}^{i,j}_b \rbrace \\ -\frac{1}{2}\sum_{i,p\in I_M }\phi^a_{ip}(\tilde{\mathbf{G}}_b^{j})\big\lbrace \tilde{G}^{i,j}_{b+1} - \tilde{G}^{i,j}_b \big\rbrace \big\lbrace \tilde{G}^{p,j}_{b+1} - \tilde{G}^{p,j}_b \big\rbrace \big) .
\end{multline}
If $\sup_{i\in I_M}|\tilde{G}^{i,j}_{b+1} - \tilde{G}^{i,j}_b| \leq (\Delta)^{2/5}$, then it follows from Taylor's Theorem that
\begin{multline*}
\big|\chi\lbrace \tilde{\boldsymbol\sigma}^j_b= \boldsymbol\alpha\rbrace \big( \phi^a(\tilde{\mathbf{G}}_{b+1}^{j} )- \phi^a(\tilde{\mathbf{G}}_{b}^{j} )-\sum_{i\in I_M} \phi
^a_i(\tilde{\mathbf{G}}^{j}_b)\lbrace \tilde{G}^{i,j}_{b+1} - \tilde{G}^{i,j}_b \rbrace \\ -\frac{1}{2}\sum_{i,p\in I_M }\phi^a_{ip}(\tilde{\mathbf{G}}_b^{j})\big\lbrace \tilde{G}^{i,j}_{b+1} - \tilde{G}^{i,j}_b \big\rbrace \big\lbrace \tilde{G}^{p,j}_{b+1} - \tilde{G}^{p,j}_b \big\rbrace \big) \big| \leq \frac{1}{6}\big|\sum_{i,p,q\in I_M} \phi^a_{ipq}(\hat{\mathbf{G}}^j)|\tilde{G}^{i,j}_{b+1} - \tilde{G}^{i,j}_b||\tilde{G}^{p,j}_{b+1} - \tilde{G}^{p,j}_b||\tilde{G}^{q,j}_{b+1} - \tilde{G}^{q,j}_b| \big| \\
\leq M^3(\Delta)^{6/5} / 6 \leq \Delta \bar{\epsilon} / 36,
\end{multline*}
once $n$ is sufficiently large (since $\Delta = Tn^{-1}$). Write $\tilde{I}_N = \big\lbrace j\in I_N: \sup_{i\in I_M}|\tilde{G}^{i,j}_{b+1} - \tilde{G}^{i,j}_b| \leq (\Delta)^{2/5}\big\rbrace$. Since the magnitude of $\phi$ and its first three derivatives are all upperbounded by $1$, it must be that there is a constant $C$ such that
\begin{multline}
\big| \chi\lbrace \tilde{\boldsymbol\sigma}^j_b= \boldsymbol\alpha\rbrace \big( \phi^a(\tilde{\mathbf{G}}_{b+1}^{j} )- \phi^a(\tilde{\mathbf{G}}_{b}^{j} )-\sum_{i\in I_M} \phi
^a_i(\tilde{\mathbf{G}}^{j}_b)\lbrace \tilde{G}^{i,j}_{b+1} - \tilde{G}^{i,j}_b \rbrace -\frac{1}{2}\sum_{i,p\in I_M }\phi^a_{ip}(\tilde{\mathbf{G}}_b^{j})\big\lbrace \tilde{G}^{i,j}_{b+1} - \tilde{G}^{i,j}_b \big\rbrace \big\lbrace \tilde{G}^{p,j}_{b+1} - \tilde{G}^{p,j}_b \big\rbrace \big)\big| \\ \leq C\big(1 + \sum_{p\in I_M} \big|\tilde{G}^{p,j}_{b+1} - \tilde{G}^{p,j}_b\big|^2 \big).
\end{multline}
The previous two equations imply that
\begin{multline}
\big|  \beta^5\big(\boldsymbol\alpha, \tilde{\boldsymbol\sigma}_b,\tilde{\mathbf{G}}_b\big) \big| \leq  \Delta \bar{\epsilon} / 36 + CN^{-1}\sum_{j\in I_N}\big[\chi\big\lbrace \sup_{i\in I_M}|\tilde{G}^{i,j}_{b+1} - \tilde{G}^{i,j}_b| > (\Delta)^{2/5} \big\rbrace \big(1 + \sum_{p\in I_M} \big|\tilde{G}^{p,j}_{b+1} - \tilde{G}^{p,j}_b\big|^2\big)\big].
\end{multline}
It thus suffices to prove that
\begin{equation*}
\lsup{N}N^{-1}\log \sup_{1\leq \mathfrak{q}\leq C^N_{\mathfrak{n}}}\mathbb{P}\big(\mathcal{J}_N, CN^{-1}\sum_{j\in I_N}\big[\chi\big\lbrace \sup_{i\in I_M}|\tilde{G}^{i,j}_{b+1} - \tilde{G}^{i,j}_b| > (\Delta)^{2/5} \big\rbrace \big(1 + \sum_{p\in I_M} \big|\tilde{G}^{p,j}_{b+1} - \tilde{G}^{p,j}_b\big|^2\big)\big]  > \Delta \bar{\epsilon} / 36 \big) < 0.
\end{equation*}
To establish the above equation, it suffices in turn to prove that (writing $\breve{\epsilon} = \bar{\epsilon} / 108$),
\begin{align}
&\lsup{N}N^{-1}\log \sup_{1\leq \mathfrak{q}\leq C^N_{\mathfrak{n}}}\mathbb{P}\big(\mathcal{J}_N, CN^{-1}\sum_{j\in \mathcal{I}_N}\big[\chi\big\lbrace \sup_{i\in I_M}|\tilde{G}^{i,j}_{b+1} - \tilde{G}^{i,j}_b| > (\Delta)^{2/5} \big\rbrace \big(1 + \sum_{p\in I_M} \big|\tilde{G}^{p,j}_{b+1} - \tilde{G}^{p,j}_b\big|^2\big)\big]  >\Delta \breve{\epsilon} \big) < 0\label{eq: to establish beta 5 1} \\
&\lsup{N}N^{-1}\log \sup_{1\leq \mathfrak{q}\leq C^N_{\mathfrak{n}}}\mathbb{P}\big(\mathcal{J}_N, CN^{-1}\sum_{j \in \mathcal{I}_N^c}\big[\chi\big\lbrace \sup_{i\in I_M}|\tilde{G}^{i,j}_{b+1} - \tilde{G}^{i,j}_b| > (\Delta)^{2/5} \big\rbrace  \sum_{p\in I_M} \big|\tilde{G}^{p,j}_{b+1} - \tilde{G}^{p,j}_b\big|^2\big]  >\Delta \breve{\epsilon} \big) < 0 \label{eq: to establish beta 5 2}\\
&\lsup{N}N^{-1}\log \sup_{1\leq \mathfrak{q}\leq C^N_{\mathfrak{n}}}\mathbb{P}\big(\mathcal{J}_N, CN^{-1}\sum_{j \in \mathcal{I}_N^c}\big[\chi\big\lbrace \sup_{i\in I_M}|\tilde{G}^{i,j}_{b+1} - \tilde{G}^{i,j}_b| > (\Delta)^{2/5} \big\rbrace \big]  >\Delta \breve{\epsilon} \big) < 0,\label{eq: to establish beta 5 3}
\end{align}
and we recall that $\mathcal{I}_N^c = \lbrace j\in I_N : j\notin \mathcal{I}_N \rbrace$. Starting with \eqref{eq: to establish beta 5 1}, observe that
\begin{multline}
\big\lbrace N^{-1}\sum_{j\in \mathcal{I}_N}\big[\chi\big\lbrace \sup_{i\in I_M}|\tilde{G}^{i,j}_{b+1} - \tilde{G}^{i,j}_b| > (\Delta)^{2/5} \big\rbrace \big(1 + \sum_{p\in I_M} \big|\tilde{G}^{p,j}_{b+1} - \tilde{G}^{p,j}_b\big|^2\big)\big]  >\Delta \breve{\epsilon}C^{-1} \big\rbrace \\ 
\subseteq \big\lbrace  N^{-1}\sum_{j\in \mathcal{I}_N}\sum_{p\in I_M} \big|\tilde{G}^{p,j}_{b+1} - \tilde{G}^{p,j}_b\big|^2 \big(\Delta^{-1/5} +1 \big) \big)  >\Delta \breve{\epsilon} /C \big\rbrace 
\end{multline}
The probability of the right hand side is exponentially decaying (for large enough $n$), as a consequence of Lemma \ref{Lemma square norm G difference}. The inequalities \eqref{eq: to establish beta 5 2} and \eqref{eq: to establish beta 5 3} are established in Lemma \ref{eq: bound square norm G}.

\end{proof}
\begin{lemma}\label{eq: bound square norm G}
For any $\epsilon > 0$, for sufficiently large $n\in \mathbb{Z}^+$ (and recalling that $\Delta = T/n$),
\begin{align*}
\lsup{N} \sup_{1\leq \mathfrak{q}\leq C^N_{\mathfrak{n}}} N^{-1}\log \mathbb{P}\big( N^{-1}\sum_{j\in \mathcal{I}^c_N}\chi\big\lbrace \big| \tilde{G}^{i,j}_{b+1} - \tilde{G}^{i,j}_b\big| \geq (\Delta)^{2/5} \text{ for some }i\in I_M\big\rbrace  \geq \epsilon \Delta\big) &<0 \\
\lsup{N} \sup_{1\leq \mathfrak{q}\leq C^N_{\mathfrak{n}}} N^{-1}\log \mathbb{P}\big( N^{-1}\sum_{j\in \mathcal{I}^c_N}\chi\big\lbrace \big| \tilde{G}^{i,j}_{b+1} - \tilde{G}^{i,j}_b\big| \geq (\Delta)^{2/5} \text{ for some }i\in I_M\big\rbrace \big| \tilde{G}^{i,j}_{b+1} - \tilde{G}^{i,j}_b\big|^2 \geq \epsilon \Delta\big) &<0 .
\end{align*}
\end{lemma}
\begin{proof}
The proofs are very similar and so we only include the second result. It follows from Lemma \ref{Lemma Many Events} that
\begin{multline}
\lsup{N} N^{-1}\log \mathbb{P}\big( N^{-1}\sum_{j\in \mathcal{I}^c_N}\chi\big\lbrace \big| \tilde{G}^{i,j}_{b+1} - \tilde{G}^{i,j}_b\big| \geq (\Delta)^{2/5} \text{ for some }i\in I_M\big\rbrace \big| \tilde{G}^{i,j}_{b+1} - \tilde{G}^{i,j}_b\big|^2 \geq \epsilon \Delta\big)\\ \leq \max\big\lbrace  \lsup{N} N^{-1}\log \mathbb{P}\big( N^{-1}\sum_{j\in \mathcal{I}^c_N,p\in I_M}\chi\big\lbrace \big| \tilde{G}^{i,j}_{b+1} - \tilde{G}^{i,j}_b\big| \geq (\Delta)^{2/5} \text{ for some }i\in I_M\big\rbrace \big| \tilde{G}^{p,j}_{b+1} - \tilde{G}^{p,j}_b\big|^2 \geq \epsilon \Delta \\ \text{ and }N^{-1} |\mathcal{I}_N|  \in [c_1\Delta / 2 , c_1 \Delta( M+1)]\big) , 
\lsup{N} N^{-1}\log \mathbb{P}\big(N^{-1} |\mathcal{I}_N|  \notin [c_1\Delta / 2 , c_1 \Delta( M+1)] \big) \big\rbrace.
\end{multline}
By Lemma \ref{Lemma Upper Bound Jump Rate}, 
$\lsup{N} N^{-1}\log \mathbb{P}\big(N^{-1} |\mathcal{I}_N|  \notin [c_1\Delta / 2 , c_1 \Delta( M+1)] \big) < 0$ as required. 

We write $\tilde{H}^{i,j} = \tilde{G}^{i,j}_{b+1} - \tilde{G}^{i,j}_b$. By Chernoff's Inequality,
\begin{multline}
 \mathbb{P}\big( N^{-1}\sum_{j\in \mathcal{I}^c_N,p\in I_M}\chi\big\lbrace \big| \tilde{H}^{i,j}\big| \geq (\Delta)^{2/5} \text{ for some }i\in I_M\big\rbrace \big| \tilde{H}^{p,j}\big|^2 \geq \epsilon \Delta \\ \text{ and }N^{-1} |\mathcal{I}_N|  \in [c_1\Delta / 2 , c_1 \Delta( M+1)]\big) \leq \exp\big(-N\epsilon \Delta\big)\times \\
\mathbb{E}\big[ \chi\big\lbrace N^{-1} |\mathcal{I}_N|  \in [c_1\Delta / 2 , c_1 \Delta( M+1)] \big\rbrace \exp\big(\sum_{j\in \mathcal{I}^c_N,p\in I_M}\chi\big\lbrace \big| \tilde{H}^{i,j} \big| \geq (\Delta)^{2/5} \text{ for some }i\in I_M\big\rbrace \big| \tilde{H}^{p,j}\big|^2 \big) \big] .\label{eq: bound exp decay H }
\end{multline}
In the above expectation, $\tilde{\boldsymbol\sigma}$ (which determines the indices $\mathcal{I}_N$) is independent of $\mathbf{J}$. Furthermore, conditionally on $\tilde{\boldsymbol\sigma}$, $\tilde{G}^{i,j}$ is independent of $\tilde{G}^{p,k}$ if $j\neq k$ and $j,k\in \mathcal{I}_N^c$. This last fact is immediate from the definition of the indices in $\mathcal{I}_N^c$: the coefficients of common edges are zero. We thus see that the conditional moments are
\begin{align*}
\mathbb{E}\big[ \tilde{H}^{i,j} | \tilde{\boldsymbol\sigma} \big] &= 0 \\
\mathbb{E}\big[ (\tilde{H}^{i,j})^2\;  | \;  N^{-1} |\mathcal{I}_N|  \in [c_1\Delta/ 2 , c_1 \Delta( M+1)] \big] &\leq N^{-1}\sum_{j\in I_N}(\tilde{\sigma}^{i,j}_{b+1} - \tilde{\sigma}^{i,j}_b)^2 + N^{-1} \\
&\leq 4N^{-1}| \mathcal{I}_N|  + N^{-1} \\
&\leq  4c_1 \Delta( M+1) + N^{-1},
\end{align*}
as long as $N^{1}|\mathcal{I}_N | \leq c_1 \Delta( M+1)$.
Standard Gaussian properties therefore dictate that as long as $N^{-1} |\mathcal{I}_N| \leq c_1 \Delta( M+1)$, there exists a constant $k > 0$ such that for all sufficiently small $\Delta$,
\[
 \mathbb{E}^{\gamma}\big[ \exp\big(\sum_{p\in I_M}\chi\big\lbrace \big| \tilde{H}^{i,j} \big| \geq (\Delta)^{2/5} \text{ for some }i\in I_M\big\rbrace \big| \tilde{H}^{p,j}\big|^2 \big) \; \big| \; \; \tilde{\boldsymbol\sigma} \big]  \leq 1 + \exp\big( - k (\Delta)^{-1/5}\big) \leq \exp\big( \exp\lbrace -k(\Delta)^{-1/5} \rbrace \big).
\]
This means that
\begin{multline*}
\exp\big(-N\epsilon \Delta\big)\mathbb{E}\big[ \chi\big\lbrace N^{-1} |\mathcal{I}_N|  \leq c_1 \Delta( M+1) \big\rbrace \exp\big(\sum_{j\in \mathcal{I}^c_N,p\in I_M}\chi\big\lbrace \big| \tilde{H}^{i,j} \big| \geq (\Delta)^{2/5} \text{ for some }i\in I_M\big\rbrace \big| \tilde{H}^{p,j}\big|^2 \big) \big] \\
\leq \exp\big(-N\epsilon \Delta+N\exp\lbrace -k(\Delta)^{-1/5} \rbrace \big).
\end{multline*}
For small enough $\Delta$, the right hand side is exponentially decaying, as required.
\end{proof}

\subsection{Bounds using Concentration Inequalities for Poisson Processes}

\begin{lemma}
For any $\tilde{\epsilon} > 0$, for all large enough $n$ (and therefore small $\Delta = T/n$), we can find $\mathfrak{n}_0(n) \in \mathbb{Z}^+$ such that for all $\mathfrak{n} \geq \mathfrak{n}_0(n)$,
\begin{equation}
\sup_{0\leq b < n}\lsup{N}\sup_{1\leq \mathfrak{q} \leq C^N_{\mathfrak{n}}} N^{-1}\log \sup_{\boldsymbol\alpha\in \mathcal{E}^M}\mathbb{P}\big(\mathcal{J}_N,\tilde{\mu}^N(\tilde{\boldsymbol\sigma},\tilde{\mathbf{G}}) \in \mathcal{V}^N_{\mathfrak{q}}, \big|\beta^1(\boldsymbol\alpha , \tilde{\boldsymbol\sigma}_b,\mathbf{G}_b)\big| \geq \tilde{\epsilon}\Delta \big) < 0.
\end{equation}
\end{lemma}
\begin{proof}
We prove that
\begin{equation}
\sup_{0\leq b < n}\lsup{N}\sup_{1\leq \mathfrak{q} \leq C^N_{\mathfrak{n}}} N^{-1}\log \sup_{\boldsymbol\alpha\in \mathcal{E}^M}\mathbb{P}\big(\mathcal{J}_N,\tilde{\mu}^N(\tilde{\boldsymbol\sigma},\tilde{\mathbf{G}}) \in \mathcal{V}^N_{\mathfrak{q}}, \beta^1(\boldsymbol\alpha , \tilde{\boldsymbol\sigma}_b,\mathbf{G}_b) \geq \tilde{\epsilon}\Delta \big) < 0.
\end{equation}
The proof of the reverse inequality, i.e.
\begin{equation*}
\sup_{0\leq b < n}\lsup{N}\sup_{1\leq \mathfrak{q} \leq C^N_{\mathfrak{n}}} N^{-1}\log \sup_{\boldsymbol\alpha\in \mathcal{E}^M}\mathbb{P}\big(\mathcal{J}_N,\tilde{\mu}^N(\tilde{\boldsymbol\sigma},\tilde{\mathbf{G}}) \in \mathcal{V}^N_{\mathfrak{q}}, \beta^1(\boldsymbol\alpha , \tilde{\boldsymbol\sigma}_b,\mathbf{G}_b) \leq -\tilde{\epsilon}\Delta \big) < 0,
\end{equation*}
is analogous. One can decompose
\begin{equation}
\beta^1(\boldsymbol\alpha,\tilde{\boldsymbol\sigma},\tilde{\mathbf{G}}) = \mathcal{Z}^N_{\mathfrak{q}}+ \Delta \mathcal{U}^N_{\mathfrak{q}}  ,
\end{equation}
where
\begin{multline}
\mathcal{Z}^N_{\mathfrak{q}} =  N^{-1}\sum_{j\in I_N} \phi^a(\tilde{\mathbf{G}}^{j}_b) \big[ \chi\lbrace \tilde{\boldsymbol\sigma}^j_{b+1}= \boldsymbol\alpha \rbrace - \chi\lbrace \tilde{\boldsymbol\sigma}^j_b= \boldsymbol\alpha \rbrace \\
- \Delta 
\sum_{i\in I_M,j\in I_N}\big\lbrace c\big(-\tilde{\sigma}^{i,j},\tilde{\mathfrak{G}}_{\mathfrak{q},t^{(n)}_b}^{i,j}\big)\chi\lbrace \tilde{\boldsymbol\sigma}^j_b= \boldsymbol\alpha[i] \rbrace - c\big(\tilde{\sigma}^{i,j},\tilde{\mathfrak{G}}_{\mathfrak{q},t^{(n)}_b}^{i,j}\big)\chi\lbrace \tilde{\boldsymbol\sigma}^j_b= \boldsymbol\alpha \rbrace \big\rbrace \big] \text{ and }
\end{multline}
\begin{align}
\mathcal{U}^N_{\mathfrak{q}} &= \mathbb{E}^{\hat{\mu}^N_{b}(\tilde{\boldsymbol\sigma},\tilde{\mathfrak{G}}_{\mathfrak{q}} )}[ H]-\mathbb{E}^{\hat{\mu}^N_{b}(\tilde{\boldsymbol\sigma},\tilde{\mathbf{G}})}[ H] \text{ where }H: \mathcal{E}^M \times \mathbb{R}^M \to \mathbb{R} \text{ is of the form  }\nonumber\\
H(\boldsymbol\zeta,\mathbf{x}) &=\sum_{i\in I_M}\big\lbrace c(-\zeta^i,x^i)\chi\lbrace \boldsymbol\zeta = \boldsymbol\alpha[i] \rbrace -c(\zeta^i,x^i)\chi\lbrace \boldsymbol\zeta= \boldsymbol\alpha \rbrace\big\rbrace .
\end{align}
Thanks to Lemma \ref{Lemma Uniform Partition Wasserstein}, for large enough $\mathfrak{n}$, if $\tilde{\mu}^N(\tilde{\boldsymbol\sigma},\tilde{\mathbf{G}}) \in \mathcal{V}^N_{\mathfrak{q}}$ then necessarily
\begin{equation}
| \mathcal{U}^N_{\mathfrak{q}}| \leq \tilde{\epsilon}\Delta / 2.
\end{equation}
Suppose that $\sum_{i\in I_M}Y_b^{i,j}\big( c_1 \Delta \big) \leq 1$ (recall the definition in \eqref{eq: Y i j b definition}). In this case, at most one of the spins $\lbrace \tilde{\sigma}^{i,j}_s \rbrace_{i \in I_M}$ flips once over the time interval $[t^{(n)}_b , t^{(n)}_{b+1}]$. In this case,
\begin{multline}
  \chi\lbrace \tilde{\boldsymbol\sigma}^j_{b+1}= \boldsymbol\alpha \rbrace - \chi\lbrace \tilde{\boldsymbol\sigma}^j_b= \boldsymbol\alpha \rbrace 
- \Delta 
\sum_{i\in I_M,j\in I_N}\big\lbrace c\big(-\tilde{\sigma}^{i,j}_b,\tilde{\mathfrak{G}}_{\mathfrak{q},t^{(n)}_b}^{i,j}\big)\chi\lbrace \tilde{\boldsymbol\sigma}^j_b= \boldsymbol\alpha[i] \rbrace - c\big(\tilde{\sigma}^{i,j}_b,\tilde{\mathfrak{G}}_{\mathfrak{q},t^{(n)}_b}^{i,j}\big)\chi\lbrace \tilde{\boldsymbol\sigma}^j_b= \boldsymbol\alpha \rbrace \big\rbrace  \\
= \sum_{i\in I_M}\bigg\lbrace Y^{i,j}_b\big( \int^{t^{(n)}_{b+1}}_{t^{(n)}_b}c(\tilde{\sigma}^{i,j}_s , \tilde{\mathfrak{G}}^{i,j}_s) ds\big)\chi\lbrace \tilde{\boldsymbol\sigma}^j_b=  \boldsymbol\alpha[i]\rbrace-Y^{i,j}_b\big( \int^{t^{(n)}_{b+1}}_{t^{(n)}_b}c(\tilde{\sigma}^{i,j}_s , \tilde{\mathfrak{G}}^{i,j}_s) ds\big)\chi\lbrace \tilde{\boldsymbol\sigma}^j_b=  \boldsymbol\alpha \rbrace \\ - \Delta 
\big\lbrace c\big(-\tilde{\sigma}^{i,j}_b,\tilde{\mathfrak{G}}_{\mathfrak{q},t^{(n)}_b}^{i,j}\big)\chi\lbrace \tilde{\boldsymbol\sigma}^j_b= \boldsymbol\alpha[i]\rbrace - c\big(\tilde{\sigma}^{i,j},\tilde{\mathfrak{G}}_{\mathfrak{q},t^{(n)}_b}^{i,j}\big)\chi\lbrace \tilde{\boldsymbol\sigma}^j_b= \boldsymbol\alpha \rbrace \big\rbrace  \bigg\rbrace .
\end{multline}
Conversely if $\sum_{i\in I_M}Y_b^{i,j}\big( c_1 \Delta \big) \geq 2$, then 
\begin{multline*}
  -  \sum_{i\in I_M}\bigg\lbrace Y^{i,j}_b\big( \int^{t^{(n)}_{b+1}}_{t^{(n)}_b}c(\tilde{\sigma}^{i,j}_s , \tilde{\mathfrak{G}}^{i,j}_s) ds\big)\chi\lbrace \tilde{\boldsymbol\sigma}^j_b=  \boldsymbol\alpha[i]\rbrace-Y^{i,j}_b\big( \int^{t^{(n)}_{b+1}}_{t^{(n)}_b}c(\tilde{\sigma}^{i,j}_s , \tilde{\mathfrak{G}}^{i,j}_s) ds\big)\chi\lbrace \tilde{\boldsymbol\sigma}^j_b=  \boldsymbol\alpha \rbrace  \bigg\rbrace \\ +\chi\lbrace \tilde{\boldsymbol\sigma}^j_{b+1}= \boldsymbol\alpha \rbrace - \chi\lbrace \tilde{\boldsymbol\sigma}^j_b= \boldsymbol\alpha \rbrace \leq 4\sum_{i\in I_M}Y^{i,j}_b(c_1\Delta) \chi\big\lbrace Y^{i,j}_b(c_1\Delta) \geq 2 \big\rbrace .
\end{multline*}
We thus find that
\begin{multline}
\big\lbrace \mathcal{Z}^N_{\mathfrak{q}}  > \frac{\tilde{\epsilon}\Delta}{2} \big\rbrace \subseteq \bigg\lbrace  N^{-1}\sum_{j\in I_N,i\in I_M}\chi\lbrace \tilde{\boldsymbol\sigma}^j_b=  \boldsymbol\alpha\rbrace \bigg(\int_{t^{(n)}_b}^{t^{(n)}_{b+1}}c(\tilde{\sigma}^{i,j}_s , \tilde{\mathfrak{G}}^{i,j}_s)ds  -Y^{i,j}_b\big( \int^{t^{(n)}_{b+1}}_{t^{(n)}_b}c(\tilde{\sigma}^{i,j}_s , \tilde{\mathfrak{G}}^{i,j}_s) ds\big) \bigg) \geq \frac{\tilde{\epsilon}\Delta}{8} \bigg\rbrace \\
\cup \bigg\lbrace  N^{-1}\sum_{j\in I_N,i\in I_M}\chi\lbrace \tilde{\boldsymbol\sigma}^j_b=  \boldsymbol\alpha[i]\rbrace\bigg(Y^{i,j}_b\big( \int^{t^{(n)}_{b+1}}_{t^{(n)}_b}c(\tilde{\sigma}^{i,j}_s , \tilde{\mathfrak{G}}^{i,j}_s) ds\big)-\int_{t^{(n)}_b}^{t^{(n)}_{b+1}}c(\tilde{\sigma}^{i,j}_s , \tilde{\mathfrak{G}}^{i,j}_s)ds  \bigg) \geq \frac{\tilde{\epsilon}\Delta}{8} \bigg\rbrace \\
\cup \bigg\lbrace N^{-1}\sum_{j\in I_N,i\in I_M}\chi\lbrace \tilde{\boldsymbol\sigma}^j_b=  \boldsymbol\alpha\rbrace\int_{t^{(n)}_b}^{t^{(n)}_{b+1}}\big\lbrace c(\tilde{\sigma}^{i,j}_b , \tilde{\mathfrak{G}}^{i,j}_b)-c(\tilde{\sigma}^{i,j}_s , \tilde{\mathfrak{G}}^{i,j}_s)\big\rbrace ds \\
- N^{-1}\sum_{j\in I_N,i\in I_M}\chi\lbrace \tilde{\boldsymbol\sigma}^j_b=  \boldsymbol\alpha[i]\rbrace\int_{t^{(n)}_b}^{t^{(n)}_{b+1}}\big\lbrace c(\tilde{\sigma}^{i,j}_b , \tilde{\mathfrak{G}}^{i,j}_b)-c(\tilde{\sigma}^{i,j}_s , \tilde{\mathfrak{G}}^{i,j}_s)\big\rbrace ds \geq \frac{\tilde{\epsilon}\Delta}{8}  \bigg\rbrace  \\
\cup \bigg\lbrace 4N^{-1}\sum_{j\in I_N,i\in I_M}Y_b^{i,j}(c_1\Delta) \chi\lbrace Y_b^{i,j}(c_1\Delta) \geq 2 \rbrace \geq \frac{\tilde{\epsilon}\Delta}{8}  \bigg\rbrace . \label{eq: beta 1 decomposition}
\end{multline}
The probability of each of the first two terms on the right hand side is exponentially decaying thanks to (i) of Lemma \ref{Lemma Concentration Poisson}. For the third term, one easily shows that as long as the event $\mathcal{J}_N$ holds,
\begin{multline*}
 \bigg| N^{-1}\sum_{j\in I_N,i\in I_M}\chi\lbrace \tilde{\boldsymbol\sigma}^j_b=  \boldsymbol\alpha\rbrace\int_{t^{(n)}_b}^{t^{(n)}_{b+1}}\big\lbrace c(\tilde{\sigma}^{i,j}_b , \tilde{\mathfrak{G}}^{i,j}_b)-c(\tilde{\sigma}^{i,j}_s , \tilde{\mathfrak{G}}^{i,j}_s)\big\rbrace ds \\
- N^{-1}\sum_{j\in I_N,i\in I_M}\chi\lbrace \tilde{\boldsymbol\sigma}^j_b=  \boldsymbol\alpha[i]\rbrace\int_{t^{(n)}_b}^{t^{(n)}_{b+1}}\big\lbrace c(\tilde{\sigma}^{i,j}_b , \tilde{\mathfrak{G}}^{i,j}_b)-c(\tilde{\sigma}^{i,j}_s , \tilde{\mathfrak{G}}^{i,j}_s)\big\rbrace ds\bigg| \leq \text{Const}\big(N^{-1}\sum_{j\in I_N,i\in I_M}\chi\lbrace Y^{i,j}_b(c_1\Delta t) \geq 1 \rbrace \big)^2.
\end{multline*}
Thanks to (ii) of Lemma \ref{Lemma Upper Bound Jump Rate}, one finds that the probability of the RHS of the above equation exceeding $\tilde{\epsilon}\Delta /8$ is exponentially decaying in $N$, once $\Delta$ is small enough. For the last term on the RHS of \eqref{eq: beta 1 decomposition}, by Chernoff's Inequality, for a constant $u > 0$,
\begin{multline}
\mathbb{P}\big( 4N^{-1}\sum_{j\in I_N,i\in I_M}Y_b^{i,j}(c_1\Delta) \chi\lbrace Y_b^{i,j}(c_1\Delta) \geq 2 \rbrace \geq \tilde{\epsilon}\Delta / 8\big)\\ \leq
\mathbb{E}\big[ \exp\big( 4uN^{-1}\sum_{j\in I_N,i\in I_M}Y_b^{i,j}(c_1\Delta) \chi\lbrace Y_b^{i,j}(c_1\Delta) \geq 2 \rbrace - u\tilde{\epsilon}\Delta /8 \big)\big] \label{eq: final beta 1 bound}
\end{multline}
Now for any positive integer $k$, thanks to the renewal property of Poisson Processes,
\begin{align*}
\mathbb{P}\big( Y_b^{i,j}(c_1\Delta) = k \big) \leq \mathbb{P}\big(Y_b^{i,j}(c_1\Delta) = 1\big)^k = \big\lbrace c_1\Delta \exp(-c_1\Delta) \big\rbrace^k,
\end{align*}
since $Y_b^{i,j}(c_1\Delta)$ is Poisson-distributed. We take $n$ sufficiently large that $\Delta \leq (c_1)^{-1}$, and $4u = c_1\Delta / 2$, and we obtain that 
\begin{align*}
\mathbb{E}\big[ \exp\big( 4uN^{-1}Y_b^{i,j}(c_1\Delta) \chi\lbrace Y_b^{i,j}(c_1\Delta) \geq 2 \rbrace - u\tilde{\epsilon}\Delta /8 \big)\big]
&\leq \exp(- u\tilde{\epsilon}\Delta /8 ) \sum_{k=2}^\infty \exp\big(-kc_1\Delta /2 \big)\\ &= \exp\big(- u\tilde{\epsilon}\Delta /8 -c_1\Delta \big)\big(1 -\exp(-c_1\Delta/2) \big)^{-1},
\end{align*}
by summing the geometric series. Since the processes $Y^{i,j}_b$ are independent, we find that the RHS of \eqref{eq: final beta 1 bound} is exponentially decaying in $N$, as required.
\end{proof}

\begin{lemma}
For any $\bar{\epsilon}$, for all sufficiently large $n$,
\begin{align}
\lsup{N}\sup_{1\leq \mathfrak{q} \leq C^N_{\mathfrak{n}}} N^{-1}\log \mathbb{P}\big(\mathcal{J}_N, \tilde{\mu}^N \in \mathcal{V}^N_{\mathfrak{q}}, \tilde{\tau}_N > t^{(n)}_b, \big|\beta^{11}(\boldsymbol\alpha,\tilde{\boldsymbol\sigma},\tilde{\mathbf{G}})\big| \geq \bar{\epsilon}\Delta\big) < 0\\
\lsup{N}\sup_{1\leq \mathfrak{q} \leq C^N_{\mathfrak{n}}} N^{-1}\log \mathbb{P}\big(\mathcal{J}_N,  \tilde{\mu}^N \in \mathcal{V}^N_{\mathfrak{q}},\tilde{\tau}_N > t^{(n)}_b, \big|\beta^{7}(\boldsymbol\alpha,\tilde{\boldsymbol\sigma},\tilde{\mathbf{G}})\big| \geq \bar{\epsilon}\Delta\big) < 0\label{Eq: beta 11 final}
\end{align}
\end{lemma}
\begin{proof}
Now since $|\phi^a_{ii}| \leq 1$, $\big\lbrace \big|\beta^{11}(\boldsymbol\alpha,\tilde{\boldsymbol\sigma},\tilde{\mathbf{G}})\big| \geq \bar{\epsilon}\Delta \big\rbrace \subseteq \big\lbrace \big| 2\Delta L^{\hat{\mu}^N_b}_{ii} - \tilde{L}^{ii}_b \big| \geq \bar{\epsilon}\Delta \big\rbrace$. The probability of this event is exponentially decaying, thanks to Lemma \ref{Lemma Matrix Convergence}.

The proof of \eqref{Eq: beta 11 final} is similar: one compares the definition of $\tilde{\mathbf{m}}_b$ in \eqref{eq: temp H 0} to the definition of $\mathbf{m}^{\hat{\mu}^N_b}$ in \eqref{eq: m xi}. Note that the condition $\tilde{\tau}_N > t^{(n)}_b$ implies that $\tilde{\mathbf{H}}_b = \mathbf{H}^{\hat{\mu}^N_b}$ (see the definition in \eqref{eq: H xi definition}) and also $\tilde{\boldsymbol\upsilon}_b = \boldsymbol\upsilon^{\hat{\mu}^N_b}$. One therefore finds that
\begin{align*}
\big|\beta^{7}(\boldsymbol\alpha,\tilde{\boldsymbol\sigma},\tilde{\mathbf{G}}) \big|  &\leq N^{-1}\sum_{j\in I_N}\sum_{i\in I_M}\big| \tilde{m}^{i,j}_b -\Delta m^{\hat{\mu}^N_b,i}(\tilde{\boldsymbol\sigma}^j_b, \tilde{\mathbf{G}}^j_b) \big| \\
\tilde{\mathbf{m}}^j_b-\Delta \mathbf{m}^{\hat{\mu}^N_b}(\tilde{\boldsymbol\sigma}^j_b, \tilde{\mathbf{G}}^j_b) &= - (\tilde{\mathbf{L}}_b - 2\Delta \mathbf{L}^{\hat{\mu}^N_b})\tilde{\mathbf{H}}_b\tilde{\mathbf{G}}^j_b -\mathfrak{s}( \tilde{\boldsymbol\kappa}_b - 2\Delta\boldsymbol\kappa^{\hat{\mu}^N_b} )\tilde{\mathbf{H}}_b\tilde{\boldsymbol\sigma}^j_b + \mathfrak{s} (\tilde{\mathbf{L}}_b - 2\Delta \mathbf{L}^{\hat{\mu}^N_b})\tilde{\mathbf{H}}_b\tilde{\boldsymbol\upsilon}_b\tilde{\mathbf{H}}_b\tilde{\boldsymbol\sigma}^j_b .
\end{align*}
Thus 
\begin{multline*}
\norm{\tilde{\mathbf{m}}^j_b-\Delta \mathbf{m}^{\hat{\mu}^N_b}(\tilde{\boldsymbol\sigma}^j_b, \tilde{\mathbf{G}}^j_b)} \leq \norm{\tilde{\mathbf{L}}_b - 2\Delta \mathbf{L}^{\hat{\mu}^N_b}} \norm{\tilde{\mathbf{H}}_b}\norm{\tilde{\mathbf{G}}^j_b} +\norm{\tilde{\boldsymbol\kappa}_b - 2\Delta\boldsymbol\kappa^{\hat{\mu}^N_b} }\norm{\tilde{\mathbf{H}}_b}\norm{\tilde{\boldsymbol\sigma}^j_b} + \norm{\tilde{\mathbf{L}}_b - 2\Delta \mathbf{L}^{\hat{\mu}^N_b}} \norm{\tilde{\mathbf{H}}_b}^2 \norm{\tilde{\boldsymbol\upsilon}_b}\norm{\tilde{\boldsymbol\sigma}^j_b} 
\end{multline*}
 Now write
\[ 
\mathcal{U}^N =\big\lbrace \norm{\tilde{\mathbf{L}}_b - \mathbf{L}^{\hat{\mu}^N_b}} \leq \epsilon_0\Delta \; , \; \norm{\tilde{\boldsymbol\kappa}_b - \boldsymbol\kappa^{\hat{\mu}^N_b}} \big\rbrace \leq \epsilon_0 \Delta \big\rbrace.
\]  
We now establish that 
\[
\big\lbrace \mathcal{J}_N, \tilde{\tau}_N > t^{(n)}_b, \big|\beta^{7}(\boldsymbol\alpha,\tilde{\boldsymbol\sigma},\tilde{\mathbf{G}})\big| \geq \bar{\epsilon}\Delta \big\rbrace \subseteq \mathcal{J}_N \cap (\mathcal{U}^N)^c \cap \big\lbrace \tilde{\tau}_N > t^{(n)}_b \big\rbrace
\]
for sufficiently small $\epsilon_0$, and sufficiently small $\Delta$. Now $\tilde{\tau}_N > t^{(n)}_b$ implies that $\norm{\tilde{\mathbf{H}}_b} \leq \mathfrak{c}^{-1}$, and the Cauchy-Schwarz Inequality (and also condition $\mathcal{J}_N$) imply that $| \upsilon_b^{pq} |^2 \leq N^{-1}\sum_{j\in I_N}|G^{p,j}|^2 \leq 3$. This means that $\norm{\boldsymbol\upsilon_b}$ is bounded. Furthermore by Jensen's Inequality $\big(N^{-1}\sum_{j\in I_N} \norm{\mathbf{G}^j_b}\big)^2 \leq N^{-1}\sum_{j\in I_N} \norm{\mathbf{G}^j_b}^2 \leq  3M$ (as a consequence of $\mathcal{J}_N$). The probability of $(\mathcal{U}^N)^c$ is exponentially decaying, thanks to Lemma \ref{Lemma Matrix Convergence}.
\end{proof}

\begin{lemma}\label{Lemma Matrix Convergence}
For any $\epsilon_0 > 0$, there exists $n_0 \in \mathbb{Z}^+$ such that for all $n\geq n_0$, there exists $\mathfrak{n}_0(n)$ such that for all $\mathfrak{n} \geq \mathfrak{n}_0(n)$,
\begin{align}
\lsup{N}\sup_{1\leq \mathfrak{q} \leq C^N_{\mathfrak{n}}}N^{-1}\log \mathbb{P}\big(\mathcal{J}_N, \tilde{\mu}^N \in \mathcal{V}^N_{\mathfrak{q}},\sup_{0\leq b \leq n-1}\sup_{p,q\in I_M}\big| \tilde{L}^{pq}_b-2\Delta L_{pq}^{\hat{\mu}^N_b(\tilde{\boldsymbol\sigma},\tilde{\mathbf{G}})} \big| \geq \epsilon_0 \Delta \big) &< 0 \\
\lsup{N}\sup_{1\leq \mathfrak{q} \leq C^N_{\mathfrak{n}}} N^{-1}\log \mathbb{P}\big(\mathcal{J}_N, \tilde{\mu}^N \in \mathcal{V}^N_{\mathfrak{q}},\sup_{0\leq b \leq n-1}\sup_{p,q\in I_M}\big| \tilde{\kappa}^{pq}_b-2\Delta \kappa_{pq}^{\hat{\mu}^N_b(\tilde{\boldsymbol\sigma},\tilde{\mathbf{G}})}\big| \geq  \epsilon_0 \Delta \big) &< 0 . \label{eq: kappa tilde kappa convergence}
\end{align}
\end{lemma}
\begin{proof}
The proofs are almost identical, so we only prove \eqref{eq: kappa tilde kappa convergence}. Recall from \eqref{eq: kappa xi defn} and \eqref{eq: tilde upsilon definition 0 0} that
\begin{equation}
\tilde{\kappa}^{ij}_b = N^{-1}\sum_{k=1}^{N} \tilde{G}^{j,k}_b \big(\tilde{\sigma}^{i,k}_b-\tilde{\sigma}^{i,k}_{b+1}\big) \; \; \;\text{ and }\; \; \; \kappa_{pq}^{\hat{\mu}^N_b(\tilde{\boldsymbol\sigma},\tilde{\mathbf{G}})}= N^{-1}\sum_{k=1}^{N} \tilde{G}^{j,k}_b\tilde{\sigma}^{i,k}_b c(\tilde{\sigma}^{i,k}_b , \tilde{G}^{i,k}_b).
\end{equation}
By Lemma \ref{Lemma Many Events}, 
\begin{multline}
\lsup{N}\sup_{1\leq \mathfrak{q} \leq C^N_{\mathfrak{n}}}N^{-1}\log \mathbb{P}\big(\mathcal{J}_N, \tilde{\mu}^N \in \mathcal{V}^N_{\mathfrak{q}},\sup_{0\leq b \leq n-1}\sup_{p,q\in I_M}\big| \tilde{\kappa}^{pq}_b-2\Delta \kappa_{pq}^{\hat{\mu}^N_b(\tilde{\boldsymbol\sigma},\tilde{\mathbf{G}})}\big| \geq  \epsilon_0\Delta \big) \\
= \sup_{0\leq b \leq n-1}\sup_{p,q\in I_M}\lsup{N}\sup_{1\leq \mathfrak{q} \leq C^N_{\mathfrak{n}}}N^{-1}\log \mathbb{P}\big(\mathcal{J}_N, \tilde{\mu}^N \in \mathcal{V}^N_{\mathfrak{q}},\big| \tilde{\kappa}^{pq}_b-2\Delta \kappa_{pq}^{\hat{\mu}^N_b(\tilde{\boldsymbol\sigma},\tilde{\mathbf{G}})}\big| \geq \epsilon_0 \Delta \big) .
\end{multline}
Now a union of events bound implies that
\begin{multline}
\mathbb{P}\big(\mathcal{J}_N, \tilde{\mu}^N \in \mathcal{V}^N_{\mathfrak{q}},\big| \tilde{\kappa}^{pq}_b-2\Delta \kappa_{pq}^{\hat{\mu}^N_b(\tilde{\boldsymbol\sigma},\tilde{\mathbf{G}})}\big| \geq \epsilon_0 \Delta \big) \leq
\mathbb{P}\big(\mathcal{J}_N, \tilde{\mu}^N \in \mathcal{V}^N_{\mathfrak{q}},\big| 2\Delta\kappa_{pq}^{\hat{\mu}^N_b(\tilde{\boldsymbol\sigma},\tilde{\mathfrak{G}}_{\mathfrak{q}})}-2\Delta \kappa_{pq}^{\hat{\mu}^N_b(\tilde{\boldsymbol\sigma},\tilde{\mathbf{G}})}\big| \geq \epsilon_0 \Delta / 3\big)\\
+\mathbb{P}\big(\mathcal{J}_N, \tilde{\mu}^N \in \mathcal{V}^N_{\mathfrak{q}},\big| \breve{\kappa}^{pq}_b-2\Delta \kappa_{pq}^{\hat{\mu}^N_b(\tilde{\boldsymbol\sigma},\tilde{\mathfrak{G}}_{\mathfrak{q}})}\big| \geq \epsilon_0 \Delta / 3 \big)+\mathbb{P}\big(\mathcal{J}_N, \tilde{\mu}^N \in \mathcal{V}^N_{\mathfrak{q}},\big| \tilde{\kappa}^{pq}_b-\breve{\kappa}_b^{pq}\big| \geq \epsilon_0 \Delta / 3 \big). \label{eq: remaining kappa derivative bound}
\end{multline}
where
\begin{equation}
\breve{\kappa}^{ij}_b = N^{-1}\sum_{k=1}^{N} \tilde{\mathfrak{G}}^{j,k}_{\mathfrak{q},t^{(n)}_b} \big(\tilde{\sigma}^{i,k}_b-\tilde{\sigma}^{i,k}_{b+1}\big).
\end{equation}
Now, by definition, $\hat{\mu}^N(\tilde{\boldsymbol\sigma},\tilde{\mathfrak{G}}_{\mathfrak{q}}) \in \mathcal{V}^N_{\mathfrak{q}}$. Thus if $\tilde{\mu}^N(\tilde{\boldsymbol\sigma},\tilde{\mathbf{G}}) \in \mathcal{V}^N_{\mathfrak{q}}$ as well, then since the radius of the set $\mathcal{V}^N_{\mathfrak{q}}$ goes to zero as $\mathfrak{n}\to\infty$, (as proved in Lemma \ref{Lemma Uniform Partition Wasserstein}), it must be that for sufficiently large $\mathfrak{n}$
\begin{align*}
\mathbb{P}\big(\tilde{\mu}^N \in \mathcal{V}^N_{\mathfrak{q}},\big| \tilde{\kappa}^{pq}_b-\breve{\kappa}_b^{pq}\big| \geq \epsilon_0 \Delta / 3 \big) =& 0,
\end{align*}
since
\begin{equation*}
\tilde{\kappa}^{pq}_b - \breve{\kappa}^{pq}_b = N^{-1}\sum_{k=1}^{N}\big\lbrace \tilde{G}^{q,k}_{t^{(n)}_b} \big(\tilde{\sigma}^{p,k}_b-\tilde{\sigma}^{p,k}_{b+1}\big)
-\tilde{\mathfrak{G}}^{q,k}_{\mathfrak{q},t^{(n)}_b} \big(\tilde{\sigma}^{p,k}_b-\tilde{\sigma}^{p,k}_{b+1}\big) \big\rbrace.
\end{equation*}
We similarly find that for large enough $\mathfrak{n}$,
\begin{equation}
\mathbb{P}\big(\mathcal{J}_N, \tilde{\mu}^N \in \mathcal{V}^N_{\mathfrak{q}},\big| 2\Delta\kappa_{pq}^{\hat{\mu}^N_b(\tilde{\boldsymbol\sigma},\tilde{\mathfrak{G}}_{\mathfrak{q}})}-2\Delta \kappa_{pq}^{\hat{\mu}^N_b(\tilde{\boldsymbol\sigma},\tilde{\mathbf{G}})}\big| \geq \epsilon_0 \Delta / 3\big) =
\mathbb{P}\big(\mathcal{J}_N, \tilde{\mu}^N \in \mathcal{V}^N_{\mathfrak{q}},\big| \kappa_{pq}^{\hat{\mu}^N_b(\tilde{\boldsymbol\sigma},\tilde{\mathfrak{G}}_{\mathfrak{q}})}-\kappa_{pq}^{\hat{\mu}^N_b(\tilde{\boldsymbol\sigma},\tilde{\mathbf{G}})}\big| \geq \epsilon_0  / 6\big) =0 
\end{equation}
Concerning the other term on the right hand side of \eqref{eq: remaining kappa derivative bound},
\begin{align*}
 \breve{\kappa}^{pq}_b-2\Delta \kappa_{pq}^{\hat{\mu}^N_b(\tilde{\boldsymbol\sigma},\tilde{\mathfrak{G}}_{\mathfrak{q}})} =& \grave{\kappa}^{pq} + \bar{\kappa}^{pq} \text{ where }\\
 \grave{\kappa}^{pq} =& N^{-1}\sum_{j\in I_N}\tilde{\mathfrak{G}}^{q,k}_{\mathfrak{q},t^{(n)}_b} \big(\tilde{\sigma}^{p,k}_b-\tilde{\sigma}^{p,k}_{b+1} -2 \int_{t^{(n)}_b}^{t^{(n)}_{b+1}}\tilde{\sigma}^{p,k}_s c(\tilde{\sigma}^{p,k}_b, \tilde{\mathfrak{G}}^{p,k}_{\mathfrak{q},t^{(n)}_b})ds \big) \\
 \bar{\kappa}^{pq} =& 2N^{-1}\sum_{j\in I_N}\tilde{\mathfrak{G}}^{q,k}_{\mathfrak{q},t^{(n)}_b} \int_{t^{(n)}_b}^{t^{(n)}_{b+1}}(\tilde{\sigma}^{p,k}_s - \tilde{\sigma}^{p,k}_{b}) c(\tilde{\sigma}^{p,k}_b, \tilde{\mathfrak{G}}^{p,k}_{\mathfrak{q},t^{(n)}_b})ds \big) 
\end{align*}
We thus find that
\begin{multline*}
\mathbb{P}\big(\mathcal{J}_N, \tilde{\mu}^N \in \mathcal{V}^N_{\mathfrak{q}},\big| \breve{\kappa}^{pq}_b-2\Delta \kappa_{pq}^{\hat{\mu}^N_b(\tilde{\boldsymbol\sigma},\tilde{\mathfrak{G}}_{\mathfrak{q}})}\big| \geq \epsilon_0 \Delta / 3 \big) \leq \\ \mathbb{P}\big(\mathcal{J}_N, \tilde{\mu}^N \in \mathcal{V}^N_{\mathfrak{q}},\big| \grave{\kappa}^{pq}\big| \geq \epsilon_0 \Delta / 6 \big)+\mathbb{P}\big(\mathcal{J}_N, \tilde{\mu}^N \in \mathcal{V}^N_{\mathfrak{q}},\big| \bar{\kappa}^{pq}\big| \geq \epsilon_0 \Delta / 6 \big).
\end{multline*}
Now $\grave{\kappa}^{pq}$ is the sum of compensated Poisson Processes (which are Martingales), since, making use of the representation in \eqref{eq: tilde sigma alternative definition},
\[
\tilde{\sigma}^{p,k}_t - \tilde{\sigma}^{p,k}_b = -2\int_{t^{(n)}_b}^{t}\tilde{\sigma}^{p,k}_s d\hat{\sigma}^{p,k}_s \text{ where }\hat{\sigma}^{p,k}_s = \tilde{\sigma}^{p,k}_bA\cdot Y_b^{p,k}\bigg( \int_{t^{(n)}_b}^{s} c(\tilde{\sigma}^{p,k}_b , \tilde{\mathfrak{G}}^{p,k}_{\mathfrak{q},b}) dr \bigg).
\]
Recalling that $|\tilde{\mathfrak{G}}^{q,k}_{\mathfrak{q},t^{(n)}_b}| \leq \mathfrak{n}$, it is therefore a consequence of Lemma \ref{Lemma Concentration Poisson} that
\[
\lsup{N}\sup_{1\leq \mathfrak{q} \leq C^N_{\mathfrak{n}}}N^{-1}\log \mathbb{P}\big(\mathcal{J}_N, \tilde{\mu}^N \in \mathcal{V}^N_{\mathfrak{q}},\big| \grave{\kappa}^{pq}\big| \geq \epsilon_0 \Delta / 6 \big) < 0.
\]
Since $\tilde{\sigma}^{p,k}_s = \tilde{\sigma}^{p,k}_b$ for all $s\in [t^{(n)}_b , t^{(n)}_{b+1}]$ if $Y^{p,k}_b(c_1\Delta) = 0$, we similarly find that
\begin{multline*}
\lsup{N}\sup_{1\leq \mathfrak{q} \leq C^N_{\mathfrak{n}}}N^{-1}\log \mathbb{P}\big(\mathcal{J}_N, \tilde{\mu}^N \in \mathcal{V}^N_{\mathfrak{q}},\big| \bar{\kappa}^{pq}\big| \geq \epsilon_0 \Delta / 6 \big) \\
\leq \lsup{N}\sup_{1\leq \mathfrak{q} \leq C^N_{\mathfrak{n}}}N^{-1}\log \mathbb{P}\big( 4c_1N^{-1}\sum_{j\in I_N}\chi\lbrace Y^{p,k}_b(c_1 \Delta) > 0 \rbrace \geq \epsilon_0  / 6 \big) 
< 0,
\end{multline*}
for large enough $n$ (recalling that $\Delta = Tn^{-1}$), thanks to Lemma \ref{Lemma Upper Bound Jump Rate} (ii).

\end{proof}

\begin{lemma}
For any $\bar{\epsilon} > 0$, for large enough $n\in \mathbb{Z}^+$,
\begin{equation}
\lsup{N}N^{-1}\log  \mathbb{P}\big(\tilde{\tau}_N > t^{(n)}_b,\mathcal{J}_N,\big|\beta^{10}(\boldsymbol\alpha,\tilde{\boldsymbol\sigma},\tilde{\mathbf{G}})  \big| \geq \bar{\epsilon}\Delta\big) < 0. \label{eq: beta 12 equation}
\end{equation}
\end{lemma}
\begin{proof}
Since $\big| \phi^a_{ip}(\tilde{\mathbf{G}}_b^{j})| \leq 1$ by definition,
\begin{align*}
\big|\beta^{10}(\boldsymbol\alpha,\tilde{\boldsymbol\sigma},\tilde{\mathbf{G}}) \big| &\leq  (2N)^{-1}\big|\sum_{j\in I_N}\sum_{i,p\in I_M}\chi\lbrace \tilde{\boldsymbol\sigma}_b^j=\boldsymbol\alpha \rbrace \tilde{m}^{i,j}_b \tilde{m}^{p,j}_b \big| \\
&=  (2N)^{-1}\sum_{j\in I_N}\chi\lbrace \tilde{\boldsymbol\sigma}_b^j=\boldsymbol\alpha \rbrace \big(\sum_{p\in I_M}\tilde{m}^{p,j}_b \big)^2 
\leq  (2N)^{-1}M\sum_{j\in I_N} \sum_{p\in I_M}(\tilde{m}^{p,j}_b)^2,
\end{align*}
by Jensen's Inequality. Thanks to the triangle inequality,
\begin{align*}
 \norm{\tilde{\mathbf{m}}^j_b } &\leq \norm{\tilde{\mathbf{L}}_b\tilde{\mathbf{H}}_b\tilde{\mathbf{G}}^j_b} +\mathfrak{s} \norm{ \tilde{\boldsymbol\kappa}_b\tilde{\mathbf{H}}_b\tilde{\boldsymbol\sigma}^j_b} + \mathfrak{s}\norm{\tilde{\mathbf{L}}_b\tilde{\mathbf{H}}_b\tilde{\boldsymbol\nu}\tilde{\mathbf{H}}_b\tilde{\boldsymbol\sigma}^j_b} \\
 &\leq \norm{\tilde{\mathbf{L}}_b}\norm{\tilde{\mathbf{H}}_b}\norm{\tilde{\mathbf{G}}^j_b} +\mathfrak{s} \norm{\tilde{\boldsymbol\kappa}_b}\norm{\tilde{\mathbf{H}}_b}\norm{\tilde{\boldsymbol\sigma}^j_b} + \mathfrak{s}\norm{\tilde{\mathbf{L}}_b}\norm{\tilde{\mathbf{H}}_b}\norm{\tilde{\boldsymbol\nu}_b}\norm{\tilde{\mathbf{H}}_b}\norm{\tilde{\boldsymbol\sigma}^j_b}.
\end{align*}
We thus find that
\begin{equation}
 \norm{\tilde{\mathbf{m}}^j_b }^2  
 \leq 3\norm{\tilde{\mathbf{L}}_b}^2\norm{\tilde{\mathbf{H}}_b}^2\norm{\tilde{\mathbf{G}}^j_b}^2 +3\mathfrak{s}^2 \norm{\tilde{\boldsymbol\kappa}_b}^2\norm{\tilde{\mathbf{H}}_b}^2\norm{\tilde{\boldsymbol\sigma}^j_b}^2 + 3\mathfrak{s}^2\norm{\tilde{\mathbf{L}}_b}^2\norm{\tilde{\mathbf{H}}_b}^4\norm{\tilde{\boldsymbol\nu}_b}^2\norm{\tilde{\boldsymbol\sigma}^j_b}^2.
\end{equation}
Since $\tilde{\tau}_N > t^{(n)}_b$, $\norm{\tilde{\mathbf{H}}_b} \leq \mathfrak{c}^{-1}$. Since $|\tilde{\sigma}^{i,j}_t | \leq 1$, $\norm{\tilde{\boldsymbol\sigma}^j_b} \leq \sqrt{M}$. The event $\mathcal{J}_N$ implies - after an application of the Cauchy-Schwarz Inequality - that $| \tilde{\upsilon}^{pq}_b| \leq \sqrt{3}$ and $| \tilde{\kappa}^{pq}_b | \leq \sqrt{3} c_1$. Now
\begin{align*}
\big| \tilde{L}_b^{pq} \big| &\leq \Delta\big(2 \big| L^{\hat{\mu}^N_b}_{pq} \big| +  \big| \tilde{L}^{pq}_b-2\Delta L_{pq}^{\hat{\mu}^N_b(\tilde{\boldsymbol\sigma},\tilde{\mathbf{G}})}\big|  \big) \\
\big| \tilde{\kappa}_b^{pq} \big| &\leq \Delta\big(2 \big| \kappa^{\hat{\mu}^N_b}_{pq} \big| + \big| \tilde{\kappa}^{pq}_b-2\Delta \kappa_{pq}^{\hat{\mu}^N_b(\tilde{\boldsymbol\sigma},\tilde{\mathbf{G}})} \big| \big).
\end{align*}
Lemma \ref{Lemma Matrix Convergence} implies that the probability of the following event not holding is exponentially decaying,
\begin{align}
\big\lbrace \sup_{p,q\in I_M}\big| \tilde{L}^{pq}_b-2\Delta L_{pq}^{\hat{\mu}^N_b(\tilde{\boldsymbol\sigma},\tilde{\mathbf{G}})} \big| \leq \Delta \; \; , \; \; \sup_{p,q \in I_M} \big| \tilde{\kappa}^{pq}_b-2\Delta \kappa_{pq}^{\hat{\mu}^N_b(\tilde{\boldsymbol\sigma},\tilde{\mathbf{G}})}\big|  \leq \Delta \big\rbrace . \label{eq: tilde L tilde kappa bound}
\end{align}
We can thus assume that the above events hold. Since $| L_{pq}^{\hat{\mu}^N_b(\tilde{\boldsymbol\sigma},\tilde{\mathbf{G}})}| \leq c_1$ and 
\[
\big| \kappa_{pq}^{\hat{\mu}^N_b(\tilde{\boldsymbol\sigma},\tilde{\mathbf{G}})}\big| \leq c_1 \mathbb{E}^{\hat{\mu}^N_b}[ |g_q|] \leq c_1 \sqrt{3},
\]
it must be that there exist positive constants $C_1,C_2$ such that
\begin{equation}
 \norm{\tilde{\mathbf{m}}^j_b }^2  
 \leq C_1 (\Delta)^2 \norm{\tilde{\mathbf{G}}^j_b}^2 +C_2 (\Delta)^2.
\end{equation}
We thus find that 
\begin{equation}
 N^{-1}\sum_{j\in I_N}\norm{\tilde{\mathbf{m}}^j_b }^2  
 \leq C_1 (\Delta)^2 \sum_{j\in I_N}\norm{\tilde{\mathbf{G}}^j_b}^2 +C_2 (\Delta)^2 \leq 3C_1M (\Delta)^2 + C_2 (\Delta)^2,
\end{equation}
as long as the event $\mathcal{J}_N$ holds. In conclusion, as long as the events $\tilde{\tau}_N \geq t^{(n)}_b$, $\mathcal{J}_N$ and \eqref{eq: tilde L tilde kappa bound} hold, it must be that
\[
\big|\beta^{10}(\boldsymbol\alpha,\tilde{\boldsymbol\sigma},\tilde{\mathbf{G}}) \big|  \leq  3C_1M (\Delta)^2 + C_2 (\Delta)^2.
\]
Clearly for small enough $\Delta$, \eqref{eq: beta 12 equation} must hold. 
\end{proof}

\section{Using the Gaussian Law to Estimate the Field Dynamics}
\label{Section Field}

In this section we continue the proof of Lemma \ref{Lemma general beta bound}: providing bounds for the terms \newline$\beta^6(\boldsymbol\alpha,\tilde{\boldsymbol\sigma},\tilde{\mathbf{G}}),\beta^8(\boldsymbol\alpha,\tilde{\boldsymbol\sigma},\tilde{\mathbf{G}}),\beta^9(\boldsymbol\alpha,\tilde{\boldsymbol\sigma},\tilde{\mathbf{G}})$. The bounding of these terms requires the law $\gamma$ of the Gaussian connections $\lbrace J^{jk} \rbrace_{j,k\in I_N}$. Recall that the processes $\lbrace \tilde{\boldsymbol\sigma}_i \rbrace_{1\leq i \leq C^N_{\mathfrak{n}}}$ are independent of the connections, and so conditioning on these processes does not affect the distribution of the connections. For fixed $\tilde{\boldsymbol\sigma}_b \in \mathcal{E}^{MN}$ and any $\mathbf{g} \in \mathbb{R}^{MN}$, let $\gamma_{\tilde{\boldsymbol\sigma}_b, \mathbf{g}} \in \mathcal{M}^+_1\big(\mathbb{R}^{N^2}\big)$ be the regular conditional probability distribution of the connections $\mathbf{J}$, conditionally on 
\begin{equation}
N^{-1/2}\sum_{k\in I_N} J^{jk}\tilde{\sigma}^{p,k}_b = g^{p,j}.
\end{equation}
Standard theory dictates that $\gamma_{\tilde{\boldsymbol\sigma}_b, \mathbf{g}}$ is Gaussian (see for instance Theorem A.1.3 in \cite{Lindgren2013}). We start by determining expressions for the conditional mean and variance of $\gamma_{\tilde{\boldsymbol\sigma}_b, \mathbf{g}}$ in Section \ref{Section Gaussian}. We then use these expressions to bound $\beta^6, \beta^8$ and $\beta^9$ in Section \ref{Section bounding beta 6 8 9}.

\subsection{The Conditional Mean and Covariance} \label{Section Gaussian}


The main result of this section is Lemma \ref{Lemma Conditional Mean}: this lemma is crucial because it demonstrates that the conditional mean of the increment $\tilde{G}^{i,j}_{b+1} - \tilde{G}^{i,j}_b$ can be written as a function of the variables $\lbrace \tilde{\boldsymbol\sigma}^{j}_b, \tilde{\boldsymbol\sigma}^{j}_{b+1},\tilde{\mathbf{G}}^{j}_b\rbrace$ and the empirical measure at time $t^{(n)}_b$, i.e. $\hat{\mu}^N(\tilde{\boldsymbol\sigma}_b,\tilde{\mathbf{G}}_b)$. This property allows us to obtain a closed expression for the dynamics of the empirical process. We also determine some bounds on the conditional variance matrix.

We write
\begin{align*}
\tilde{G}^{i,j}_b = N^{-\frac{1}{2}}\sum_{k=1}^N J^{jk}\tilde{\sigma}^{i,k}_b \; \; , \; \;
\tilde{F}^{i,j}_b = N^{-\frac{1}{2}}\sum_{k=1}^N J^{jk}\big(\tilde{\sigma}^{i,k}_{b+1} - \tilde{\sigma}^{i,k}_b \big).
\end{align*}
Let $\tilde{\gamma}_{\tilde{\boldsymbol\sigma}_b,\tilde{\boldsymbol\sigma}_{b+1}}\in \mathcal{M}^+_1\big(\mathbb{R}^{2MN}\big)$ be the law of $\lbrace \tilde{\mathbf{G}}_b,\mathbf{F}_b\rbrace$ under $\gamma$ (for fixed $\lbrace \tilde{\boldsymbol\sigma}_b,\tilde{\boldsymbol\sigma}_{b+1}\rbrace$). Since the above definitions are linear, standard theory dictates that $\tilde{\gamma}_{\tilde{\boldsymbol\sigma}_b,\tilde{\boldsymbol\sigma}_{b+1}}$ is Gaussian.
Next define $\gamma_{\tilde{\boldsymbol\sigma}_b,\tilde{\boldsymbol\sigma}_{b+1},\tilde{\mathbf{G}}_b}^N \in \mathcal{M}^+_1\big(\mathbb{R}^{MN}\big)$ to be the law of $\mathbf{F}_b$ under $\tilde{\gamma}_{\tilde{\boldsymbol\sigma}_b,\tilde{\boldsymbol\sigma}_{b+1}}$, conditionally on $\tilde{\mathbf{G}}_b$. The rest of this section is devoted to finding tractable expressions for the mean and variance of $\gamma_{\tilde{\boldsymbol\sigma}_b,\tilde{\boldsymbol\sigma}_{b+1},\tilde{\mathbf{G}}_b}^N$. We define the density of $\gamma_{\tilde{\boldsymbol\sigma}_b,\tilde{\boldsymbol\sigma}_{b+1},\tilde{\mathbf{G}}_b}^N$ to be $\grave{\Upsilon}_{\tilde{\boldsymbol\sigma}_b,\tilde{\boldsymbol\sigma}_{b+1},\tilde{\mathbf{G}}_b} \in \mathcal{C}\big(\mathbb{R}^{MN}\big)$. 

Let $\Upsilon^N_{\tilde{\boldsymbol \sigma}_b, \tilde{\boldsymbol \sigma}_{b+1}} \in \mathcal{C}(\mathbb{R}^{2MN}) $ be the Gaussian density of $\lbrace \tilde{G}^{i,j}_b, \tilde{F}^{i,j}_{b+1} \rbrace_{j=1}^N$ under $\gamma$,  i.e.
\begin{equation}
\Upsilon^N_{\tilde{\boldsymbol\sigma}_b, \tilde{\boldsymbol \sigma}_{b+1}}(\tilde{\mathbf{G}}_b, \tilde{\mathbf{F}}_b) = (2\pi)^{-N}\big(\det(\mathcal{L}_N)\big)^{-\frac{1}{2}}\exp\big( -(\tilde{\mathbf{G}}_b,\tilde{\mathbf{F}}_b)^T \bar{\mathcal{K}}_N(\tilde{\boldsymbol \sigma}_b,\tilde{\boldsymbol \sigma}_{b+1})^{-1}(\tilde{\mathbf{G}}_b,\tilde{\mathbf{F}}_{b}) / 2\big),
\end{equation}
and $\bar{\mathcal{K}}_N\big(\tilde{\boldsymbol \sigma}_b,\tilde{\boldsymbol \sigma}_{b+1}\big)$ is the $2NM \times 2NM$ covariance matrix of $\big\lbrace \tilde{G}^{i,j}_b, \tilde{F}^{i,j}_b \big\rbrace_{i\in I_N,j \in I_N}$, i.e. 
\begin{equation}
\bar{\mathcal{K}}_N = \left(\begin{array}{c c}
\mathcal{K}_N(\tilde{\boldsymbol\sigma}_b) & \grave{\mathcal{K}}_N(\tilde{\boldsymbol\sigma}_b, \tilde{\boldsymbol\sigma}_{b+1} ) \\
(\grave{\mathcal{K}}_N)(\tilde{\boldsymbol\sigma}_b, \tilde{\boldsymbol\sigma}_{b+1} )^T & \tilde{\mathcal{K}}_N(\tilde{\boldsymbol\sigma}_{b+1})
\end{array}\right).
\end{equation}
The contents of $\bar{\mathcal{K}}_N\big(\tilde{\boldsymbol\sigma}_b,\tilde{\boldsymbol \sigma}_{b+1}\big)$ are the following $MN \times MN$ square matrices, with the replica indices at the top, and the spin indices at the bottom, i.e. for $i,m \in I_M$ and $j,k \in I_N$,
\begin{align}
\mathcal{K}_N(\tilde{\boldsymbol\sigma}_b)^{im}_{jk} &=  \Exp^{\gamma}\big[ \tilde{G}^{i,j}_b \tilde{G}^{m,k}_b\big]=  \delta(j,k)N^{-1}\sum_{l=1}^N \tilde{\sigma}^{i,l}_b \tilde{\sigma}^{m,l}_b+\frac{\mathfrak{s}}{N} \tilde{\sigma}^{m,j}_b\tilde{\sigma}^{i,k}_b \label{eq: covariance 1} \\
\tilde{\mathcal{K}}_N(\tilde{\boldsymbol\sigma}_b,\tilde{\boldsymbol\sigma}_{b+1} )^{im}_{jk}&=  \Exp^{\gamma}\big[ \tilde{F}^{i,j}_{b}\tilde{F}^{m,k}_b\big]=  \delta(j,k)N^{-1}\sum_{l=1}^N \big(\tilde{\sigma}^{i,l}_{b+1} - \tilde{\sigma}^{i,l}_b\big)\big(\tilde{\sigma}^{m,l}_{b+1} - \tilde{\sigma}^{m,l}_b\big)  +\frac{\mathfrak{s}}{N} \big(\tilde{\sigma}^{m,j}_{b+1} - \tilde{\sigma}^{m,j}_b\big)\big( \tilde{\sigma}^{i,k}_{b+1} - \tilde{\sigma}^{i,k}_b\big)\label{eq: covariance 2}\\
\grave{\mathcal{K}}_N(\tilde{\boldsymbol\sigma}_b, \tilde{\boldsymbol\sigma}_{b+1})^{im}_{jk}&= \Exp^{\gamma}\big[ \tilde{F}^{i,j}_b \tilde{G}^{m,k}_b \big] = \delta(j,k)N^{-1}\sum_{l=1}^N \big(\tilde{\sigma}^{i,l}_{b+1} - \tilde{\sigma}^{i,l}_b\big)\tilde{\sigma}^{m,l}_b + \frac{\mathfrak{s}}{N} \big(\tilde{\sigma}^{i,k}_{b+1} - \tilde{\sigma}^{i,k}_b\big)\tilde{\sigma}^{m,j}_b.\label{eq: covariance 3}
\end{align}
Standard theory (see for instance Theorem A.1.3 in \cite{Lindgren2013}) dictates that the density $\grave{\Upsilon}^N_{\tilde{\boldsymbol\sigma}_b \tilde{\boldsymbol \sigma}_{b+1},\tilde{\mathbf{G}}_b}(\tilde{\mathbf{F}}_b) $ of $\gamma_{\tilde{\boldsymbol\sigma}_b,\tilde{\boldsymbol\sigma}_{b+1},\tilde{\mathbf{G}}_b}$ assumes the form
\begin{multline}\label{eq: conditional gaussian density upsilon}
\grave{\Upsilon}^N_{\tilde{\boldsymbol\sigma}_b \tilde{\boldsymbol \sigma}_{b+1},\tilde{\mathbf{G}}_b}(\tilde{\mathbf{F}}_b) = (2\pi)^{-N/2} \det\big(\mathcal{R}^N(\tilde{\boldsymbol\sigma}_{b},\tilde{\boldsymbol\sigma}_{b+1}) \big)^{-\frac{1}{2}}\\ \exp\bigg( -\frac{1}{2}\big\lbrace \tilde{\mathbf{F}}_b-\tilde{\mathbf{m}}_b(\tilde{\boldsymbol\sigma}_{b},\tilde{\boldsymbol\sigma}_{b+1},\tilde{\mathbf{G}}_b) \big\rbrace^T \mathcal{R}^N\big(\tilde{\boldsymbol\sigma}_{b},\tilde{\boldsymbol\sigma}_{b+1}\big) ^{-1}\big\lbrace \tilde{\mathbf{F}}_b-\tilde{\mathbf{m}}_b(\tilde{\boldsymbol\sigma}_{b},\tilde{\boldsymbol\sigma}_{b+1},\tilde{\mathbf{G}}_b) \big\rbrace \bigg).
\end{multline}
Here $\tilde{\mathbf{m}}_b(\tilde{\boldsymbol\sigma}_{b},\tilde{\boldsymbol\sigma}_{b+1},\tilde{\mathbf{G}}_b) := \lbrace \tilde{m}_b^{i,j}(\tilde{\boldsymbol\sigma}_{b},\tilde{\boldsymbol\sigma}_{b+1},\tilde{\mathbf{G}}_b) \rbrace_{(i,j)\in I_{M,N}}$ is the vector of conditional means of $\lbrace F^{i,j}_{b+1} \rbrace$  i.e.
\begin{equation}\label{eq: mean definition}
\tilde{m}^{i,j}_b(\tilde{\boldsymbol\sigma}_{b},\tilde{\boldsymbol\sigma}_{b+1},\tilde{\mathbf{G}}_b) = \big(\grave{\mathcal{K}}_N(\tilde{\boldsymbol\sigma}_{b},\tilde{\boldsymbol\sigma}_{b+1})\mathcal{K}_N(\tilde{\boldsymbol\sigma}_{b})^{-1}\tilde{\mathbf{G}}_b\big)^{i,j},
\end{equation}
i.e. in the above $\tilde{m}^{i,j}_{b}$ is the element with index $(i,j)$ in the above vector resulting from two matrix multiplications on the vector $\tilde{\mathbf{G}}_b$.
 $\mathcal{R}_N\big(\tilde{\boldsymbol\sigma}_{b},\tilde{\boldsymbol\sigma}_{b+1}\big)$ is the $MN\times MN$ conditional covariance matrix of $\tilde{\mathbf{F}}_b$, i.e. 
 \begin{align}
\mathcal{R}_N\big(\tilde{\boldsymbol\sigma}_{b},\tilde{\boldsymbol\sigma}_{b+1}\big) &= \tilde{\mathcal{K}}_N(\tilde{\boldsymbol\sigma}_{b},\tilde{\boldsymbol\sigma}_{b+1}) - \mathcal{L}_N(\tilde{\boldsymbol\sigma}_{b},\tilde{\boldsymbol\sigma}_{b+1})\text{ where }\label{eq: LN definition 0} \\ 
\mathcal{L}_N\big(\tilde{\boldsymbol\sigma}_{b},\tilde{\boldsymbol\sigma}_{b+1}\big)  &= \grave{\mathcal{K}}_N\big(\tilde{\boldsymbol\sigma}_{b},\tilde{\boldsymbol\sigma}_{b+1}\big)  \mathcal{K}_N(\tilde{\boldsymbol\sigma}_{b})^{-1}\grave{\mathcal{K}}_N\big(\tilde{\boldsymbol\sigma}_{b},\tilde{\boldsymbol\sigma}_{b+1}\big) ^T,\label{eq: LN definition}
\end{align}
noting that $\mathcal{L}_N$ is an $MN\times MN$ matrix. 

\begin{lemma}\label{Lemma Bound Matrix Norms}
Recall that $\norm{\cdot }$ is the operator norm and the definition of $\tilde{\mathbf{L}}_b$ in \eqref{eq: tilde K definition 0 0}. We have the following bounds on $NM \times NM$ square matrices
\begin{align}
\norm{\mathcal{R}_N(\tilde{\boldsymbol\sigma}_{b},\tilde{\boldsymbol\sigma}_{b+1})} &\leq \norm{\tilde{\mathcal{K}}(\tilde{\boldsymbol\sigma}_{b},\tilde{\boldsymbol\sigma}_{b+1})}\label{eq: K covariance matrix 1}\\
\norm{\tilde{\mathcal{K}}(\tilde{\boldsymbol\sigma}_{b},\tilde{\boldsymbol\sigma}_{b+1})} &\leq N^{-1}\big\lbrace 4M + 4\mathfrak{s}M\big\rbrace \sup_{i\in I_M}\sum_{l=1}^N \chi\big\lbrace \tilde{\sigma}^{i,l}_{b+1} \neq \tilde{\sigma}^{i,l}_b \big\rbrace \label{eq: K covariance matrix 2}\\
\norm{\grave{\mathcal{K}}(\tilde{\boldsymbol\sigma}_{b},\tilde{\boldsymbol\sigma}_{b+1})} &\leq \norm{\tilde{L}_b} + \sqrt{2}N^{-1/2}\label{eq: K covariance matrix 3}\\
\norm{\mathcal{K}(\tilde{\boldsymbol\sigma}_{b})} &\leq M\big\lbrace 1 + \mathfrak{s}\big\rbrace \label{eq: K covariance matrix 4}
\end{align}
\end{lemma}
\begin{proof}
\eqref{eq: K covariance matrix 1} is a known property of finite Gaussian systems: the conditional variance is always less than or equal to the variance. It follows from the fact that $\mathcal{R}_N$, $\tilde{\mathcal{K}}_N$ and $\mathcal{L}_N$ are positive definite.

For \eqref{eq: K covariance matrix 2}, write $U^i_b = N^{-1}\sum_{l=1}^N \chi\big\lbrace \tilde{\sigma}^{i,l}_{b+1} \neq \tilde{\sigma}^{i,l}_b \big\rbrace$ and $\mathfrak{a} = (\mathfrak{a}^{i,j})_{i\in I_M, j\in I_N}$. Observe that
\begin{align*}
\sum_{i,m \in I_M}&\sum_{j,k \in I_N}\tilde{\mathcal{K}}_N(\tilde{\boldsymbol\sigma}_b,\tilde{\boldsymbol\sigma}_{b+1} )^{im}_{jk}\mathfrak{a}^{i,j}\mathfrak{a}^{m,k} \\ &=  N^{-1}\sum_{l,j \in I_N}\sum_{i,m \in I_M}\mathfrak{a}^{i,j}\mathfrak{a}^{m,j} \big(\tilde{\sigma}^{i,l}_{b+1} - \tilde{\sigma}^{i,l}_b\big)\big(\tilde{\sigma}^{m,l}_{b+1} - \tilde{\sigma}^{m,l}_b\big)  +\frac{\mathfrak{s}}{N} \sum_{i\in I_M , j,k\in I_N}\mathfrak{a}^{i,j}\mathfrak{a}^{m,k}\big(\tilde{\sigma}^{m,j}_{b+1} - \tilde{\sigma}^{m,j}_b\big)\big( \tilde{\sigma}^{i,k}_{b+1} - \tilde{\sigma}^{i,k}_b\big)\\
&\leq  4  \sum_{i,m \in I_M}\big\lbrace \sum_{j\in I_N}\big|\mathfrak{a}^{i,j}\big|^2 \sum_{k\in I_N}\big|\mathfrak{a}^{m,k}\big|^2 U^i_b U^m_b\big\rbrace^{\frac{1}{2}} + 4\mathfrak{s} \sum_{i,m \in I_M}\big\lbrace U^i_b U^m_b \sum_{j\in I_N}\big(\mathfrak{a}^{i,j}\big)^2\sum_{k\in I_N}\big(\mathfrak{a}^{m,k}\big)^2 \big\rbrace^{\frac{1}{2}}
\end{align*}
using the Cauchy-Schwarz Inequality, and the fact that since $\big|\sigma^{i,l}_{b+1} - \sigma^{i,l}_b \big| \leq 2$,
\[
N^{-1}\sum_{l=1}^N \big(\sigma^{i,l}_{b+1} - \sigma^{i,l}_b\big)^2 \leq 4 U^{i}_b.
\]
 Now 
\begin{align*}
\sum_{i,m \in I_M}\big\lbrace \sum_{j\in I_N}\big(\mathfrak{a}^{i,j}\big)^2\sum_{k\in I_N}\big(\mathfrak{a}^{m,k}\big)^2 \big\rbrace^{\frac{1}{2}}& =
\big(\sum_{i \in I_M}\big\lbrace \sum_{j\in I_N}\big(\mathfrak{a}^{i,j}\big)^2\big\rbrace^{\frac{1}{2}}\big)^2 
\leq M\sum_{i \in I_M} \sum_{j\in I_N}\big(\mathfrak{a}^{i,j}\big)^2,
\end{align*}
by the (discrete) Jensen's Inequality. We thus find that
\begin{equation*}
\sum_{i,m \in I_M}\sum_{j,k \in I_N}\tilde{\mathcal{K}}_N(\tilde{\boldsymbol\sigma}_b,\tilde{\boldsymbol\sigma}_{b+1} )^{im}_{jk}\mathfrak{a}^{i,j}\mathfrak{a}^{m,k}
\leq 4M(1+\mathfrak{s})\sup_{i\in I_M}U^{i}_b\sum_{p \in I_M} \sum_{j\in I_N}\big(\mathfrak{a}^{p,j}\big)^2.
\end{equation*}
which implies (ii). The proofs of \eqref{eq: K covariance matrix 3} and \eqref{eq: K covariance matrix 4} are analogous to the proof of (ii) and are neglected.
\end{proof}
Recall that the  $M\times M$ matrices $\lbrace \tilde{\mathbf{K}}_b,\tilde{\mathbf{L}}_b,\tilde{\boldsymbol\kappa}_b,\tilde{\boldsymbol\upsilon}_b \rbrace$ were defined to have the following elements
\begin{align}
\tilde{K}^{ij}_b &= N^{-1}\sum_{l=1}^N \tilde{\sigma}^{i,l}_{b}\tilde{\sigma}^{j,l}_b \; \; , \; \; 
\tilde{L}^{ij}_b =  N^{-1}\sum_{k=1}^{N}\tilde{\sigma}^{j,k}_b\big(\tilde{\sigma}^{i,k}_{b}- \tilde{\sigma}^{i,k}_{b+1}\big)\label{eq: tilde K definition 2}  \\
\tilde{\kappa}^{ij}_b &= N^{-1}\sum_{k=1}^{N} \tilde{G}^{j,k}_b \big(\tilde{\sigma}^{i,k}_b-\tilde{\sigma}^{i,k}_{b+1}\big) \; \; , \; \;
\tilde{\upsilon}_b^{ij} = N^{-1}\sum_{k=1}^{N}\tilde{\sigma}^{i,k}_b \tilde{G}^{j,k}_b.\label{eq: tilde upsilon definition 2}
\end{align}
The following lemma is useful because it implies that the covariance matrix of the fields $\lbrace \tilde{G}^{p,j}_b \rbrace_{p\in I_M, j \in I_N}$ is nonsingular whenever $\tilde{\mathbf{K}}_b$ is nonsingular.
\begin{lemma}\label{Lemma NonSingular KN}
$\norm{\mathcal{K}_N(\tilde{\boldsymbol\sigma}_b)} \geq \norm{\tilde{\mathbf{K}}_b}$ and $\norm{\mathcal{K}_N(\tilde{\boldsymbol\sigma}_b)^{-1}} \leq \norm{\tilde{\mathbf{K}}^{-1}_b}$ 
\end{lemma}
\begin{proof}
Making use of double index notation, define the $M^2 \times NM$ matrix $Q$ to have elements, for $u,v,q \in I_M$ and $k\in I_N$,
\begin{equation}
Q_{uv, qk} =N^{-1/2} x^{v,k}\delta(q,u).
\end{equation}
Write $\mathcal{U}_N$ to be the $NM\times NM$ block diagonal matrix, with each $M\times M$ diagonal block equal to $\tilde{\mathbf{K}}_b$. We then find that
\begin{equation}
\mathcal{K}_N(\tilde{\boldsymbol\sigma}_b) = \mathcal{U}_N + \mathfrak{s}Q^T Q.
\end{equation}
It is thus clear that $\norm{\mathcal{K}_N(\tilde{\boldsymbol\sigma}_b)} \geq \norm{\mathcal{U}_N} = \norm{\tilde{\mathbf{K}}_b}$. The second identity in the lemma follows from the fact that the smallest eigenvalue of $\norm{\mathcal{K}_N(\tilde{\boldsymbol\sigma}_b)}$ must be greater than or equal to the smallest eigenvalue of $\tilde{\mathbf{K}}_b$.
\end{proof}

We now determine a precise expression for the conditional mean. It is fundamental to the entire paper that $\tilde{\mathbf{m}}^j $ can be written as a function purely of (i) `local variables' (i.e. $\tilde{\mathbf{G}}^j_b $, $\tilde{\boldsymbol\sigma}^j_b$ and $\tilde{\boldsymbol\sigma}^j_{b+1}$, and (ii) the empirical measure (i.e. via the definitions in \eqref{eq: tilde K definition 2}  - \eqref{eq: tilde upsilon definition 2}). 
\begin{lemma}\label{Lemma Conditional Mean}
Assume that $\tilde{\boldsymbol\sigma}_b \in \mathcal{X}^N$.
(i) $\tilde{\mathbf{K}}_b$ is invertible, and we write $\tilde{\mathbf{H}}_b = \tilde{\mathbf{K}}_b^{-1}$. $\norm{\tilde{\mathbf{H}}_b} \leq \mathfrak{c}^{-1}$.

(ii) Writing $\tilde{\boldsymbol\sigma}^j_b = \big(\sigma^{1,j}_b,\ldots,\sigma^{M,j}_b\big)^T$, $\tilde{\mathbf{G}}^j_b = \big(\tilde{G}^{1,j}_b,\ldots,\tilde{G}^{M,j}_b \big)^T$ and $\tilde{\mathbf{m}}^j_b = \big(\tilde{m}^{1,j}_b,\ldots,\tilde{m}_b^{M,j}\big)^T$, we have that
	\begin{align}
	\tilde{\mathbf{m}}^j_b =& - \tilde{\mathbf{L}}_b\tilde{\mathbf{H}}_b\tilde{\mathbf{G}}^j_b -\mathfrak{s} \tilde{\boldsymbol\kappa}_b\tilde{\mathbf{H}}_b\tilde{\boldsymbol\sigma}^j_b + \mathfrak{s}\tilde{\mathbf{L}}_b\tilde{\mathbf{H}}_b\tilde{\boldsymbol\upsilon}_b\tilde{\mathbf{H}}_b\tilde{\boldsymbol\sigma}^j_b .\label{eq: temp H 0 0}
	\end{align}
\end{lemma}
\begin{proof}
The fact that $\tilde{\boldsymbol\sigma}_b \in \mathcal{X}^N$ implies that the $M\times M$ square matrix $\tilde{\mathbf{K}}_b$ (with elements defined in \eqref{eq: tilde K definition 2}) has eigenvalues greater than $\mathfrak{c}$. Since it is co-diagonal with its inverse, it must be that $\norm{\tilde{\mathbf{H}}_b} \leq \mathfrak{c}^{-1}$. It then follows from Lemma \ref{Lemma NonSingular KN} that $\mathcal{K}_N(\tilde{\boldsymbol\sigma}_b)$ must also be invertible. Let $\mathbf{V} = \mathcal{K}_N(\tilde{\boldsymbol\sigma}_b)^{-1}\tilde{\mathbf{G}}_b$. Writing $\mathbf{V} = \big\lbrace V^{i,j} \big\rbrace_{i\in I_M, j \in I_N}$, it must be that
	\begin{equation}\label{eq: hja}
	\tilde{G}^{i,j}_b = \sum_{k\in I_N,m \in I_M}\Exp^{\gamma}\big[\tilde{G}^{i,j}_b \tilde{G}^{m,k}_b\big]V^{m,k}
	\end{equation}
	Substituting the identity in \eqref{eq: covariance 1} we find that
	\begin{equation}\label{eq: h}
	\tilde{G}^{i,j}_b =\sum_{m \in I_M} \big\lbrace \tilde{K}_b^{im} V^{m,j} + \mathfrak{s}N^{-1} \sum_{k\in I_N, m\in I_M} \tilde{\sigma}^{m,j}_b \tilde{\sigma}^{i,k}_{b}V^{m,k}\big\rbrace.
	\end{equation}
	Rearranging \eqref{eq: h}, we find that
	\begin{align}
	V^{i,j} &= \sum_{m \in I_M}\tilde{H}_b^{im}\big\lbrace G^{m,j}_b - \frac{\mathfrak{s}}{N}\sum_{k\in I_N, p\in I_M} \tilde{\sigma}^{p,j}_b \tilde{\sigma}^{m,k}_bV^{p,k} \big\rbrace\nonumber \\
	&= \sum_{m \in I_M}\tilde{H}_b^{im}\big\lbrace G^{m,j}_b -\mathfrak{s}\sum_{ p\in I_M} Q^{mp} \tilde{\sigma}^{p,j}_b \big\rbrace
		\end{align}
	where 
	\[
	Q^{mp} = N^{-1}\sum_{k\in I_N}\tilde{\sigma}^{m,k}_b V^{p,k}.
	\]
	In matrix / vector notation, this means that $\mathbf{V}^j =\tilde{\mathbf{H}}\tilde{\mathbf{G}}^j_b - \mathfrak{s}\tilde{\mathbf{H}}\mathbf{Q}\tilde{\boldsymbol\sigma}^j_b$.
	Now using the identities in \eqref{eq: covariance 3} and \eqref{eq: mean definition},
	\begin{align}
	\tilde{\mathbf{m}}^j_b =& - \tilde{\mathbf{L}}_b\big(\tilde{\mathbf{H}}_b\tilde{\mathbf{G}}^j_b - \mathfrak{s} \tilde{\mathbf{H}}_b\mathbf{Q}\tilde{\boldsymbol\sigma}^j_b\big) -\mathfrak{s} \tilde{\boldsymbol\kappa}_b\tilde{\mathbf{H}}\tilde{\boldsymbol\sigma}^j_b + \mathfrak{s}^2 \tilde{\mathbf{L}}_b\mathbf{Q}^T \tilde{\mathbf{H}}_b\tilde{\boldsymbol\sigma}^j_b \nonumber \\
	=& - \tilde{\mathbf{L}}_b\tilde{\mathbf{H}}_b\tilde{\mathbf{G}}^j_b -\mathfrak{s} \tilde{\boldsymbol\kappa}_b\tilde{\mathbf{H}}_b\tilde{\boldsymbol\sigma}^j_b +\mathfrak{s} \tilde{\mathbf{L}}_b\big( \tilde{\mathbf{H}}_b\mathbf{Q}+ \mathfrak{s}\mathbf{Q}^T \tilde{\mathbf{H}}_b\big)\tilde{\boldsymbol\sigma}^j_b\label{eq: temp H}
	\end{align}
	We add $\tilde{\sigma}_b^{p,j}$ to both sides of \eqref{eq: h}, and sum over $j$, obtaining that
	\begin{equation}
	\tilde{\boldsymbol\nu} = \mathbf{Q} \tilde{\mathbf{K}}_b + \mathfrak{s} \tilde{\mathbf{K}}_b\mathbf{Q}^T.
	\end{equation}
	Multiplying both sides of the above equation by $\tilde{\mathbf{H}}_b$, we find that
	\begin{equation}
	\tilde{\mathbf{H}}_b\tilde{\boldsymbol\nu}_b\tilde{\mathbf{H}}_b = \tilde{\mathbf{H}}_b\tilde{\mathbf{Q}} + \mathfrak{s}\tilde{\mathbf{Q}}^T \tilde{\mathbf{H}}_b.
	\end{equation}
	Substituting this into \eqref{eq: temp H}, we find that, as required,
	\begin{align}
	\tilde{\mathbf{m}}^j_b =& - \tilde{\mathbf{L}}_b\tilde{\mathbf{H}}_b\tilde{\mathbf{G}}^j_b -\mathfrak{s} \tilde{\boldsymbol\kappa}_b\tilde{\mathbf{H}}_b\tilde{\boldsymbol\sigma}^j_b + \mathfrak{s}\tilde{\mathbf{L}}_b\tilde{\mathbf{H}}_b\tilde{\boldsymbol\nu}_b\tilde{\mathbf{H}}_b\tilde{\boldsymbol\sigma}^j_b .\label{eq: temp H 2}
	\end{align}	
	\end{proof}

\subsection{Bounding $\beta^6,\beta^8,\beta^9$}\label{Section bounding beta 6 8 9}
These terms are defined in \eqref{eq: beta 6},  \eqref{eq: beta 8} and  \eqref{eq: beta 9}. $\beta^6$ and $\beta^8$ concern the linear increments in the fields $\tilde{G}^{q,j}_{b+1} - \tilde{G}^{q,j}_n$, and $\beta^9$ concerns the quadratic increments in the fields.
\begin{lemma}\label{Lemma Mean Gaussian Increments}
For any $\bar{\epsilon} > 0$, for all large enough $n$,
\begin{align}
\sup_{\boldsymbol\alpha \in \mathcal{E}^M}\sup_{0\leq b < n}\lsup{N}\sup_{1\leq \mathfrak{q} \leq C^N_{\mathfrak{n}}}N^{-1}\log \mathbb{P}\big(\mathcal{J}_N , \tilde{\mu}^N \in \mathcal{V}^N_{\mathfrak{q}},\tilde{\tau}_N > t^{(n)}_b,\big|\beta^6(\boldsymbol\alpha,\tilde{\boldsymbol\sigma},\tilde{\mathbf{G}})\big|  \geq \bar{\epsilon}\Delta \big)< 0\label{eq: cheby gamma 1 0} \\
\sup_{\boldsymbol\alpha \in \mathcal{E}^M}\sup_{0\leq b < n}\lsup{N}\sup_{1\leq \mathfrak{q} \leq C^N_{\mathfrak{n}}}N^{-1}\log  \mathbb{P}\big(\mathcal{J}_N, \tilde{\mu}^N \in \mathcal{V}^N_{\mathfrak{q}},\tilde{\tau}_N > t^{(n)}_b,\big|\beta^8(\boldsymbol\alpha,\tilde{\boldsymbol\sigma},\tilde{\mathbf{G}})\big|  \geq \bar{\epsilon}\Delta\big) <0 \label{eq: cheby gamma 3}.
\end{align}

\end{lemma}
\begin{proof}
The proofs of the above two terms are very similar, thus we only prove \eqref{eq: cheby gamma 1 0}.

Define $R^{N}_{\mathfrak{q},\tilde{\boldsymbol\sigma}_b} \in \mathcal{M}^+_1\big( \mathcal{D}([t^{(n)}_b , T],\mathcal{E}^M)^N \big)$ to be the law of the stochastic process $\tilde{\boldsymbol\sigma}_{\mathfrak{q},t}$, conditioned on its value $\tilde{\boldsymbol\sigma}_{b}$ at time $t^{(n)}_b$. (Recall the definition of this process in Section \ref{Subsection Girsanov}). As previously, we drop the subscript and write $\tilde{\boldsymbol\sigma}_{\mathfrak{q},t} := \tilde{\boldsymbol\sigma}_t$. Define $Q^N_{\tilde{\boldsymbol\sigma}_b,\tilde{\mathbf{G}}_b}$ to be the regular conditional probability distribution of $(\mathbf{J}, \tilde{\boldsymbol\sigma}_b)$, conditionally on both  $\tilde{\boldsymbol\sigma}_b$ and $\tilde{\mathbf{G}}_b$. Since $\tilde{\boldsymbol\sigma}_b$ and $\mathbf{J}$ are independent, we have that
\begin{equation}
Q^N_{\tilde{\boldsymbol\sigma}_b,\tilde{\mathbf{G}}_b} = \gamma_{\tilde{\boldsymbol\sigma}_b,\tilde{\mathbf{G}}_b} \otimes R^N_{\mathfrak{q},\tilde{\boldsymbol\sigma}_b}.
\end{equation}
Writing $\mathcal{Z}^N = \big\lbrace \mathbf{g} \in \mathbb{R}^{MN} \; : \sup_{i\in I_M}\sum_{j\in I_N}|g^{i,j}|^2 \leq 3N \big\rbrace$, this means that
\begin{equation}
\mathbb{P}\big(\mathcal{J}_N \text{ and }\big|\beta^6(\boldsymbol\alpha,\tilde{\boldsymbol\sigma},\tilde{\mathbf{G}})\big|  \geq \bar{\epsilon}\Delta \big) \leq  \sup_{\tilde{\boldsymbol\sigma}_b \in \mathcal{E}^{MN},\tilde{\mathbf{G}}_b \in \mathcal{Z}^N} Q^N_{\tilde{\boldsymbol\sigma}_b,\tilde{\mathbf{G}}_b} \big(\big|\beta^7(\boldsymbol\alpha,\tilde{\boldsymbol\sigma},\tilde{\mathbf{G}})\big|  \geq \bar{\epsilon}\Delta  \big).
\end{equation}
It thus suffices to prove that
\begin{equation}
\lsup{N}\sup_{1\leq \mathfrak{q} \leq C^N_{\mathfrak{n}}}N^{-1}\log \sup_{\tilde{\boldsymbol\sigma}_b \in \mathcal{E}^{MN},\tilde{\mathbf{G}}_b \in \mathcal{Z}^N} Q^N_{\tilde{\boldsymbol\sigma}_b,\tilde{\mathbf{G}}_b} \big(\big|\beta^7(\boldsymbol\alpha,\tilde{\boldsymbol\sigma},\tilde{\mathbf{G}})\big|  \geq \bar{\epsilon}\Delta \big)< 0\label{eq: cheby gamma 1 0 0}
\end{equation}
For a constant $\mathfrak{r} > 0$, by Chernoff's Inequality,
 \begin{align}
Q^N_{\tilde{\boldsymbol\sigma}_b,\tilde{\mathbf{G}}_b} &\big(|\beta^6(\boldsymbol\alpha,\tilde{\boldsymbol\sigma},\tilde{\mathbf{G}})\big| \geq \bar{\epsilon}\Delta  \big)
=Q^N_{\tilde{\boldsymbol\sigma}_b,\tilde{\mathbf{G}}_b} \big(\beta^6(\boldsymbol\alpha,\tilde{\boldsymbol\sigma},\tilde{\mathbf{G}})  \geq \bar{\epsilon}\Delta\big)+Q^N_{\tilde{\boldsymbol\sigma}_b,\tilde{\mathbf{G}}_b} \big(\beta^6(\boldsymbol\alpha,\tilde{\boldsymbol\sigma},\tilde{\mathbf{G}})   \leq -\bar{\epsilon}\Delta  \big)\nonumber \\
&\leq \mathbb{E}^{Q^N_{\tilde{\boldsymbol\sigma}_b,\tilde{\mathbf{G}}_b} }\big[ \exp\big(\mathfrak{r}N \beta^6(\boldsymbol\alpha,\tilde{\boldsymbol\sigma},\tilde{\mathbf{G}})  - N\mathfrak{r}\bar{\epsilon}\Delta\big) + \exp\big(-\mathfrak{r}N \beta^7(\boldsymbol\alpha,\tilde{\boldsymbol\sigma},\tilde{\mathbf{G}}) -N \mathfrak{r}\bar{\epsilon}\Delta\big) \big]\nonumber \\
&= \mathbb{E}^{R^N_{i,\tilde{\boldsymbol\sigma}_b} }\big[\mathbb{E}^{\gamma_{\tilde{\boldsymbol\sigma}_b,\tilde{\mathbf{G}}_b}}\big[ \exp\big(\mathfrak{r} N\beta^6(\boldsymbol\alpha,\tilde{\boldsymbol\sigma},\tilde{\mathbf{G}})- N\mathfrak{r}\bar{\epsilon}\Delta\big)+ \exp\big(-\mathfrak{r} N\beta^6(\boldsymbol\alpha,\tilde{\boldsymbol\sigma},\tilde{\mathbf{G}})- N\mathfrak{r}\bar{\epsilon}\Delta\big)  \big]\big].\label{eq: cheby 1 gaussian}
\end{align}
Under $Q^N_{\tilde{\boldsymbol\sigma}_b , \tilde{\mathbf{G}}_b}$, and conditionally on $\tilde{\boldsymbol\sigma}_{t}$,
\[
\beta^7(\boldsymbol\alpha,\tilde{\boldsymbol\sigma},\tilde{\mathbf{G}}) = \sum_{j\in \tilde{I}_N}\sum_{i\in I_M}\phi^a_i(\tilde{\mathbf{G}}^{j}_b)\chi\big\lbrace \tilde{\boldsymbol\sigma}^j_b=\boldsymbol\alpha \big\rbrace\big\lbrace G^{i,j}_{b+1} - G^{i,j}_b  -\tilde{m}^{i,j}_b\big\rbrace 
\]
is Gaussian and of zero mean, using the expression for the conditional mean in Lemma \ref{Lemma Conditional Mean}. The covariance can be upperbounded using (i) and (ii) in Lemma \ref{Lemma Bound Matrix Norms}, i.e.
\begin{align*}
\mathbb{E}^{\gamma_{\tilde{\boldsymbol\sigma}_b, \tilde{\mathbf{G}}_b}}\big[ \big(N\beta^6(\boldsymbol\alpha,\tilde{\boldsymbol\sigma},\tilde{\mathbf{G}})  \big)^2 \big] &\leq 4M(1+\mathfrak{s})\sum_{i \in I_M , j\in I_N}\big\lbrace \phi^a_i(\tilde{\mathbf{G}}^{j}_b)\chi\big\lbrace \tilde{\boldsymbol\sigma}^j_b=\boldsymbol\alpha \big\rbrace \big\rbrace^2 \\
&\leq 4M^2(1+\mathfrak{s})N,
\end{align*}
using the fact that $|\phi^a_i| \leq 1$. We thus find that, using the formula for the moment-generating function of a Gaussian distribution,
\begin{align}
\mathbb{E}^{Q^N_{\tilde{\boldsymbol\sigma}_b,\tilde{\mathbf{G}}_b}}\big[ \exp\big(\mathfrak{r} N\beta^6(\boldsymbol\alpha,\tilde{\boldsymbol\sigma},\tilde{\mathbf{G}}) - N\mathfrak{r}\bar{\epsilon}\Delta\big) \; | \; \tilde{\boldsymbol\sigma}_{b+1} \big] &\leq \exp\big(2M^2\mathfrak{r}^2(1+\mathfrak{s})N \big)\label{eq: cheby 2 gaussian 0}  \\
\mathbb{E}^{Q^N_{\tilde{\boldsymbol\sigma}_b,\tilde{\mathbf{G}}_b}}\big[ \exp\big(-\mathfrak{r} N\beta^6(\boldsymbol\alpha,\tilde{\boldsymbol\sigma},\tilde{\mathbf{G}}) - N\mathfrak{r}\bar{\epsilon}\Delta\big) \; | \;\tilde{\boldsymbol\sigma}_{b+1} \big] &\leq \exp\big(2M^2\mathfrak{r}^2(1+\mathfrak{s})N \big).\label{eq: cheby 2 gaussian}
\end{align}
We now choose $\mathfrak{r} = \bar{\epsilon} \Delta / \big(4M^2(1+\mathfrak{s})\big)$, which means that
\begin{equation}
- \mathfrak{r}\bar{\epsilon}\Delta + 2M^2\mathfrak{r}^2(1+\mathfrak{s})  = -\frac{\bar{\epsilon}^2 \Delta^2 }{8M^2(1+\mathfrak{s})^2}\label{eq: cheby 3 gaussian}
\end{equation}
We thus find from \eqref{eq: cheby 1 gaussian}, \eqref{eq: cheby 2 gaussian 0}, \eqref{eq: cheby 2 gaussian} and \eqref{eq: cheby 3 gaussian} that
\begin{equation}
Q^N_{\tilde{\boldsymbol\sigma}_b,\tilde{\mathbf{G}}_b}\big(\big|\beta^6(\boldsymbol\alpha,\tilde{\boldsymbol\sigma},\tilde{\mathbf{G}})\big|  \geq \bar{\epsilon}\Delta \big) \leq 2\exp\big(-N\frac{\bar{\epsilon}^2 \Delta^2 }{8M^2(1+\mathfrak{s})^2} \big).\label{Eq: first cheby bound gauss}
\end{equation}
This implies \eqref{eq: cheby gamma 1 0 0}. The proof of \eqref{eq: cheby gamma 3} is analogous to the proof of \eqref{eq: cheby gamma 1 0}.
\end{proof}
\begin{lemma}
For any $\bar{\epsilon} > 0$, for all sufficiently large $n$ 
\begin{equation}
\sup_{\boldsymbol\alpha \in \mathcal{E}^M}\sup_{0\leq b < n}\lsup{N}\sup_{1\leq \mathfrak{q} \leq C^N_{\mathfrak{n}}} N^{-1}\log\mathbb{P}\big(\tilde{\tau}_N > t^{(n)}_b, \mathcal{J}_N, \tilde{\mu}^N \in \mathcal{V}^N_{\mathfrak{q}},\big|\beta^9(\boldsymbol\alpha ,\tilde{\boldsymbol\sigma}, \tilde{\mathbf{G}}_b)\big|  \geq \bar{\epsilon}\Delta\big) <0\label{eq: cheby gamma 2} 
\end{equation}
\end{lemma}
\begin{proof}
Taking conditional expectations (analogously to the proof of Lemma \ref{Lemma Mean Gaussian Increments}), it suffices to prove that
\begin{equation}
\lsup{N}N^{-1}\log  \sup_{\tilde{\boldsymbol\sigma}_b  \in \mathcal{E}^{MN}  , \boldsymbol\alpha\in\mathcal{E}^M ; , \tilde{\mathbf{G}}_b \in \mathcal{Z}^N} Q^N_{\tilde{\boldsymbol\sigma}_b,\tilde{\mathbf{G}}_b}\big(\big|\beta^9(\boldsymbol\alpha ,\tilde{\boldsymbol\sigma}, \tilde{\mathbf{G}}_b)\big|  \geq \bar{\epsilon}\Delta\big) <0\label{eq: cheby gamma 2 2} 
\end{equation}
Thanks to Lemma \ref{Lemma Upper Bound Jump Rate}, $\lsup{N}N^{-1}\log \mathbb{P}\big(\sup_{q\in I_M}N^{-1}\sum_{l\in I_N} \chi\lbrace \tilde{\sigma}^{q,l}_{b+1} \neq \tilde{\sigma}^{q,l}_b \rbrace > (c_1 + 1)\Delta \big) < 0$. We can thus assume henceforth that
\begin{equation}\label{eq: assumed jump rate}
N^{-1}\sup_{q\in I_M}\sum_{l\in I_N} \chi\lbrace \tilde{\sigma}^{q,l}_{b+1} \neq \tilde{\sigma}^{q,l}_b \rbrace \leq (c_1 + 1)\Delta .
\end{equation}
By Chernoff's Inequality,
\begin{multline}
 Q^N_{\tilde{\boldsymbol\sigma}_b,\tilde{\mathbf{G}}_b}\big(\big|\beta^9(\boldsymbol\alpha ,\tilde{\boldsymbol\sigma}, \tilde{\mathbf{G}}_b)\big|  \geq \bar{\epsilon}\Delta \; | \; \tilde{\boldsymbol\sigma} \big) \\
 \leq \mathbb{E}^{Q^N_{\tilde{\boldsymbol\sigma}_b,\tilde{\mathbf{G}}_b}}\big[\exp\big(N\mathfrak{r}\beta^9(\boldsymbol\alpha ,\tilde{\boldsymbol\sigma}, \tilde{\mathbf{G}}_b) - N\bar{\epsilon}\Delta \mathfrak{r}\big)+\exp\big(-N\mathfrak{r}\beta^9(\boldsymbol\alpha ,\tilde{\boldsymbol\sigma}, \tilde{\mathbf{G}}_b) - N\bar{\epsilon}\Delta \mathfrak{r}\big)\; | \; \tilde{\boldsymbol\sigma} \big] .
\end{multline}
We now bound the first of the expectations on the right hand side: the bound of the other is similar. Let $\mathcal{O}$ be an $NM\times NM$ square matrix (indexed using the following double-indexed notation). The element of $\mathcal{O}$ with indices $(i,j) , (p,k)$ (for $i,p\in I_M \; , \; j,k\in I_N$) is defined to be $\mathfrak{r}\delta(j,k)\phi^a_{ip}(\tilde{\mathbf{G}}^{j}_b)\chi \lbrace \tilde{\boldsymbol\sigma}^j_b= \boldsymbol\alpha \rbrace $. Define $\bar{\mathcal{O}} = \frac{1}{2}(\mathcal{O} + \mathcal{O}^T)$. Under $\gamma_{\tilde{\boldsymbol\sigma}_b,\tilde{\mathbf{G}}_b}$, $\big\lbrace \tilde{G}^{i,j}_{b+1} - \tilde{G}^{i,j}_b - \tilde{m}^{i,j}_b\big\rbrace_{i\in I_M,j\in I_N}$ are centered Gaussian variables, with their $NM \times NM$ covariance matrix equal to $\mathcal{R}_N\big(\tilde{\boldsymbol\sigma}_b , \tilde{\boldsymbol\sigma}_{b+1}\big)$ (as defined in \eqref{eq: LN definition 0}). Gaussian arithmetic thus implies that 
\begin{align}
\mathbb{E}^{Q^N_{\tilde{\boldsymbol\sigma}_b,\tilde{\mathbf{G}}_b}}\big[ &\exp\big(\mathfrak{r} N\tilde{\beta}^9(\boldsymbol\alpha ,\tilde{\boldsymbol\sigma}, \tilde{\mathbf{G}}_b) \big) \; | \; \tilde{\boldsymbol\sigma} \big]\nonumber\\ 
 &= \exp\big(-\mathfrak{r}\sum_{j\in I_{N}}\sum_{i \in I_M}\chi\big(\tilde{\boldsymbol\sigma}^{j}_b=\boldsymbol\alpha \big)\phi^a_{ii}(\tilde{\mathbf{G}}_b^{j})\tilde{L}^{ii}_b\big) \det\big(\mathcal{R}_N\big(\tilde{\boldsymbol\sigma}_b , \tilde{\boldsymbol\sigma}_{b+1}\big)\big)^{-1/2}\det\big(\mathcal{R}_N\big(\tilde{\boldsymbol\sigma}_b , \tilde{\boldsymbol\sigma}_{b+1}\big)^{-1} - \bar{\mathcal{O}} \big)^{-1/2}\nonumber\\
&= \exp\big(-\mathfrak{r}\sum_{j\in I_{N}}\sum_{i \in I_M}\chi\big(\tilde{\boldsymbol\sigma}^{j}_b=\boldsymbol\alpha \big)\phi^a_{ii}(\tilde{\mathbf{G}}_b^{j})\tilde{L}^{ii}_b\big) \det\big(\mathbf{Id} -\mathcal{R}_N(\tilde{\boldsymbol\sigma}_{b},\tilde{\boldsymbol\sigma}_{b+1})^{1/2}  \bar{\mathcal{O}}\mathcal{R}_N(\tilde{\boldsymbol\sigma}_{b},\tilde{\boldsymbol\sigma}_{b+1})^{1/2} \big)^{-\frac{1}{2}} \nonumber\\
&= \exp\big(-\mathfrak{r}\sum_{j\in I_{N}}\sum_{i \in I_M}\chi\big(\tilde{\boldsymbol\sigma}^{j}_b=\boldsymbol\alpha \big)\phi^a_{ii}(\tilde{\mathbf{G}}_b^{j})\tilde{L}^{ii}_b\big) \prod_{j=1}^{MN} (1-\lambda_j)^{-\frac{1}{2}},\label{eq: beta 9 bound tmpry}
\end{align}
where $\lbrace \lambda_j \rbrace_{j=1}^{MN}$ are the eigenvalues of the real symmetric matrix $\mathcal{R}_N(\tilde{\boldsymbol\sigma}_{b},\tilde{\boldsymbol\sigma}_{b+1})^{1/2}  \bar{\mathcal{O}}\mathcal{R}_N(\tilde{\boldsymbol\sigma}_{b},\tilde{\boldsymbol\sigma}_{b+1})^{1/2} $ (assuming for the moment that the modulus of each of these eigenvalues is strictly less than one). We thus find that
\begin{align}
 N^{-1}\log \det\big(\mathbf{Id} -\mathcal{R}_N(\tilde{\boldsymbol\sigma}_{b},\tilde{\boldsymbol\sigma}_{b+1})^{1/2}  \bar{\mathcal{O}}\mathcal{R}_N(\tilde{\boldsymbol\sigma}_{b},\tilde{\boldsymbol\sigma}_{b+1})^{1/2} \big)^{-\frac{1}{2}} 
=&-(2N)^{-1}\sum_{u=1}^{NM} \log (1-\lambda_u)\nonumber \\
=& -(2N)^{-1}\sum_{u=1}^{NM}\big\lbrace -\lambda_u - \lambda_u^2 / Z_u^2\big\rbrace, \label{eq: Taylor Expansion log u}
\end{align}
where $Z_u \in [1-\lambda_u, 1]$ if $\lambda_u > 0$, else $Z_u \in [1,1-\lambda_u]$ if $\lambda_u < 0$, using the second-order Taylor Expansion of $\log$ about $1$. Now
 \begin{multline}
\big|\lambda_j \big|\leq \norm{\mathcal{R}_N(\tilde{\boldsymbol\sigma}_{b},\tilde{\boldsymbol\sigma}_{b+1})^{1/2}  \bar{\mathcal{O}}\mathcal{R}_N(\tilde{\boldsymbol\sigma}_{b},\tilde{\boldsymbol\sigma}_{b+1})^{1/2}} \leq \norm{\bar{\mathcal{O}}}\norm{\mathcal{R}_N(\tilde{\boldsymbol\sigma}_{b},\tilde{\boldsymbol\sigma}_{b+1})} \leq \norm{\bar{\mathcal{O}}}\bar{C}N^{-1}\sup_{q\in I_M}\sum_{l \in I_N} \chi\big\lbrace \tilde{\sigma}^{q,l}_{b+1} \neq \tilde{\sigma}^{q,l}_b \big\rbrace .\label{eq: phi j bound}
\end{multline}
(using Lemma \ref{Lemma Bound Matrix Norms}, and writing $\bar{C} = 4M(1+\mathfrak{s}))$. It may be observed from the block diagonal structure of $\bar{\mathcal{O}}$ (i.e. $\bar{\mathcal{O}}$ is `diagonal' with respect to the $I_N$ indices) that
\begin{align}
\norm{\bar{\mathcal{O}}} &= \mathfrak{r}/2 \sup_{j\in I_N}\sup_{\mathfrak{a} \in\mathbb{R}^M \; : \; \norm{\mathfrak{a}}=1}\big| \mathfrak{a}^i \mathfrak{a}^p\big( \phi^a_{ip}(\tilde{\mathbf{G}}^{j}_b)\chi \lbrace \tilde{\boldsymbol\sigma}^j_b= \boldsymbol\alpha \rbrace + \phi^a_{pi}(\tilde{\mathbf{G}}^{j}_b)\chi \lbrace \tilde{\boldsymbol\sigma}^j_b= \boldsymbol\alpha \rbrace \big) \big|\nonumber \\
&\leq \mathfrak{r}\big( \sum_{i,p\in I_M} \big|\phi^a_{ip}(\tilde{\mathbf{G}}^{j}_b)\big|^2 \big)^{1/2} \leq M \mathfrak{r},\label{eq: norm O bound}
\end{align}
since $\big| \phi^a_{ip}(\tilde{\mathbf{G}}^{j}_b)\chi \lbrace \tilde{\boldsymbol\sigma}^j_b= \boldsymbol\alpha \rbrace \big| \leq 1$, and utilizing the fact that the operator norm is upper-bounded by the Frobenius matrix norm.

We thus find that $\lambda_j \leq \frac{1}{2}$, as long as $\frac{M\mathfrak{r}\bar{C}}{N}\sup_{q\in I_M}\sum_{l=1}^N \chi\big\lbrace \tilde{\sigma}^{q,l}_{b+1} \neq \tilde{\sigma}^{q,l}_b \big\rbrace \leq \frac{1}{2}$, and this follows from our earlier assumption \eqref{eq: assumed jump rate} as long as $\Delta$ is small enough. This means that $Z_j \geq 1/2$. Since
\[
\sum_{j=1}^{MN} \lambda_j = \rm{tr}\big(\mathcal{R}_N(\tilde{\boldsymbol\sigma}_{b},\tilde{\boldsymbol\sigma}_{b+1})^{1/2}  \bar{\mathcal{O}}\mathcal{R}_N(\tilde{\boldsymbol\sigma}_{b},\tilde{\boldsymbol\sigma}_{b+1})^{1/2} \big) =  \rm{tr}\big( \bar{\mathcal{O}}\mathcal{R}_N(\tilde{\boldsymbol\sigma}_{b},\tilde{\boldsymbol\sigma}_{b+1}) \big),
\]
we find that \eqref{eq: Taylor Expansion log u} implies that 
\begin{align}
N^{-1}\log\det\big(\mathbf{I} -\mathcal{R}_N(\tilde{\boldsymbol\sigma}_{b},\tilde{\boldsymbol\sigma}_{b+1})^{1/2} & \bar{\mathcal{O}}\mathcal{R}_N(\tilde{\boldsymbol\sigma}_{b},\tilde{\boldsymbol\sigma}_{b+1})^{1/2}  \big)^{-\frac{1}{2}}\leq   (2N)^{-1}\sum_{j\in I_N} \big(\lambda_j + 4\norm{\mathcal{R}_N(\tilde{\boldsymbol\sigma}_{b},\tilde{\boldsymbol\sigma}_{b+1})^{1/2}  \bar{\mathcal{O}}\mathcal{R}_N(\tilde{\boldsymbol\sigma}_{b},\tilde{\boldsymbol\sigma}_{b+1})^{1/2}}^2\big)\nonumber\\
=&  (2N)^{-1}\rm{tr}\big(\bar{\mathcal{O}}\mathcal{R}_N(\tilde{\boldsymbol\sigma}_{b},\tilde{\boldsymbol\sigma}_{b+1})\big) +2 \norm{\mathcal{R}_N(\tilde{\boldsymbol\sigma}_{b},\tilde{\boldsymbol\sigma}_{b+1})^{1/2}  \bar{\mathcal{O}}\mathcal{R}_N(\tilde{\boldsymbol\sigma}_{b},\tilde{\boldsymbol\sigma}_{b+1})^{1/2}}^2\nonumber \\
\leq &  (2N)^{-1}\rm{tr}\big(\bar{\mathcal{O}}\mathcal{R}_N(\tilde{\boldsymbol\sigma}_{b},\tilde{\boldsymbol\sigma}_{b+1})\big) + 2\bigg(\frac{M\mathfrak{r}\bar{C}}{N}\sup_{q\in I_M}\sum_{l \in I_N} \chi\big\lbrace \tilde{\sigma}^{q,l}_{b+1} \neq \tilde{\sigma}^{q,l}_b \big\rbrace\bigg)^2 ,\label{eq: beta 9 tmpry 2}
\end{align}
using \eqref{eq: phi j bound} and \eqref{eq: norm O bound}. Now, noting the definition of $\mathcal{R}_N(\tilde{\boldsymbol\sigma}_{b},\tilde{\boldsymbol\sigma}_{b+1})$ in \eqref{eq: LN definition 0},
\[
\rm{tr}\big(\bar{\mathcal{O}}\mathcal{R}_N(\tilde{\boldsymbol\sigma}_{b},\tilde{\boldsymbol\sigma}_{b+1})\big)=\rm{tr}\big(\bar{\mathcal{O}}\tilde{\mathcal{K}}_N(\tilde{\boldsymbol\sigma}_{b},\tilde{\boldsymbol\sigma}_{b+1})\big)+\rm{tr}\big(\bar{\mathcal{O}}\mathcal{L}_N(\tilde{\boldsymbol\sigma}_{b},\tilde{\boldsymbol\sigma}_{b+1})\big).
\]
We bound $\rm{tr}\big(\bar{\mathcal{O}}\mathcal{L}_N(\tilde{\boldsymbol\sigma}_{b},\tilde{\boldsymbol\sigma}_{b+1})\big)$ using Von Neumann's Trace Inequality. Thanks to \eqref{eq: norm O bound}, the singular values of $\bar{\mathcal{O}}$ are upperbounded by $\mathfrak{r}M$. We thus find that, since each matrix is $NM \times NM$,
\[
\rm{tr}\big(\bar{\mathcal{O}}\mathcal{L}_N(\tilde{\boldsymbol\sigma}_{b},\tilde{\boldsymbol\sigma}_{b+1})\big) \leq NM\mathfrak{r} \norm{\mathcal{L}_N(\tilde{\boldsymbol\sigma}_{b},\tilde{\boldsymbol\sigma}_{b+1}))}.
\]
Furthermore it is immediate from the definition that
\begin{align*}
\norm{\mathcal{L}_N(\tilde{\boldsymbol\sigma}_{b},\tilde{\boldsymbol\sigma}_{b+1}))} \leq \norm{\mathcal{K}_N(\tilde{\boldsymbol\sigma}_b,\tilde{\mathbf{G}}_b)^{-1}}\norm{\grave{\mathcal{K}}_N(\tilde{\boldsymbol\sigma}_b,\tilde{\mathbf{G}}_b)}^2.
\end{align*}
Thanks to Lemma \ref{Lemma NonSingular KN}, $\norm{\mathcal{K}_N(\tilde{\boldsymbol\sigma}_b,\tilde{\mathbf{G}}_b)^{-1}} \leq \norm{\tilde{\mathbf{K}}_b^{-1}} \leq \mathfrak{c}^{-1}$. Also Lemma \ref{Lemma Bound Matrix Norms}  implies that
\begin{align*}
\norm{\grave{\mathcal{K}}_N(\tilde{\boldsymbol\sigma}_b,\tilde{\mathbf{G}}_b)} \leq \norm{\tilde{\mathbf{L}}_b} + \sqrt{2}N^{-1/2} 
&\leq MN^{-1}\sup_{p\in I_M}\sum_{j\in I_N}\chi\big\lbrace \tilde{\sigma}^{p,j}_{b+1} \neq  \tilde{\sigma}^{p,j}_{b} \big\rbrace + \sqrt{2}N^{-1/2} \\
&\leq M(c_1+1)\Delta + \sqrt{2}N^{-1/2}
\end{align*}
using \eqref{eq: assumed jump rate}. We thus find that
\[
\rm{tr}\big(\bar{\mathcal{O}}\mathcal{L}_N(\tilde{\boldsymbol\sigma}_{b},\tilde{\boldsymbol\sigma}_{b+1})\big) = O\big(N\Delta^2\mathfrak{r} + \sqrt{N}\Delta \mathfrak{r}\big).
\]
Now substituting the definition of $\tilde{\mathcal{K}}_N$ in \eqref{eq: covariance 2},
\begin{multline*}
\rm{tr}\big(\bar{\mathcal{O}}\tilde{\mathcal{K}}_N(\tilde{\boldsymbol\sigma}_{b},\tilde{\boldsymbol\sigma}_{b+1})\big)  = \mathfrak{r}\sum_{j\in I_N,i,p\in I_M}
\chi\lbrace \tilde{\boldsymbol\sigma}^{j}_b=\boldsymbol\alpha \rbrace\phi^a_{ip}(\tilde{\mathbf{G}}_b^{j})\big( N^{-1}\sum_{k\in I_N}(\tilde{\sigma}^{p,k}_{b+1} - \tilde{\sigma}^{p,k}_{b})(\tilde{\sigma}^{i,k}_{b+1} - \tilde{\sigma}^{i,k}_{b}) + \mathfrak{s}N^{-1}(\tilde{\sigma}^{p,j}_{b+1} - \tilde{\sigma}^{i,j}_{b})(\tilde{\sigma}^{i,j}_{b+1} - \tilde{\sigma}^{i,j}_{b} )\big). 
\end{multline*}
One can easily demonstrate using Martingale concentration inequalities (similar to those in the Appendix) that there exists a constant $\tilde{C}$ such that for all $n\in \mathbb{Z}^+$, if $p\neq i$ then
\begin{equation}
N^{-1}\log\mathbb{P}\big( N^{-1}\big|\sum_{k\in I_N}(\tilde{\sigma}^{p,k}_{b+1} - \tilde{\sigma}^{p,k}_{b})(\tilde{\sigma}^{i,k}_{b+1} - \tilde{\sigma}^{i,k}_{b})\big| \geq \tilde{C}\Delta^2 \big) < 0.\label{eq: beta 9 exponentially decaying identity 0}
\end{equation}
Using the definition of $\tilde{L}$ in \eqref{eq: tilde K definition 0 0}, and the fact that $(\tilde{\sigma}^{i,k}_{b+1} - \tilde{\sigma}^{i,k}_{b})^2 = 2\tilde{\sigma}^{i,k}_b(\tilde{\sigma}^{i,k}_{b} - \tilde{\sigma}^{i,k}_{b+1})$ (since $\tilde{\sigma}^{i,k}_u \in \lbrace -1,1 \rbrace$), we obtain that the probability that the following event does not hold is exponentially decaying in $N$,
\begin{equation}
\rm{tr}\big(\bar{\mathcal{O}}\tilde{\mathcal{K}}_N(\tilde{\boldsymbol\sigma}_{b},\tilde{\boldsymbol\sigma}_{b+1})\big) = \mathfrak{r} \sum_{j\in I_{N}}\sum_{i \in I_M}\chi\big(\tilde{\boldsymbol\sigma}^{j}_b=\boldsymbol\alpha \big)\phi^a_{ii}(\tilde{\mathbf{G}}_b^{j})2\tilde{L}^{ii}_b + O\big(N\Delta^2\mathfrak{r}  + \sqrt{N}\Delta \mathfrak{r} \big).\label{eq: beta 9 exponentially decaying identity}
\end{equation}
In summary, we obtain from \eqref{eq: assumed jump rate}, \eqref{eq: beta 9 bound tmpry}, \eqref{eq: beta 9 tmpry 2}, \eqref{eq: beta 9 exponentially decaying identity 0} and \eqref{eq: beta 9 exponentially decaying identity} that
\begin{align*}
\lsup{N} N^{-1} \log Q^N_{\tilde{\boldsymbol\sigma}_b,\tilde{\mathbf{G}}_b}\big(\beta^9(\boldsymbol\alpha ,\tilde{\boldsymbol\sigma}, \tilde{\mathbf{G}}_b)   \geq \bar{\epsilon}\Delta  \; | \; \tilde{\boldsymbol\sigma} \big) \leq -\bar{\epsilon}\Delta \mathfrak{r} + \mathfrak{r}\text{Const}\Delta^2.
\end{align*}
This clearly implies \eqref{eq: cheby gamma 2 2} (and therefore the lemma) as long as $\Delta$ and $\mathfrak{r}$ are sufficiently small.
\end{proof}

\section{Appendix: Properties of Poisson Processes}
The following lemma contains some standard results concerning Poisson counting processes \cite{Ethier1986}. The first three can be demonstrated using Chernoff's Inequality, and the last is a standard formula.
\begin{lemma}\label{Lemma Upper Bound Jump Rate}
(i) For any $t \geq t^{(n)}_b$, and any $i\in I_M, j\in I_N$,
\begin{align}
\mathbb{P}\big( \tilde{\sigma}^{i,j}_t \neq \tilde{\sigma}^{i,j}_b \big) \leq c_1 (t - t^{(n)}_b).
\end{align}
(ii) For any $\epsilon > 0$, 
\begin{equation}
\lsup{N}N^{-1}\log \mathbb{P}\big(N^{-1}\sum_{j\in I_N} \chi\big\lbrace \sum_{i\in I_M} Y^{i,j}_b( c_1  t) > 0 \big\rbrace > (M c_1+\epsilon)  t \big) < 0\label{eq: number jumps 1} 
\end{equation}
(iii) For any $\epsilon \in (0, c_1)$, 
\begin{equation}
\lsup{N}N^{-1} \log \mathbb{P}\big(N^{-1}\sum_{j\in I_N} \chi\big\lbrace \sum_{i \in I_M} Y^{i,j}_b( c_1  t) > 0 \big\rbrace < (c_1-\epsilon)  t \big) < 0\label{eq: number jumps 2} 
\end{equation}
(iv) For any $u,x> 0$,
\begin{equation}\label{eq: exp moment Poisson}
\mathbb{E}\big[\exp\big(uY^{i,j}(x t) \big) \big] = \exp\big(x t \lbrace e^u - 1 \rbrace \big).
\end{equation}
\end{lemma}
The following general lemma yields a concentration inequality for compensated Poisson Processes.
\begin{lemma}\label{Lemma Concentration Poisson}
Suppose that $\lbrace u^{q,j}_s , v^{q,j}_t \rbrace_{j\in I_N}$ are adapted c\`{a}dl\`{a}g stochastic processes, with $u^{q,j}_t \geq 0$ and that
\begin{align}
Z^{q,j}_t =& Y^{q,j}\bigg( \int_0^t u^{q,j}_s ds \bigg) \\
X^{q,j}_t =& \int_0^t v^{q,j}(s)dZ^{q,j}_s - \int_0^t v^{q,j}_s u^{q,j}_s ds.
\end{align}
Assume that $ u^{q,j}_t  \leq u_{max}$ for some constant $u_{max}$.\\
(i) Suppose that $| v^{q,j}_t | \leq v_{max}$ for some constant $v_{max}$. Then there exists $z_0$ and a constant $C$ such that for all $z \in [0,z_0]$,
\begin{equation}
\mathbb{P}\big(\sup_{t\in [0,x]}\sum_{j\in I_N,q\in I_M}X^{q,j}_t \geq Nz \big) \leq \exp\big(-NCz^2 / x^2 \big)
\end{equation}
(ii) Suppose that $N^{-1}\sup_{t\in [0,T]}\sum_{j\in I_N}\chi\lbrace v^{q,j}_t > 0 \rbrace \exp(v^{q,j}_t) \leq C$. Then for all $z > 0$,
\begin{equation}
\sup_{q\in I_M}\mathbb{P}\big(\sup_{t\in [0,x]}\sum_{j\in I_N}X^{q,j}_t \geq Nz \big) \leq \exp\big(Nu_{max}xC-Nz\big).
\end{equation}
\end{lemma}
\begin{proof}
Now since the exponential function is increasing, for a constant $y > 0$,
\begin{align}
\mathbb{P}\big(\sup_{t\in [0,x]}\sum_{j\in I_N,q\in I_M}X^{q,j}_t \geq Nz \big) &= \mathbb{P}\big(\sup_{t\in [0,x]}\exp(y\sum_{j\in I_N,q\in I_M}X^{q,j}_t - Nzy) \geq 1 \big) \\
&\leq \mathbb{E}\big[\exp\big(\exp(y\sum_{j\in I_N,q\in I_M}X^{q,j}_x - Nzy\big) \big],
\end{align}
by Doob's Submartingale Inequality, and using the fact that the compensated Poisson Process $X^{q,j}_t$ is a Martingale \cite{Anderson2015}. Choose $y$ to be such that $\exp(y v_{max}) \leq 2$. We now demonstrate that
\begin{equation}
\mathbb{E}\big[\exp\big(y\sum_{j\in I_N,q\in I_M}X^{q,j}_x \big) \big] \leq \exp\big( NM y^2 u_{max}v_{max}^2 x\big).\label{eq: to show exponential compensated}
\end{equation}
First notice that, since the functions are c\`{a}dl\`{a}g,
\[
\lim_{h\to 0}h^{-1}\big\lbrace \int_t^{t+h}y v^{q,j}_s u^{q,j}_s ds - hy v^{q,j}_t u^{q,j}_t\big\rbrace = 0.
\]
We then find that, for $t\in [0,x)$, 
\begin{align*}
\frac{d}{dt}\mathbb{E}\big[&\exp\big(y\sum_{j\in I_N,q\in I_M}X^{q,j}_t \big) \big] = \lim_{h \to 0}h^{-1}\big\lbrace \mathbb{E}\big[\exp\big(y\sum_{j\in I_N,q\in I_M}X^{q,j}_{t+h} \big) \big]-\mathbb{E}\big[\exp\big(y\sum_{j\in I_N,q\in I_M}X^{q,j}_{t} \big) \big]\big\rbrace \\
=& \lim_{h \to 0}h^{-1}\big\lbrace \mathbb{E}\big[\exp\big(y\sum_{j\in I_N,q\in I_M}X^{q,j}_{t}  + yv^{q,j}_t(Z^{q,j}_{t+h}-Z^{q,j}_t)- hy v^{q,j}_t u^{q,j}_t  \big) \big]-\mathbb{E}\big[\exp\big(y\sum_{j\in I_N,q\in I_M}X^{q,j}_{t} \big) \big]\big\rbrace \\
= & \lim_{h \to 0}h^{-1}\big\lbrace \mathbb{E}\big[\exp\big(\sum_{j\in I_N,q\in I_M}\big[ yX^{q,j}_{t} + h u^{q,j}_t\big\lbrace \exp(yv^{q,j}_t)-1 \big\rbrace  - hy v^{q,j}_t u^{q,j}_t \big]\big) \big]\\ &-\mathbb{E}\big[\exp\big(\exp(y\sum_{j\in I_N,q\in I_M}X^{q,j}_{t} \big) \big]\big\rbrace,
\end{align*}
using the expression for the Poisson moment in \eqref{eq: exp moment Poisson}. Now since $\exp(yv^{q,j}_t) \leq 2$, Taylor's Theorem implies that $\exp(yv^{q,j}_t)-1 \leq yv^{q,j}_t + (yv^{q,j}_t)^2$. On taking $h\to 0$, we thus obtain that
\begin{align*}
\frac{d}{dt}\mathbb{E}\big[\exp\big(y\sum_{j\in I_N,q\in I_M}X^{q,j}_t \big) \big] \leq \mathbb{E}\big[\exp\big(y\sum_{j\in I_N,q\in I_M}X^{q,j}_t \big) \big] y^2 NM u_{max}v_{max}^2.
\end{align*}
Gronwall's Inequality thus implies \eqref{eq: to show exponential compensated}. We now choose $y= \min\big\lbrace z / (Mxu_{max} v_{max}^2) ,( \log 2) / v_{max} \big\rbrace$ and we have obtained (i). (ii) follows analogously.
\end{proof}
\textbf{Acknowledgements:} Much thanks to Colin MacLaurin (U. Queensland) for obtaining some preliminary numerical results that were incorporated into the introduction. Much thanks also to Gerard Ben Arous (NYU), David Shirokoff (NJIT), Victor Matveev (NJIT), Bruno Cessac (INRIA) and Etienne Tanre (INRIA) for interesting discussions and very helpful feedback.

\bibliographystyle{plain}
\bibliography{SGbib2}

\end{document}